\documentclass[12pt, reqno]{amsart}
\usepackage{amssymb}
\usepackage{amsmath}
\allowdisplaybreaks[1]
\usepackage{graphicx,color}
\usepackage{wrapfig,framed}

\usepackage[height=22.5cm, width=16.5cm, hmarginratio={1:1}]{geometry}
\usepackage{hyperref}

\theoremstyle{plain}
\newtheorem{theorem}{Theorem}
\newtheorem{prop}{Proposition}
\newtheorem{lemma}{Lemma}
\newtheorem{cor}{Corollary}
\theoremstyle{definition}

\newtheorem{example}{Example}
\newtheorem{remark}{Remark}

\newcommand{\beq}{\begin{equation}}
\newcommand{\eeq}{\end{equation}}

\newcommand{\nn}{\nonumber}
\newcommand{\e}{\epsilon}
\newcommand{\p}{\partial}
\newcommand{\CC}{\mathbb{C}}
\newcommand{\F}{\mathcal{F}}
\newcommand{\Q}{\mathcal{Q}}
\newcommand{\V}{\mathcal{V}}
\newcommand{\I}{\mathcal{I}}
\newcommand{\bt}{{\bf t}}

\begin{document}

\title[Analytic theory of Legendre-type transformations]
{Analytic theory of Legendre-type transformations \\ for a Frobenius manifold}

\author{Di Yang}
\address{School of Mathematical Sciences, University of Science and Technology of China, Hefei 230026, P. R. China}
\email{diyang@ustc.edu.cn}
\date{}

\begin{abstract} 
Let $M$ be an $n$-dimensional Frobenius manifold. Fix $\kappa\in\{1,\dots,n\}$. Assuming certain invertibility, Dubrovin introduced the Legendre-type 
transformation $S_\kappa$, which transforms $M$ to an $n$-dimensional Frobenius manifold 
$S_\kappa(M)$.  
In this paper, we show that these $S_\kappa(M)$ share the same monodromy data 
at the Fuchsian singular point of the Dubrovin connection, and that for the case when $M$ is semisimple they also share the same 
 Stokes matrix and the same central connection matrix. 
A straightforward application of the monodromy identification is the following:
if we know the monodromy data of some semisimple Frobenius manifold~$M$, we immediately obtain  
those of its Legendre-type transformations. Another application gives the identification between the  
$\kappa$th partition function of a semisimple Frobenius manifold $M$ and the topological partition function of $S_{\kappa}(M)$.
\end{abstract}
\maketitle

\setcounter{tocdepth}{1}
\tableofcontents

\section{Introduction}
The notion of Frobenius manifold was introduced by B.~Dubrovin in the beginning of the 1990s to give 
a coordinate-free description of the 
Witten--Dijkgraaf--Verlinde--Verlinde (WDVV) equations \cite{DVV91-1, DVV91-2, W90} that 
appeared in two-dimensional topological field theories; see~\cite{Du92, Du96, DZ-norm, Givental, KM94}. 
Nowadays, it is also well known as {\it Dubrovin--Frobenius manifold}, and is 
 a central subject in the studies of the bihamiltonian structures of hydrodynamic type~\cite{Du96, Du98-2, DN83}, 
the Gromov--Witten theory~\cite{BF, KM94, LT, MS, RT}, 
the Darboux--Egorov geometry and isomonodromic deformations~\cite{BJL, Du92, Du96, Du99, Sabbah}, the quantum 
singularity theory~\cite{FJR13, Hertling, Saito81}, and etc. 

Recall that a {\it Frobenius algebra} is a triple $(A, \, e, \, \langle\,,\,\rangle)$, 
where $A$ is a commutative and associative algebra over~$\CC$
with unity~$e$, and $\langle\,,\,\rangle: A\times A\rightarrow \mathbb{C}$ is a symmetric,  
non-degenerate and bilinear product satisfying 
$\langle x\cdot y , z\rangle = \langle x, y\cdot z\rangle$, $\forall\,x,y,z \in A$. 
A {\it Frobenius structure of charge~$D$} \cite{Du96} on a complex manifold~$M$ is a family of 
Frobenius algebras $(T_p M, \, e_p, \, \langle\,,\,\rangle)_{p\in M}$, 
depending holomorphically on~$p$ and satisfying the following three axioms:
\begin{itemize}
\item[{\bf A1}]  The metric $\langle\,,\,\rangle$ is flat; moreover, 
denote by~$\nabla$ the Levi--Civita connection of~$\langle\,,\,\rangle$, then it is required that 
\beq\label{flatunity511}
\nabla e = 0.
\eeq
\item[{\bf A2}] Define a $3$-tensor field~$c$ by 
$c(X,Y,Z):=\langle X\cdot Y,Z\rangle$, for $X,Y,Z$ being holomorphic vector fields on~$M$. 
Then the $4$-tensor field $\nabla c$ is required to be 
symmetric. 
\item[{\bf A3}]
There exists a holomorphic vector field~$E$ on~$M$ satisfying
\begin{align}
&  [E, X\cdot Y]-[E,X]\cdot Y - X\cdot [E,Y] = X\cdot Y, \label{E2}\\
&  E \langle X,Y \rangle - \langle [E,X],Y\rangle - \langle X, [E,Y]\rangle 
= (2-D) \langle X,Y\rangle. \label{E3}
\end{align}
\end{itemize}
A complex manifold endowed with a Frobenius structure of charge~$D$ is  
called a {\it Frobenius manifold of charge~$D$}, with $\langle \, ,\, \rangle$ 
being called the {\it invariant flat metric} and $E$ the {\it Euler vector field}.
We denote the Frobenius manifold by $(M,\, \cdot,\, e,\, \langle \,,\, \rangle,\, E)$. 
In Dubrovin's original definition~\cite{Du96} it is also required that 
\beq \nabla\nabla E = 0. \label{linear}
\eeq
This condition can actually follow from~\eqref{E3}; see e.g.~\cite{CDG20}.

If we drop the axiom A3 from the above definition and keep all other axioms, 
then we obtain a structure that we call the {\it pre-Frobenius structure}. 
A complex manifold with a pre-Frobenius structure
is called a {\it pre-Frobenius manifold}. 
A non-zero element~$a$ of a Frobenius algebra is called a {\it nilpotent element} or a {\it nilpotent}, if 
there exists $m\ge0$ such that $a^m=0$. A Frobenius algebra 
is called {\it semisimple}, if it contains no nilpotents. 
A point~$p$ on a (pre-)Frobenius manifold~$M$ is called a {\it semisimple point}, if $T_pM$ is semisimple. 
Since semisimplicity is an open condition, if $p$ is a semisimple point of~$M$ then there is an open 
neighborhood~$U$ of~$p$ such that all points in~$U$ are semisimple points. 
A (pre-)Frobenius manifold is called {\it semisimple}, if its generic points are semisimple points.

Let $(M, \, \cdot, \, e, \, \langle\,,\,\rangle)$ be an $n$-dimensional pre-Frobenius manifold. 
Take $(B, \phi: B\rightarrow \CC^n)$ a flat coordinate chart of~$M$ with respect to $\langle\,,\,\rangle$. 
Here $B$ is an open set and $\phi(B)$ is a small ball (for simplicity). The notation will be used throughout the paper. 
The flat coordinates ${\bf v}=(v^1,\dots,v^n)\in \phi(B)$ can be chosen so that 
one of the coordinate vector field, say, $\frac{\p}{\p v^\iota}$,  
coincides with the unity vector field~$e$, i.e., 
$\frac{\p}{\p v^\iota}=e$.
Here $\iota\in\{1,\dots,n\}$.
The coordinates of the center of~$\phi(B)$
are denoted by $v^1_o,\dots,v^n_o$.
Let 
\beq\label{defetaandetamat}
\eta_{\alpha\beta} := \Bigl \langle \frac{\p}{\p v^\alpha}, \frac{\p}{\p v^\beta}\Bigr\rangle, ~ \alpha,\beta=1,\dots,n, \qquad \eta:=(\eta_{\alpha\beta}).
\eeq
By definition, $\eta$ is a symmetric and non-singular constant matrix. Define $\eta^{\alpha\beta}$ via $(\eta^{\alpha\beta})=\eta^{-1}$, 
and denote 
$$c_{\alpha\beta\gamma}({\bf v}) = \Bigl\langle\frac{\p}{\p v^\alpha}\cdot\frac{\p}{\p v^\beta}, \frac{\p}{\p v^\gamma} \Bigr\rangle, 
\quad \alpha,\beta,\gamma=1,\dots,n.$$
Clearly, $c_{\alpha\beta\iota}({\bf v})\equiv \eta_{\alpha\beta}$. 
Let $v_\alpha=\sum_{\rho=1}^n \eta_{\alpha\rho} v^\rho$, $\alpha=1,\dots,n$.
Since $\eta$ is non-degenerate, $v_1,\dots,v_n$ also give a system of flat coordinates of $\langle\,,\,\rangle$. 
Let 
$c^{\alpha}_{\beta\gamma}({\bf v}) = \sum_{\rho=1}^n \eta^{\alpha\rho} c_{\rho \beta\gamma}({\bf v})$. 
Then we have
\beq
\frac{\p}{\p v^\alpha}\cdot\frac{\p}{\p v^\beta}=\sum_{\gamma=1}^n c^\gamma_{\alpha\beta}({\bf v}) \frac{\p}{\p v^\gamma}, \qquad \alpha,\beta=1,\dots,n.
\eeq
By Axiom~A2 and by the symmetric property of $c_{\alpha\beta\gamma}({\bf v})$ we know the local existence of 
a holomorphic function $F^M({\bf v})$ such that 
\beq\label{dddF519}
\frac{\p^3 F^M({\bf v})}{\p v^\alpha \p v^\beta \p v^\gamma} = c_{\alpha\beta\gamma}({\bf v}), \qquad \alpha,\beta,\gamma=1,\dots,n.
\eeq
We call $F^M({\bf v})$ the {\it potential}\footnote{It is also known as a pre-potential or a Frobenius potential.} of the pre-Frobenius manifold~$M$, which is uniquely determined by 
$c_{\alpha\beta\gamma}({\bf v})$ up to adding an arbitrary quadratic function of~${\bf v}$. The associativity in terms of $F^M({\bf v})$ 
reads as follows: for arbitrary $\alpha,\beta,\gamma,\delta=1,\dots,n$, 
\beq\label{wdvv64}
\sum_{\rho,\sigma=1}^n \frac{\p^3 F^M}{\p v^\alpha \p v^\beta \p v^\rho} 
\eta^{\rho\sigma} \frac{\p^3 F^M}{\p v^\sigma \p v^\gamma \p v^\delta} = \sum_{\rho,\sigma=1}^n \frac{\p^3 F^M}{\p v^\delta \p v^\beta \p v^\rho} 
\eta^{\rho\sigma} \frac{\p^3 F^M}{\p v^\sigma \p v^\gamma \p v^\alpha}.
\eeq
Equations~\eqref{wdvv64} are called the {\it WDVV equations of associativity} \cite{Du92, Du96}. We also have
$$
\frac{\p^3 F^M({\bf v})}{\p v^\iota \p v^\alpha \p v^\beta} = \eta_{\alpha\beta}, \quad \alpha,\beta=1,\dots,n.
$$
When there is an Euler vector field satisfying~\eqref{E2} and~\eqref{E3} we have 
$$
E\bigl(F^M({\bf v})\bigr)= (3-D) F^M({\bf v}) + \, {\rm some~quadratic~function~of~{\bf v}}.
$$
The above three sets of equations are called {\it WDVV equations}~\cite{Du96, Du99}. 
(Sometimes the associativity equations~\eqref{wdvv64} alone are called WDVV equations.)

In his study of symmetries of WDVV equations, Dubrovin~\cite{Du96} introduced 
the notion of {\it Legendre-type transformations} of a (pre-)Frobenius manifold. 
For any fixed $\kappa=1,\dots,n$, under a certain invertibility assumption, the Legendre-type transformation, 
labelled by~$\kappa$ and denoted by~$S_\kappa$, transforms~$M$ 
 to an $n$-dimensional Frobenius manifold $S_\kappa(M)$. One of $S_\kappa(M)$ always exists and is $M$ itself. 
 These Frobenius manifolds $S_\kappa(M)$ share the same multiplication, the same unity vector field
and the same Euler vector field, but they have different invariant flat metrics. 
Roughly speaking, the symmetry of the WDVV equations given by the Legendre-type transformation $S_\kappa$ 
can be described~\cite{Du96} as follows:
\begin{align}
& \hat v_{\alpha} = \frac{\p^2 F^M({\bf v})}{\p v^\kappa\p v^\alpha}, \quad \alpha=1,\dots,n, \label{vvhatcoord}\\
&\hat \eta_{\alpha\beta} = \eta_{\alpha\beta}, \quad \alpha,\beta=1,\dots,n, \label{hatmetricvv} \\
&\frac{\p^2 F^{S_\kappa(M)}}{\p \hat v^\alpha \p \hat v^\beta}
=\frac{\p^2 F^M}{\p v^\alpha \p v^\beta}, \quad \alpha,\beta=1,\dots,n,\label{FF56}
\end{align}
where $(\hat v^1,\dots,\hat v^n)=:\hat {\bf v}$ 
is a flat coordinate system for $S_\kappa(M)$, and 
$F^{S_\kappa(M)}=F^{S_\kappa(M)}(\hat {\bf v})$ is the potential for $S_\kappa(M)$ (satisfying WDVV equations of associativity). Here, 
$\hat v^\alpha=\sum_{\alpha,\beta=1}^n \eta^{\alpha\beta} \hat v_\beta$. 
We also use $(B,\hat \phi)$ to denote the coordinate chart for $\hat {\bf v}$. Note that $\hat \phi(B)$ itself may not be a ball, but it will be shown to be diffeomorphic to the small ball $\phi(B)$ under the invertibility assumption.
For the existence of $F^{S_\kappa(M)}$ and more details see Section~\ref{section22}.
It will be convenient to collect the Frobenius manifolds $S_\kappa(M)$ all together, and 
we call them the {\it cluster of Frobenius manifolds}. 

Recall that the {\it Dubrovin connection}~\cite{Du96, MS}, {\it aka} the {\it deformed flat connection}, on the (pre-)Frobenius manifold~$M$, 
denoted by~$\widetilde{\nabla}(z)$, $z\in\CC$,
is a one-parameter family of flat affine connections on the tangent bundle $TM$, defined by 
\beq\label{defDubrovinconn}
\widetilde \nabla(z)_{X} Y := \nabla_X Y + z \, X \cdot Y, \quad \forall \, X,Y\in \mathcal{X}(M).
\eeq
Here, we remind the reader that $\nabla$ is the Levi--Civita connection of the invariant flat metric. 
The Euler vector field enables 
one to extend the Dubrovin connection to a flat connection on $M\times \CC^*$ (see~\cite{Du96} 
or~\eqref{defiextendedDubconn} of Section~\ref{section2}), called the 
{\it extended Dubrovin connection} (aka the {\it extended deformed flat connection} or the {\it Dubrovin connection}), 
denoted by $\widetilde{\nabla}$. The differential system 
 associated to~$\widetilde{\nabla}$ has a Fuchsian singularity at $z=0$ and 
 an irregular singularity of Poincar\'e rank~1 at $z=\infty$. 
 By the {\it monodromy data at $z=0$ of~$M$} we mean the 
  monodromy data at $z=0$ of the differential system associated to the extended Dubrovin connection 
 $\widetilde{\nabla}$ on $M\times \CC^*$. It consists of two constant matrices $\mu$ and $R$. Similarly, by the 
 {\it monodromy data at $z=\infty$ of~$M$}, we mean 
  the monodromy data at $z=\infty$ of the differential system associated to 
  $\widetilde{\nabla}$, which is more difficult due to the irregularity.  
However, if $M$ is semisimple, then the monodromy data at $z=\infty$ can be given by 
a constant matrix~$S$ called the {\it Stokes matrix}~\cite{Du96, Du99}. Finally, the 
monodromy data of~$M$ also contains the {\it central connection matrix}~$C$, which 
is the transition matrix between two canonically-constructed fundamental matrix solutions to the 
differential system and is again constant. 
An important theorem of Dubrovin~\cite{Du96, Du99} is that, for the case when $M$ is semisimple, 
one can reconstruct the Frobenius  
structure of~$M$ from the monodromy data $\mu, R, S, C$. Note that when $M$ is semisimple, the 
Frobenius 
manifolds $S_\kappa(M)$ are also semisimple as they share the same multiplication. 
It is then very natural to ask the relations among the monodromy data of the 
Frobenius manifold cluster $S_\kappa(M)$. 
Actually, Dubrovin already showed that the same matrix $\mu$ 
can be served as those for $S_\kappa(M)$ without the semisimplicity assumption. 
We continue Dubrovin's study and show in Section~\ref{sectionmono} 
that the same constant matrix $R$ can be served as those for $S_\kappa(M)$ 
without the semisimplicity assumption.
Moreover, when $M$ is semisimple we show that the  
Frobenius manifold cluster $S_\kappa(M)$ actually share the same monodromy data $\mu, R, S, C$. 

The major source that motivates this study comes from the hermitian matrix model 
and its relations to Frobenius manifolds~\cite{Du09} (cf. also~\cite{CvdLPS, Y1, Y2, YZhou}).

To see this we first recall a few more terminologies. 
Recall that the {\it principal hierarchy} of a (pre-)Frobenius manifold~$M$ is 
a hierarchy of pairwise commuting evolutionary systems of PDEs of hydrodynamic-type, 
introduced by Dubrovin~\cite{Du92, Du96}:
\beq\label{principalh520}
\frac{\p v^\alpha}{\p t^{\beta,m}}  = 
\sum_{\gamma=1}^n \eta^{\alpha\gamma} \p_x \biggl(\frac{\p \theta_{\beta,m+1}}{\p v^\gamma}\biggr), \quad 
\alpha,\beta=1,\dots,n, \, m\ge0.
\eeq
Here, 
 $\theta_{\alpha, i}$ are the coefficients of a choice of the deformed flat coordinates of~$\widetilde{\nabla}(z)$ 
 (see Section~\ref{section2} for more details). 
 Since the $\p_{t^{\iota,0}}$-flow 
of the principal hierarchy~\eqref{principalh520} reads
\beq\label{ix1225}
 \frac{\p v^\alpha}{\p t^{\iota,0}} = \p_x (v^\alpha), \quad \alpha=1,\dots,n,
\eeq
we identify $x$ with $t^{\iota,0}$.
When $M$ is semisimple, Dubrovin--Zhang~\cite{DZ-norm} and independently Givental~\cite{Givental} (cf.~also~\cite{DZ-norm})
construct by using different approaches 
 the {\it topological partition function of~$M$}, aka 
the {\it total descendant potential 
of~$M$}. Under the semisimplicity assumption, for the case when $M$ comes from the 
quantum cohomology~\cite{RT}  of a smooth projective variety~$X$, the topological partition function of~$M$ 
coincides with the partition function of Gromov--Witten invariants of~$X$ (cf.~\cite{BF, Du96, DZ-norm, Givental, KM94, LT, Manin, MS, RT, Teleman}). 
Again under the semisimplicity assumption, when $M$ comes from a certain Landau--Ginzburg $A$-model, 
the topological partition function of~$M$ gives the partition 
function of the corresponding FJRW (Fan--Jarvis--Ruan, Witten) invariants \cite{FJR13, FFJMR, W93}.
The approach of Dubrovin and Zhang uses Virasoro constraints and the jet-variable representation, 
while the approach of Givental uses a quantization formulation. 
Details of their constructions are briefly reviewed in Appendix~\ref{appa86}.

It was observed by Dubrovin~\cite{Du09} (see also~\cite{Y1, Y2})
that the Gaussian Unitary Ensemble (GUE) partition function, that comes from the study of integrals over hermitian 
matrices~\cite{BIZ, DuY17, HZ, Mehta, Zhou}, 
can be identified with part of the topological partition function of the NLS Frobenius manifold. 
Here ``NLS" stands for {\it nonlinear Schr\"odinger}, as the Dubrovin--Zhang 
hierarchy (cf.~\cite{DZ-norm} or Section~\ref{DZapproachA}) associated to this Frobenius manifold 
is equivalent to the extended nonlinear Schr\"odinger hierarchy~\cite{CDZ, CvdLPS, FY, Y1, Y2}. 
Dubrovin~\cite{Du09} also found that the genus expansion of the logarithm of the GUE partition function can be 
obtained from the free energies of the Frobenius manifold 
corresponding to the quantum cohomology of the complex 
projective line~$\mathbb{P}^1$ (for short the $\mathbb{P}^1$ Frobenius manifold); see also~\cite{DuY17, Y1}. 
More precisely, the genus zero GUE free energy 
can be obtained by using the so-called genus zero two-point correlation functions 
of the $\mathbb{P}^1$ Frobenius manifold 
and by using a particular solution to the principal hierarchy of the $\mathbb{P}^1$ Frobenius manifold, and 
the higher genus GUE free energies can be obtained by using the jet representation of the higher genus free energies of the 
$\mathbb{P}^1$ Frobenius manifold 
and by using again the particular solution. 
These two observations suggest 
the identifications between certain particular partition functions of Frobenius manifolds under Legendre-type transformations.
The result on these identifications is given by Theorem~\ref{mainthm1229} of Section~\ref{sectionapp3}. 

In view of principal hierarchies of Frobenius manifolds, 
the Legendre-type transformation exchanges the roles of the space variable and a particular time variable. This means that, 
the principal hierarchy of the Legendre-type transformation $S_\kappa(M)$ of 
a Frobenius manifold~$M$ can be obtained by viewing the time variable $t^{\kappa,0}$
of the principal hierarchy of~$M$ as the new space variable and by using the 
coordinate transformation given by the Legendre-type transformation. Exchanging space and time in general 
belongs to linear reciprocal transformations~\cite{CDZ, LWZ, Pavlov, Tsarev, XZ}.  
It is shown in~\cite{XZ} that under a space-time exchange a bihamiltonian hierarchy of 
hydrodynamic type is transformed to a bihamiltonian one.
In~\cite{LWZ}, it is shown that under the space-time exchange a bihamiltonian hierarchy with dispersions 
is again transformed to a bihamiltonian one and that the central 
invariants~\cite{DLZ06, LZ05, Lorenzoni} of the bihamiltonian structures are kept invariant.
According to~\cite{DLZ}, when the central invariants are all constants, the bihamiltonian integrable hierarchies 
admit $\tau$-functions. It is then natural expect tau-functions for bihamiltonian hierarchies with constant central invariants under 
linear reciprocal transformations are related. Corollary~\ref{cortau} of Section~\ref{sectionapp3} give a positive answer to this for 
Dubrovin--Zhang hierarchies~\cite{BPS12-1, BPS12-2, DZ-norm, LWZ21} of semisimple Frobenius manifolds. 
We note that a few special cases along this direction can be found in \cite{CDZ, CvdLPS, CvdLPS2, FY, FYZ, Y1, Y2}.


The paper is organized as follows. In Section~\ref{section2}, we review the theory of Frobenius manifolds. 
In Section~\ref{section22}, we recall the definition of the Legendre-type transformation of a Frobenius manifold.
In Section~\ref{sectionmono} we prove the main theorems of this paper.  
Sections~\ref{sectionapp1}--\ref{sectionapp3} are devoted to applications of the main theorems. 
Appendix~\ref{appexamples} contains more examples of Legendre-type transformations and their monodromy data.

\section{Monodromy data for Frobenius manifolds: a brief review}\label{section2}
In this section, we review the theory of (pre-)Frobenius manifolds~\cite{Du96, Du99, DZ-norm}. The notations will be 
the same as those in the Introduction. 

As we have mentioned in the Introduction, a crucial geometric object on a pre-Frobenius manifold~$M$ 
is the Dubrovin connection $\widetilde \nabla(z)$ (see~\eqref{defDubrovinconn}). 
It is shown in~\cite{Du92} that 
 for any $z\in \CC$, $\widetilde \nabla(z)$ is flat.
Thus there exist $n$ power series 
\beq\label{expandtheta517}
\theta_\alpha({\bf v};z) =: \sum_{m\geq0} \theta_{\alpha,m}({\bf v})z^m,  \quad \alpha=1,\dots,n,
\eeq 
with $\theta_{\alpha,m}({\bf v})$, $\alpha=1,\dots,n$, $m\geq0$,  
being holomorphic functions of~${\bf v}$, such that  
\begin{align}
&
\theta_{\alpha,0}({\bf v})=v_\alpha, \qquad \alpha=1,\dots,n,\label{theta000511}\\
&
\frac{\p^2 \theta_{\gamma,m+1}({\bf v})}{\p v^\alpha\p v^\beta} = \sum_{\sigma=1}^n c^\sigma_{\alpha\beta}({\bf v}) 
\frac{\p \theta_{\gamma,m}({\bf v})}{\p v^\sigma}, \qquad \alpha,\beta=1,\dots,n, \, m\ge0,\label{theta1c}\\
&\bigl\langle \nabla \theta_\alpha({\bf v};z),  \nabla \theta_\beta({\bf v};-z) \bigr\rangle = \eta_{\alpha\beta}, \qquad \alpha,\beta=1,\dots,n.
\label{orthogtheta511}
\end{align}
We can further normalize $\theta_\alpha({\bf v};z)$ by requiring 
\beq\label{thetacali512}
\frac{\p\theta_{\alpha,m+1}({\bf v})}{\p v^\iota} = \theta_{\alpha,m}({\bf v}), \quad \alpha=1,\dots,n, \, m\geq0.
\eeq 

A choice of $(\theta_{\alpha,m}({\bf v}))_{\alpha=1,\dots,n,\,m\geq 0}$ 
satisfying \eqref{theta000511}--\eqref{thetacali512} is called a {\it calibration} on~$M$. With a chosen calibration 
the pre-Frobenius manifold is called {\it calibrated}. 
Below, by a pre-Frobenius manifold we often mean it is a calibrated one.

Following Dubrovin~\cite{Du92, Du96} (cf.~\cite{DZ-norm}), define the {\it genus zero two-point correlation functions} 
$\Omega_{\alpha,m_1;\beta,m_2}^{M, [0]}({\bf v})$, $\alpha,\beta=1,\dots,n$, $m_1,m_2\geq0$, 
for~$M$ by means of generating series as follows:
\beq\label{deftwopoint511}
\sum_{m_1,m_2\ge0} \Omega_{\alpha,m_1;\beta,m_2}^{M, [0]}({\bf v}) z^{m_1} w^{m_2} = 
\frac{\langle \nabla\theta_\alpha({\bf v};z), \nabla\theta_\beta({\bf v};w) \rangle-\eta_{\alpha\beta}}{z+w} , \qquad \alpha,\beta=1,\dots,n.
\eeq
From the definition we have
\begin{align}
& \Omega_{\alpha,m_1;\beta,m_2}^{M, [0]}({\bf v}) = \Omega_{\beta,m_2; \alpha,m_1}^{M, [0]}({\bf v}), \label{symmOO517}\\
& \Omega_{\alpha,m_1;\beta,m_2}^{M, [0]}({\bf v}) = 
\sum_{m=0}^{m_2} (-1)^m \sum_{\rho,\sigma=1}^n\frac{\p \theta_{\alpha,m_1+m+1}({\bf v})}{\p v^\rho} \eta^{\rho\sigma}
\frac{\p \theta_{\beta,m_2-m}({\bf v})}{\p v^\sigma}, \label{genuszerotwopointfor514}
\end{align}
where $\alpha,\beta=1,\dots,n$, $m_1,m_2\geq0$.
In particular, 
\beq
\Omega_{\alpha,m;\beta,0}^{M, [0]}({\bf v}) = \frac{\p \theta_{\alpha,m+1}({\bf v})}{\p v^\beta}, \quad 
\Omega_{\alpha,m;\iota,0}^{M, [0]}({\bf v}) = \theta_{\alpha,m}({\bf v}), \qquad \alpha=1,\dots,n,\, m\ge0. 
\eeq
It also follows from~\eqref{genuszerotwopointfor514} that 
\beq
 \frac{\p\Omega_{\alpha,m_1;\beta,m_2}^{M, [0]}({\bf v})}{\p v^\gamma} 
 = \sum_{\rho,\sigma=1}^n \frac{\p \theta_{\alpha,m_1}({\bf v})}{\p v^\rho} c^{\rho\sigma}_\gamma({\bf v}) \frac{\p \theta_{\beta,m_2}({\bf v})}{\p v^\sigma}, \quad \alpha,\beta,\gamma=1,\dots,n,\, m_1,m_2\geq0. \label{derivative_Omega}
\eeq
More particularly, 
\beq\label{deriOmisc514}
\frac{\p \Omega_{\alpha,0;\beta,0}^{M, [0]}({\bf v})}{\p v^\gamma} 
= c_{\alpha\beta\gamma}({\bf v}), \qquad \alpha,\beta,\gamma=1,\dots,n.
\eeq
By~\eqref{symmOO517} and~\eqref{deriOmisc514} 
we can require the potential 
$F^M({\bf v})$ (cf.~\eqref{dddF519}) to further satisfy 
\beq\label{FOmega512}
\frac{\p^2 F^M({\bf v})}{\p v^\alpha \p v^\beta} = \Omega_{\alpha,0;\beta,0}^{M, [0]}({\bf v}), \qquad \alpha,\beta=1,\dots,n.
\eeq

Let $(M,\, \cdot,\, e,\, \langle \,,\, \rangle,\, E)$ be a Frobenius manifold of charge~$D$. 
In the flat coordinate system ${\bf v}=(v^1,\dots,v^n)$ with $e=\p_\iota$, 
the Euler vector field $E$ has the form
\beq
E = \sum_{\alpha,\beta=1}^n Q^\alpha_\beta v^\beta \frac{\p}{\p v^\alpha} + \sum_{\beta=1}^n \tilde r^\beta \p_\beta,
\eeq
where $Q^\alpha_\beta$ and $\tilde r^\beta$ are constants. (See~\eqref{linear}.) 
It follows from~\eqref{E2} that 
\beq
-[E,e]=e.
\eeq
Therefore, $Q^\alpha_\iota=\delta^\alpha_\iota$, $\tilde r^\iota=0$, $\alpha=1,\dots,n$.
It is shown in~\cite{Du96} (cf.~\cite{Du99}) 
that the deformed flat connection $\widetilde\nabla(z)$ (see~\eqref{defDubrovinconn}) 
can be extended to a flat affine connection $\widetilde\nabla$ on $M\times \CC^*$ by 
$$
\widetilde \nabla_X Y = \widetilde \nabla(z)_X Y
$$
and 
\beq\label{defiextendedDubconn}
\widetilde \nabla_{\frac{\p}{\p z}} X := \frac{\p X}{\p z} + E \cdot X - \frac1 z \mathcal{V} X, \quad 
\widetilde \nabla_{\frac{\p}{\p z}} \frac{\p}{\p z} := 0,  \quad  \widetilde \nabla_X \frac{\p}{\p z} := 0
\eeq
for $X, Y \in \{\mbox{holomorphic vector field on } M \times \mathbb{C}^* \mbox{ with 
zero component along } \frac{\p}{\p z}\}$. Here, 
\beq
\mathcal{V}:=\frac{2-d}2 {\rm id} - \nabla E.
\eeq
We call~$\widetilde\nabla$ the {\it extended deformed flat connection} on $M\times\CC^*$. 
A holomorphic function $f=f({\bf v};z)$ on some open subset of $M\times \CC^*$ is called {\it $\widetilde \nabla$-flat} if 
\beq
\widetilde \nabla df=0, \quad d:=\sum_{\alpha=1}^n \frac{\p}{\p v^\alpha} dv^\alpha.
\eeq

We will assume that the flat coordinates $v^1,\dots,v^n$ can be chosen so that $\frac{\p}{\p v^\iota}=e$ and 
\begin{align}
&E=\sum_{\beta=1}^n E^\beta({\bf v}) \frac{\p}{\p v^\beta}, \quad E^\beta({\bf v}) 
= \Bigl(1-\frac{d}2-\mu_\beta\Bigr) v^\beta + r^\beta, \label{E517} 
\end{align}
where $\mu_1,\dots,\mu_n, r^1,\dots,r^n$ are constants,  
$\mu_\iota=-\frac{d}2$ and $r^\iota=0$. Denote $\mu={\rm diag}(\mu_1,\dots,\mu_n)$. 
Clearly, $\mu$ is the matrix for the linear operator $\mathcal{V}$ with respect to $\frac{\p}{\p v^\alpha}$, $\alpha=1,\dots,n$.

Following~\cite{Du96,Du99}, let $\mathcal{U}$ denote the matrix of multiplication by the Euler vector field:
\beq
\mathcal{U}^\alpha_\beta({\bf v}) = \sum_{\rho=1}^n E^\rho({\bf v}) c^\alpha_{\rho\beta}({\bf v}).
\eeq
Then 
the differential system for the gradient $y^\alpha=y^\alpha({\bf v};z)=\sum_{\beta=1}^n \eta^{\alpha\beta} \frac{\p f({\bf v};z)}{\p v^\beta}$ 
of a {\it $\widetilde \nabla$-flat} holomorphic function $f({\bf v};z)$ 
reads 
\begin{align}
& \frac{\p y}{\p v^\alpha} = z c_\alpha({\bf v}) y, \quad \alpha=1,\dots,n, \label{pdedub517}\\
& \frac{d y}{d z} = \Bigl(\mathcal{U}({\bf v})+\frac{\mu}{z} \Bigr) y, \label{odedub517}
\end{align}
where $y=(y^1,\dots,y^n)^T$ and $c_\alpha({\bf v})=(c_{\alpha\beta}^\gamma({\bf v}))$.
It is shown in~\cite{Du96, Du99} that for $\phi(B)$ sufficiently small, 
there exists a fundamental matrix solution to \eqref{pdedub517}--\eqref{odedub517} of the form
\beq\label{mY517}
\mathcal{Y}_0({\bf v};z) = \Theta({\bf v};z) z^\mu z^R,
\eeq
where $R$ is a constant matrix, $\Theta({\bf v};z)$ is an analytic matrix-valued function on $B\times \CC$ satisfying 
\beq
\Theta({\bf v};0) = I, \quad \eta^{-1} \Theta({\bf v};-z)^T \eta \Theta({\bf v};z)= I.
\eeq
The matrices $\mu,R$ are 
 the {\it monodromy data at $z=0$ of the Frobenius manifold} \cite{Du96, Du99}; they have the following properties:
\begin{align}
& R=\sum_{s\ge1} R_s, \quad z^\mu R z^{-\mu} = \sum_{s\ge1} R_s z^s, \label{Rmu512}\\
& \eta^{-1} R_s^T \eta = (-1)^{s+1} R_s, \quad s\ge1.
\end{align}
Here the two summations in~\eqref{Rmu512} are both finite sums. From~\eqref{Rmu512} we know that for $s\ge1$,
\beq
(R_s)^\alpha_\beta \neq 0 \quad \mbox{only if } \mu_\alpha-\mu_\beta =s.
\eeq 
This implies that $R$ is nilpotent. 
Note that the constant matrix $R$ may not be unique~\cite{Du99}; below we always fix a choice of~$R$.
 
The fundamental matrix solution~\eqref{mY517} gives the existence of $\widetilde \nabla$-flat multi-valued functions 
$\tilde{v}_1({\bf v};z),\dots,\tilde{v}_n({\bf v};z)$ on $B \times \CC^*$ of the form
\begin{align}
(\tilde{v}_1({\bf v};z),\dots,\tilde{v}_n({\bf v};z)) = (\theta_1({\bf v};z),\dots,\theta_n({\bf v};z)) z^{\mu} z^R,
\end{align}
where $\theta_1({\bf v};z),\dots,\theta_n({\bf v};z)$ 
are analytic on $B\times \CC$. 
We call $\tilde{v}_1({\bf v};z),\dots,\tilde{v}_n({\bf v};z)$ 
the {\it deformed flat coordinates} for $\widetilde \nabla$.
From the definition of~$\widetilde\nabla$ we see that 
the coefficients $\theta_{\alpha,m}({\bf v})$ of $\theta_{\alpha}({\bf v};z)$ as a power series of~$z$, 
$\alpha=1,\dots,n$, $m\geq0$, satisfy \eqref{theta000511}--\eqref{thetacali512} (that is why we use the same notation $\theta_{\alpha}({\bf v};z)$) and the following equation 
\begin{align}
& E \biggl(\frac{\p\theta_{\alpha,m}({\bf v})}{\p v^\beta}\biggr) = 
(m+\mu_\alpha+\mu_\beta) \, \frac{\p\theta_{\alpha,m}({\bf v})}{\p v^\beta} + 
\sum_{\gamma=1}^n \sum_{k=1}^m (R_k)^\gamma_\alpha  \frac{\p\theta_{\gamma,m-k}({\bf v})}{\p v^\beta} , \quad m\geq 0. \label{theta2c} 
\end{align}

A choice of holomorphic functions $(\theta_{\alpha,m}({\bf v}))_{\alpha=1,\dots,n,\,m\geq 0}$ 
satisfying \eqref{theta000511}--\eqref{thetacali512} and~\eqref{theta2c} 
is called a {\it calibration on~$M$ associated to $(\mu,R)$} (cf.~\cite{DLYZ, DZ-norm}), 
for short a {\it calibration}.
A Frobenius manifold with a chosen calibration is called {\it calibrated}. 
With a calibration for~$M$, 
one can verify 
that the genus zero two-point correlation functions 
defined in~\eqref{deftwopoint511} must satisfy  
\begin{align}
& 
E \Bigl(\Omega_{\alpha,m_1;\beta,m_2}^{M,[0]}({\bf v})\Bigr) 
= (m_1+m_2+1+\mu_\alpha+\mu_\beta) \Omega_{\alpha,p;\beta,q}^{M,[0]}({\bf v}) 
+ \sum_{r=1}^{m_1} \sum_{\gamma=1}^n (R_r)^\gamma_\alpha \Omega_{\gamma,m_1-r;\beta,m_2}^{M,[0]}({\bf v})  \nn\\
& \quad 
+\sum_{r=1}^{m_2} \sum_\gamma (R_r)^\gamma_\beta \Omega_{\alpha,m_1;\gamma,m_2-r}^{M,[0]}({\bf v})
+ \sum_{\gamma=1}^n (-1)^{m_2} (R_{m_1+m_2+1})_\alpha^\gamma \eta_{\gamma\beta}, \label{homoEO227}
\end{align}
where $\alpha,\beta=1,\dots,n$, $m_1,m_2\geq0$.
In particular, 
\begin{align}
& E \Bigl(\Omega_{\alpha,0;\beta,0}^{M,[0]}({\bf v})\Bigr)
= (1+\mu_\alpha+\mu_\beta) \Omega_{\alpha,0;\beta,0}^{M,[0]}({\bf v})
+ \sum_{\gamma=1}^n (R_{1})_\alpha^\gamma \eta_{\gamma\beta}. \label{EparticularO517}
\end{align}

We now restrict the consideration to the case that the Frobenius manifold~$M$ is {\it semisimple}. 
Following~\cite{Du96, Du99}, let $u_1({\bf v}),\dots, u_n({\bf v})$ be the eigenvalues 
of $\mathcal{U}({\bf v})$, i.e., 
\beq\label{charE}
\det (\lambda I-\mathcal{U}({\bf v})) = (\lambda-u_1({\bf v}))\cdots(\lambda-u_n({\bf v})),
\eeq
and let $M^0\subset M$ be the open set where all the roots are pairwise distinct.
According to~\cite{Du92, Du96, Du99}, 
the functions $u_1({\bf v}),\dots, u_n({\bf v})$ can be served as local coordinates on $M^0$ and 
\begin{align}
& \frac{\p}{\p u_i} \cdot \frac{\p}{\p u_j} = \delta_{ij} \frac{\p}{\p u_j}, \qquad i,j=1,\dots,n, \label{mulitiuiuj}\\
& \Bigl\langle \frac{\p}{\p u_i}, \frac{\p}{\p u_j} \Bigr\rangle =0, \qquad i\neq j. \label{produiuj}
\end{align}
The local coordinates $u_1,\dots,u_n$ on $M^0$ are called the {\it canonical coordinates}~\cite{Du96, Du98}.  
They are defined by~\eqref{charE} up to permutations, and are also uniquely determined by \eqref{mulitiuiuj}, \eqref{produiuj}
up to permutations and shifts by constants. 
Clearly, 
\beq\label{ecanonical227}
e= \sum_{i=1}^n \frac{\p}{\p u_i}.
\eeq

Take some point $p^*\in M^0$, and we will stay in some open neighborhood of~$p^*$, which is contained in~$M^0$. 
Denote by $f_1,\dots,f_n$ the orthonormal frame
\beq\label{finormalize}
f_i := \frac{1}{\sqrt{\eta_{ii}({\bf u})}}\frac{\p}{\p u_i}, \quad {\rm with}~\eta_{ii}({\bf u}):=\Bigl\langle \frac{\p}{\p u_i}, \frac{\p}{\p u_i} \Bigr\rangle, \quad i=1,\dots,n,
\eeq
choosing arbitrary signs of the square roots. 
Let $\Psi({\bf u})=(\psi_{i\alpha}({\bf u}))$ be the 
transition matrix from the orthonormal frame $(f_1,\dots,f_n)$ to the frame $(\frac{\p}{\p v^1},\dots,\frac{\p}{\p v^n})$, i.e.,
\beq\frac{\p}{\p v^\alpha}=\sum_{i=1}^n \psi_{i\alpha}({\bf u}) f_i , \quad \alpha=1,\dots,n.\eeq 
Then we have 
\beq
\Psi({\bf u}) \mathcal{U}({\bf v}({\bf u})) \Psi({\bf u})^{-1} = {\rm diag}(u_1,\dots,u_n) =: U({\bf u})
\eeq
and 
\beq\label{orthoeP}
\Psi({\bf u})^T \Psi({\bf u}) = \eta.
\eeq

Denote 
\beq\label{defUV0103}
 V({\bf u})=\Psi({\bf u}) \mu \Psi({\bf u})^{-1}.
\eeq
We know from~\cite{Du96} that $V({\bf u})$ is anti-symmetric. In other words, it takes values in ${\rm so}(n,\CC)$. 
This matrix-valued function satisfies~\cite{Du99} the following hamiltonian equations:
\beq\label{veq0103}
\frac{\p V}{\p u_i} = \{V, H_i(V; {\bf u})\}, \quad i=1,\dots,n,
\eeq
where $\{\,,\}$ denotes the Lie--Poisson bracket on ${\rm so}(n,\CC)$, i.e., 
\beq
\{V_{ij}, V_{kl}\} = V_{il} \delta_{j,k} - V_{jl}\delta_{i,k}+V_{jk} \delta_{i,l} - V_{ik} \delta_{j,l},
\eeq
and $H_i$ are hamiltonians (explicitly containing ${\bf u}$) given by 
\beq
H_i(V; {\bf u}) = \frac12 \sum_{j\neq i} \frac{V_{ij}^2}{u_i-u_j}, \quad i=1,\dots,n.
\eeq
It can be verified that the one-form $\sum_{i=1}^n H_i (V({\bf u}); {\bf u}) d u_i$
is closed. So by the Poincar\'e lemma we know that locally there exists an analytic function 
$\tau_I({\bf u})$, such that 
\beq\label{isotaueq}
d \log \tau_I({\bf u}) = \sum_{i=1}^n H_i (V({\bf u}); {\bf u}) d u_i.
\eeq
We call $\tau_I({\bf u})$ the {\it isomonodromic tau-function of the Frobenius manifold} (see e.g.~\cite{Du99}).

For a fundamental matrix solution $\mathcal{Y}({\bf v};z)$ to~\eqref{pdedub517}--\eqref{odedub517}, let 
\beq
Y({\bf u};z)=\Psi({\bf u}) \mathcal{Y}({\bf v}({\bf u});z). 
\eeq
Then $Y=Y({\bf u};z)$ satisfies the following equations:
\begin{align}
&\frac{\p Y}{\p u_i} = (z E_i+V_i({\bf u})) Y, \quad i=1,\dots,n, \label{Yeq10102}\\
&\frac{dY}{dz} = \Bigl(U({\bf u})+\frac{V({\bf u})}z\Bigr) Y, \label{Yeq20102}
\end{align}
where $(E_i)_{ab}= \delta_{ia} \delta_{ib}$ and $V_i({\bf u})={\rm ad}_{U({\bf u})}^{-1} ([E_i,V({\bf u})])$. 
Here we remind the reader that $u_1,\dots,u_n$ are pairwise distinct. 
Then according to~\cite{JMU} (cf.~\cite{Du96}), 
equations~\eqref{Yeq10102}--\eqref{Yeq20102} guarantee the isomonodromicity of the ODE system~\eqref{Yeq20102}. 
We note that equations \eqref{veq0103}, which are the compatibility equations for \eqref{Yeq10102}--\eqref{Yeq20102}, also 
guarantee~\cite{JMU} (cf.~\cite{Du96}) the isomonodromicity of~\eqref{Yeq20102} (we remind the reader that our matrix $V$ has been assumed to be diagonalizable).

Note that $z=\infty$ is an irregular singularity of~\eqref{Yeq20102} 
of Poincar\'e rank~1. Then \cite{Du96, Du99} (cf.~\cite{CDG19, CDG20}) 
 equations~\eqref{Yeq10102}--\eqref{Yeq20102} have a unique formal solution $Y_{\rm formal}({\bf u};z)$ of the form
\beq
Y_{\rm formal}({\bf u};z) = \Phi ({\bf u};z) e^{z U({\bf u})},
\eeq
where $\Phi ({\bf u};z)$ is a formal power series of~$z^{-1}$ 
\beq
\Phi ({\bf u};z) = \sum_{k\geq0} \Phi_{k} ({\bf u}) z^{-k}, \quad \Phi_{0} ({\bf u}) = I,
\eeq
satisfying the orthogonality condition
\beq
\Phi ({\bf u};-z)^T \Phi ({\bf u};z) = I.
\eeq

A line $\ell$ through the origin in the complex $z$-plane is called
{\it admissible} if
\beq\label{admissible32}
{\rm Re} \, z(u_i- u_j)|_{z \in \ell \setminus 0} \neq 0, \quad \forall\, i\neq j.
\eeq
Fix any admissible line $\ell$ with an orientation. Denote by $\phi\in[0,2\pi)$ the angle from the positive real axis and to the positive direction 
of~$\ell$.  We then have a positive part $\ell_+$ of~$\ell$ and a negative part $\ell_-$:
\beq
\ell_+=\{z \mid {\rm arg} \, z=\phi\}, \quad \ell_- = \{z \mid {\rm arg} \, z= \phi-\pi\}.
\eeq
Construct two sectors
\begin{align}
\Pi_{\rm right}: \phi-\pi-\varepsilon<{\rm arg} \, z< \phi + \varepsilon, \\ 
\Pi_{\rm left}: \phi-\varepsilon<{\rm arg} \, z< \phi + \pi+ \varepsilon,
\end{align}
where $\varepsilon$ is a sufficiently small positive number. We know from~\cite{BJL0, BJL, Du96, Du99, Gu21} (cf.~\cite{Gu16}) that 
there exist unique fundamental matrix solutions $Y_{{\rm right}/{\rm left}}({\bf u};z)$
to the ODE system~\eqref{Yeq20102}, analytic in $\Pi_{{\rm right}/{\rm left}}$, 
respectively, such that 
\beq
Y_{{\rm right}/{\rm left}}({\bf u};z) \sim Y_{\rm formal}({\bf u};z)
\eeq
as $|z|\to\infty$ within the sectors. Within the narrow sector
\beq
\Pi_+: \phi-\varepsilon < {\rm arg} z < \phi+\varepsilon
\eeq
we have two analytic fundamental matrix solutions to the ODE system~\eqref{Yeq20102}, which must be related by a matrix $S$:
\beq\label{defiStokes}
Y_{{\rm left}}({\bf u};z) = Y_{{\rm right}}({\bf u};z) S, \quad z\in \Pi_+.
\eeq 
Similarly, according to~\cite{Du96, Du99} (cf.~\cite{BJL}), in the narrow sector
\beq
\Pi_-: \phi-\pi-\varepsilon < {\rm arg} z < \phi-\pi+\varepsilon,
\eeq
we have
\beq
Y_{{\rm left}}({\bf u};z) = Y_{{\rm right}}({\bf u};z) S^T, \quad z\in \Pi_-.
\eeq 
Moreover, by isomonodromicity the matrix~$S$ 
does not depend on~${\bf u}$ (cf.~\cite{Du96, Du99}). We call~$S$ the {\it Stokes matrix subjected to the 
admissible line~$\ell$}, for short the {\it Stokes matrix}. This Stokes matrix gives the monodromy data of~\eqref{Yeq20102} at $z=\infty$.

Within the narrow sector $\Pi_+$, the fundamental matrix solutions 
$Y_{0}({\bf u};z):=\Psi({\bf u}) \mathcal{Y}_0({\bf v}({\bf u});z)$ and $Y_{{\rm right}}({\bf u};z)$ 
to the ODE system~\eqref{Yeq20102} must also be related by a matrix~$C$: 
\beq
Y_{{\rm right}}({\bf u};z) = Y_{0}({\bf u};z) C.
\eeq
The matrix $C$, called the {\it central connection matrix subjected to the 
admissible line~$\ell$}, again does not depend on~${\bf u}$. The following key identities were given in~\cite{Du96, Du99}:
\begin{align*}
& C S^T S^{-1} C^{-1} = e^{2\pi i(\mu+R)}, \nn\\
& S= C^{-1} e^{-\pi i R} e^{-\pi i \mu} \eta^{-1} (C^{-1})^T. \nn
\end{align*}

The quintuple $(\mu, e, R, S, C)$ is called the {\it monodromy data of~$M$~\cite{Du99}}. 
Here $e$ is a selected eigenvector of~$V$ with the eigenvalue~$\mu_\iota$.

Before ending this section, we mention that some of the terminologies were generalized to the set of semisimple points of~$M$; we refer to~\cite{CDG19, CDG20, Gu21} for details.

\section{The Legendre-type transformations}\label{section22}
In this section we give a revisit of 
the Legendre-type transformations~\cite{Du96}. 

Let $M$ be a calibrated pre-Frobenius manifold. 
Fix $\kappa\in\{1,\dots,n\}$. Following Dubrovin~\cite{Du96} (cf.~\eqref{vvhatcoord}, \eqref{FOmega512}),  
we define the holomorphic functions $\hat v_\alpha=\hat v_\alpha({\bf v})$, $\alpha=1,\dots,n$, on~$B$ by 
\beq\label{defhatv59}
\hat v_\alpha({\bf v}) = \Omega_{\alpha,0;\kappa,0}^{M, [0]}({\bf v}), \quad \alpha=1,\dots,n.
\eeq

\begin{lemma}\label{invertible}
Assuming that $\frac{\p}{\p v^\kappa}\cdot$ is invertible on~$B$ for $\phi(B)$ being sufficiently small,  
the holomorphic functions $\hat v_\alpha({\bf v})$, $\alpha=1,\dots,n$,
give a system of local coordinates on~$B$.
\end{lemma}
\begin{proof}
Using the definition~\eqref{defhatv59} and using~\eqref{deriOmisc514}, we find 
\beq\label{1jacobianvhatv513}
\frac{\p \hat v_\alpha({\bf v})}{\p v_\beta} = c_{\kappa\alpha}^\beta({\bf v}), \quad \alpha,\beta=1,\dots,n.
\eeq
Since $\frac{\p}{\p v^\kappa}\cdot$ is invertible on~$B$, we see from~\eqref{1jacobianvhatv513} that the Jacobian determinant 
$\det\bigl(\frac{\p \hat v_\alpha({\bf v})}{\p v_\beta}\bigr)_{1\le\alpha,\beta\le n}$ 
is nowhere vanishing on~$B$. The lemma is proved.
\end{proof}

We have the following proposition.

\begin{prop}[\cite{Du96, DZ-norm}] \label{propprefm510}
Assume that $\frac{\p}{\p v^\kappa}\cdot$ is invertible on the open set~$B$.
Define a family of bilinear products $\langle\,,\,\rangle^{\wedge}$ on $T_pB$, $p\in B$, as follows:
\begin{align}
& \langle x, y\rangle^{\wedge} = \Bigl\langle \frac{\p}{\p v^\kappa} \cdot x\,,\, \frac{\p}{\p v^\kappa} \cdot y \Bigr\rangle, \quad \forall \, x,y\in T_p B. \label{defprod59}
\end{align}
Then the triple 
$(\cdot,\, e, \, \langle\,,\,\rangle^{\wedge})$
forms a pre-Frobenius structure on~$B$. 
\end{prop}
\begin{proof}
For every $p \in B$, we know that $(T_p B,\, e_p, \, \langle\,,\,\rangle)$ is 
a Frobenius algebra with the multiplication~``$\,\cdot\,$". Then using the definition~\eqref{defprod59} it can be verified easily that 
$(T_p B,\, e_p, \, \langle\,,\,\rangle^\wedge)$ is 
a Frobenius algebra with the same multiplication. Here we note that the invertibility of~$\frac{\p}{\p v^\kappa}\cdot$ implies the 
nondegeneracy for $\langle\,,\,\rangle^\wedge$.
Obviously, the product $\langle\,,\, \rangle^\wedge$ depends on~$p$ 
holomorphically.

By Lemma~\ref{invertible} we know that the holomorphic 
$\hat v_1,\dots,\hat v_n$, defined in~\eqref{defhatv59}, can serve as local coordinates on~$B$. 
Before continuing, let us prove a helpful lemma.

\begin{lemma} \label{anotherusefullemma}
The following identities hold:
\beq
\frac{\p}{\p v_\alpha} = \frac{\p}{\p v^\kappa} \cdot \frac{\p}{\p \hat v_\alpha}, \quad \alpha=1,\dots,n.
\eeq
\end{lemma}
\begin{proof}
We have
\beq
\frac{\p}{\p v^\kappa} \cdot \frac{\p}{\p \hat v_\alpha} 
= \sum_{\beta=1}^n \frac{\p v_\beta}{\p \hat v_\alpha} \frac{\p}{\p v^\kappa} \cdot \frac{\p}{\p v_\beta} 
= \sum_{\beta,\gamma=1}^n \frac{\p v_\beta}{\p \hat v_\alpha} c_{\kappa\gamma}^\beta \frac{\p}{\p v_\gamma}
= \frac{\p}{\p v_\alpha} .
\eeq
Here in the last equality we used~\eqref{1jacobianvhatv513}. The lemma is proved.
\end{proof}

We continue to give the verifications of the axioms A1, A2 for 
$(B,\, \cdot,\, e,\, \langle\,,\,\rangle^\wedge)$. Using the definition~\eqref{defprod59} and Lemma~\ref{anotherusefullemma}, we have
\beq\label{metricnew513}
\Bigl\langle \frac{\p }{\p \hat v_\alpha} , \frac{\p }{\p \hat v_\beta}\Bigr\rangle^\wedge
=
\Bigl\langle  \frac{\p}{\p v^\kappa}  \cdot \frac{\p }{\p \hat v_\alpha}, 
 \frac{\p}{\p v^\kappa}  \cdot \frac{\p }{\p \hat v_\beta} \Bigr\rangle
=\Bigl\langle \frac{\p}{\p v_\alpha}, \frac{\p}{\p v_\beta} \Bigr\rangle = \eta^{\alpha\beta}, \quad \alpha,\beta=1,\dots,n.
\eeq
Therefore, $\langle\,,\, \rangle^\wedge$ is a flat metric on~$B$, and 
$\hat v_1,\dots,\hat v_n$ are flat coordinates for $\langle\,,\, \rangle^\wedge$. 
Denote by~$\hat \nabla$ the Levi-Civita connection of $\langle\,,\, \rangle^\wedge$, and denote 
as in the Introduction 
$\hat v^\alpha := \sum_{\beta=1}^n \eta^{\alpha\beta} \hat v_\beta$,  $\alpha=1,\dots,n$, and $\hat {\bf v}:=(\hat v^1,\dots,\hat v^n)$.
Since $(\eta^{\alpha\beta})$ is non-degenerate, the holomorphic functions $\hat v^1,\dots, \hat v^n$ 
also form a flat coordinate system for $\langle\,,\, \rangle^\wedge$ on~$B$. 
The identity~\eqref{metricnew513} can be equivalently written as
\beq\label{flatmetricnew2513}
\Bigl\langle \frac{\p}{\p \hat v^\alpha}, \frac{\p}{\p \hat v^\beta}\Bigr\rangle^\wedge=\eta_{\alpha\beta}, \quad \alpha,\beta=1,\dots,n.
\eeq
We also have
\beq\label{ee513}
e = \frac{\p}{\p v^\iota} = \sum_{\rho=1}^n \frac{\p \hat v^\rho}{\p v^\iota}  \frac{\p}{\p \hat v^\rho}
= \sum_{\rho=1}^n c_{\kappa\iota}^\rho \frac{\p}{\p \hat v^\rho}
=  \frac{\p}{\p \hat v^\kappa},
\eeq
which is flat with respect to~$\hat \nabla$. Hence we have finished the verification for  
the axiom A1.

The multiplication 
tensor $c_{\alpha\beta}^\gamma$, 
written in the coordinate system $(\hat v^1,\dots,\hat v^n)$ on~$B$, denoted $\tilde c_{\alpha\beta}^\gamma$, has the expression
\begin{align}
\tilde c_{\alpha\beta}^\gamma 
&= \sum_{\rho,\sigma,\phi=1}^n
c_{\rho\sigma}^\phi \frac{\p v^\rho}{\p \hat v^\alpha}\frac{\p v^\sigma}{\p \hat v^\beta}\frac{\p \hat v^\gamma}{\p v^\phi}
= \sum_{\rho=1}^n \frac{\p v^\rho}{\p \hat v^\alpha} c^\gamma_{\rho\beta}, \label{tildec513}
\end{align}
where we used associativity and~\eqref{1jacobianvhatv513}. By definition, 
\beq\label{muliplicationonhats513}
\frac{\p}{\p \hat v^\alpha} \cdot \frac{\p}{\p \hat v^\beta} = \sum_{\rho=1}^n \tilde c_{\alpha\beta}^\rho \frac{\p}{\p \hat v^\rho}. 
\eeq
To verify the axiom A2, define  
 the 3-tensor~$\hat c$ by 
$$
 \hat c (x,y,z):=\langle x \cdot y, z\rangle^\wedge, \quad \forall\, x,y,z\in T_p B, \, \forall\,p\in B.
$$
 We know that $\hat c$ is a symmetric 3-tensor on~$B$. 
In the coordinate system $(\hat v^1,\dots,\hat v^n)$,
\beq\label{hatcc519}
\hat c_{\alpha\beta\gamma} := \Bigl\langle \frac{\p}{\p \hat v^\alpha} \cdot \frac{\p}{\p \hat v^\beta}, \frac{\p}{\p \hat v^\gamma} \Bigr\rangle^\wedge 
=\sum_{\rho=1}^n
 \frac{\p v^\rho}{\p \hat v^\alpha} c_{\rho\beta\gamma},
\eeq
where we used \eqref{muliplicationonhats513}, \eqref{tildec513} and~\eqref{flatmetricnew2513}.
So
\beq
\frac{\p \hat c_{\alpha\beta\gamma}}{\p \hat v^\delta} = 
\sum_{\rho=1}^n \frac{\p^2 v^\rho}{\p \hat v^\alpha \p \hat v^\delta} c_{\rho\beta\gamma} + 
 \sum_{\rho,\varepsilon=1}^n
 \frac{\p v^\rho}{\p \hat v^\alpha} \frac{\p c_{\rho\beta\gamma}}{\p v^\varepsilon} \frac{\p v^\varepsilon}{\p \hat v^\delta}.
\eeq
Since $\nabla c$ is a symmetric 4-tensor, we see that the right-hand side of the above identity is 
symmetric with respect to exchanging the indices~$\alpha$ and~$\delta$. 
Therefore, the 4-tensor field $\hat \nabla \hat c$ is symmetric. The axiom A2 is verified.  The proposition is proved.
\end{proof}

As in~\cite{Du96}, we call the pre-Frobenius manifold $(B,\, \cdot,\, e,\, \langle \,,\, \rangle^\wedge)$
 the {\it Legendre-type transformation of~$M$ given by $\frac{\p}{\p v^\kappa}$} 
 and denote this pre-Frobenius manifold for short by $S_\kappa(M)$. (Of course, the Legendre-type transformation 
 of $S_\kappa(M)$ given by $\frac{\p}{\p \hat v^\iota}$ is the pre-Frobenius manifold~$M$.)
We note that the following identity obviously holds:
\beq
\tilde c_{\alpha\beta}^\gamma = \hat c_{\alpha\beta}^\gamma, \quad \alpha,\beta,\gamma=1,\dots,n,
\eeq
where $\hat c_{\alpha\beta}^\gamma:=\sum_{\rho=1}^n \hat c_{\alpha\beta\rho} \eta^{\rho\gamma}$. 

\begin{prop}\label{propprefm510-2}
The holomorphic functions $\hat \theta_{\alpha,m}=\hat \theta_{\alpha,m}(\hat {\bf v})$, defined by
\beq\label{caliSkappa}
\hat \theta_{\alpha,m}(\hat {\bf v}) := \Omega^{M,[0]}_{\alpha, m; \kappa, 0} ({\bf v}(\hat {\bf v})), \quad 
\alpha=1,\dots,n,\, m\geq0,
\eeq
give a calibration for the pre-Frobenius manifold $S_\kappa(M)$. 
\end{prop}
\begin{proof}
Let us first prove the following identity, which was known in~\cite{DLZ, DZ-norm}:
\beq\label{hattt517}
\frac{\p \hat\theta_{\gamma,m}}{\p \hat v^\beta} = \frac{\p \theta_{\gamma,m}}{\p v^\beta}, \quad \forall\,\beta,\gamma=1,\dots,n,\, \forall\,m\geq0.
\eeq
Using \eqref{caliSkappa}, \eqref{derivative_Omega}, \eqref{1jacobianvhatv513}, we find that the left-hand side of~\eqref{hattt517} is equal to
\beq
\frac{\p \Omega^{M, [0]}_{\gamma,p; \kappa, 0}}{\p \hat v^\beta}
= \sum_{\rho=1}^n \frac{\p v^\rho}{\p \hat v^{\beta}}
\frac{\p \Omega^{M, [0]}_{\gamma, m; \kappa, 0}}{\p v^\rho} 
= \sum_{\rho, \phi=1}^n \frac{\p v^\rho}{\p \hat v^{\beta}}
\frac{\p \theta_{\gamma, m}}{\p v^\phi} c^{\phi}_{\rho\kappa} 
= \sum_{\rho, \phi=1}^n \frac{\p v^\rho}{\p \hat v^{\beta}}
\frac{\p \theta_{\gamma, m}}{\p v^\phi} \frac{\p \hat v^{\phi}}{\p v^\rho},
\eeq
which equals the right-hand side of~\eqref{hattt517}.

We proceed to the verifications of \eqref{theta000511}--\eqref{thetacali512} for $S_\kappa(M)$.
When $m=0$, we have $\hat \theta_{\alpha,0}(\hat {\bf v})=\hat v_\alpha$.  
For $m\geq0$, we have
\begin{align}
&  \frac{\p^2 \hat \theta_{\gamma,m+1}}{ \p \hat v^\alpha \p \hat v^\beta}
= \frac{\p}{\p \hat v^\alpha} \biggl(\frac{\p \theta_{\gamma,m+1}}{\p v^\beta} \biggr)  
= \sum_{\sigma=1}^n
\frac{\p^2 \theta_{\gamma,m+1}}{\p v^\beta \p v^\sigma} \frac{\p v^\sigma}{\p \hat v^\alpha} 
= \sum_{\sigma,\rho=1}^n c_{\beta\sigma}^\rho 
\frac{\p \theta_{\gamma,m}}{\p v^\rho} \frac{\p v^\sigma}{\p \hat v^\alpha} \nn\\
&= \sum_{\rho, \sigma, \phi, \chi=1}^n \frac{\p v^\sigma}{\p \hat v^\alpha} c^\rho_{\sigma\beta}    
\frac{\p \theta_{\gamma,m}}{\p v^\chi} \frac{\p \hat v^\chi}{\p v^\phi} \frac{\p v^\phi}{\p \hat v^\rho} 
= \sum_{\rho, \sigma, \phi, \chi=1}^n\frac{\p v^\sigma}{\p \hat v^\alpha} c^\rho_{\sigma\beta}   
\frac{\p \theta_{\gamma,m}}{\p v^\chi} c_{\phi \kappa}^\chi \frac{\p v^\phi}{\p \hat v^\rho} \nn\\
&= \sum_{\rho, \sigma, \phi=1}^n \frac{\p v^\sigma}{\p \hat v^\alpha} c^\rho_{\sigma\beta} 
 \frac{\p \Omega^{M, [0]}_{\gamma,m; \kappa, 0}}{\p v^\phi} \frac{\p v^\phi}{\p \hat v^\rho} 
 = \sum_{\rho=1}^n 
 \hat c^\rho_{\alpha\beta} \frac{\p \Omega^{M, [0]}_{\gamma,m; \kappa, 0}}{\p \hat v^\rho} = 
  \sum_{\rho=1}^n \hat c_{\alpha\beta}^\rho \frac{\p\hat \theta_{\gamma,m}}{\p \hat v^\rho}. \nn
\end{align}
Here, we have used \eqref{hattt517}, \eqref{caliSkappa} and \eqref{derivative_Omega}. 
This verifies~$\eqref{theta1c}$ for $S_\kappa(M)$. 
By using~\eqref{hattt517},  \eqref{orthogtheta511} and~\eqref{flatmetricnew2513}, we immediately get 
\beq
\bigl\langle \hat \nabla \hat \theta_\alpha(\hat {\bf v};z),  
\hat \nabla \hat \theta_\beta(\hat {\bf v};-z) \bigr\rangle^\wedge = \eta_{\alpha\beta}, \qquad \alpha,\beta=1,\dots,n,
\eeq 
where 
\beq \label{defhatthetaz}
\hat \theta_\alpha(\hat{\bf v};z):=\sum_{m\ge0} \hat \theta_{\alpha,m}(\hat {\bf v}) z^m.
\eeq
This verifies~\eqref{orthogtheta511} for $S_\kappa(M)$. Finally, 
\begin{align}
\frac{\p \hat \theta_{\alpha,m+1}}{\p \hat v^\kappa} = \frac{\p \theta_{\alpha,m+1}}{\p v^\kappa} 
= \Omega^{M, [0]}_{\alpha,m;\kappa,0} = \hat\theta_{\alpha,m}, \quad \alpha=1,\dots,n, \,m\geq0,
\end{align}
which verifies~\eqref{thetacali512} for $S_\kappa(M)$.
The proposition is proved.
\end{proof}

The above Proposition~\ref{propprefm510-2} tells the following: once the pre-Frobenius manifold $M$ is calibrated, 
under the assumption that $\frac{\p}{\p v^\kappa}\cdot$ is invertible on~$B$, 
the pre-Frobenius manifold $S_\kappa(M)$, i.e., the quadruple $(B,\, \cdot,\, e,\, \langle \,,\, \rangle^\wedge)$, can 
be calibrated by using the calibration of~$M$. 

From now on we assume that 
$\frac{\p}{\p v^\kappa}\cdot$ is invertible on $B$ for $\phi(B)$ being sufficiently small.
By Proposition~\ref{propprefm510-2}, the holomorphic functions $\hat\theta_{\alpha,m}(\hat {\bf v})$ 
(see~\eqref{caliSkappa}) give a calibration on $S_\kappa(M)$. 
Denote by $\Omega_{\alpha,m_1;\beta,m_2}^{S_\kappa(M), [0]}$ the two-point correlations functions of 
$S_\kappa(M)$ corresponding to this calibration. By definition, we know that 
\beq\label{hatdeftwopoint514}
\sum_{m_1,m_2\ge0} \Omega_{\alpha,m_1;\beta,m_2}^{S_\kappa(M), [0]}(\hat {\bf v}) z^{m_1} w^{m_2} = 
\frac{\bigl\langle \hat\nabla \hat\theta_\alpha(\hat {\bf v};z), \hat \nabla \hat \theta_\beta(\hat {\bf v};w) \bigr\rangle^\wedge-\eta_{\alpha\beta}}{z+w} .
\eeq
Then by using~\eqref{genuszerotwopointfor514} and~\eqref{hattt517}, we arrive at the following corollary. 
\begin{cor} \label{simplecor514}
We have the identities 
\beq\label{identitiestwopointOmega}
\Omega_{\alpha,m_1;\beta,m_2}^{S_\kappa(M), [0]}(\hat {\bf v}) 
= \Omega_{\alpha,m_1;\beta,m_2}^{M, [0]}({\bf v}(\hat {\bf v})), \quad \forall\,\alpha,\beta=1,\dots,n,\, \forall\,m_1,m_2\geq0.
\eeq
\end{cor}

Let $F^{S_\kappa(M)}(\hat {\bf v})$ denote the potential (cf.~\eqref{FOmega512}) of the pre-Frobenius 
manifold $S_\kappa(M)$ calibrated by using~$\hat\theta_\alpha(\hat {\bf v};z)$. 
Then from Corollary~\ref{simplecor514} we know that the above formulation agrees with 
the original set of equations~\eqref{vvhatcoord}--\eqref{FF56} given by Dubrovin. 

We proceed to the case of a Frobenius manifold 
$(M, \, \cdot, \, e, \, \langle\,,\,\rangle, \, E)$. 
Denote by $\mu, R$ the monodromy data at $z=0$ of this Frobenius manifold. 
As in Section~\ref{section2}, we assume that $\frac{\p}{\p v^\iota}=e$ and \eqref{E517}.  
We know from Proposition~\ref{propprefm510} that $(\cdot,\, e,\, \langle\,,\,\rangle^{\wedge})$ 
gives a pre-Frobenius structure on~$B$ with the flat coordinates $\hat v^1,\dots,\hat v^n$.
Write the Euler vector field~$E$ in this coordinate system as follows:
$$
E =: \sum_{\beta=1}^n \hat E^\beta(\hat {\bf v}) \frac{\p}{\p \hat v^\beta}.
$$
Then by a straightforward calculation using~\eqref{EparticularO517} we find 
\begin{align}
& \hat E^\beta(\hat {\bf v}) = \biggl(1-\frac{\hat D}2- \mu_\beta\biggr) \hat v^\beta + \hat r^\beta, \quad \beta=1,\dots,n, \label{skME0102-2}
\end{align}
where $\hat D$ and $\hat r^\beta$ are contants given by 
\beq\label{skME0102-3}
\hat D := -2 \mu_\kappa, \qquad \hat r^\beta := (R_1)^{\beta}_{\kappa}, \quad \beta=1,\dots,n.
\eeq
It follows immediately that $\hat \nabla \hat \nabla E=0$, 
as well as that 
\begin{align}
&\sum_{\gamma=1}^n \frac{\p \hat E^\gamma}{\p \hat v^\alpha} \hat \eta_{\gamma\beta} + 
\sum_{\gamma=1}^n \frac{\p \hat E^\gamma}{\p \hat v^\beta} \hat \eta_{\gamma\alpha}
= \biggl(1-\frac{\hat D}2 - \mu_\alpha\biggr) \eta_{\alpha\beta} + 
\biggl(1-\frac{\hat D}2- \mu_\beta\biggr) \eta_{\beta\alpha} = \bigl(2-\hat D\bigr) \eta_{\alpha\beta}.
\end{align}
It is also not difficult to verify that 
\begin{align}
&\sum_{\gamma=1}^n\biggl(\hat E^\gamma \frac{\p \hat c_{\alpha\beta}^\sigma}{\p \hat v^\gamma} + \frac{\p \hat E^\gamma}{\p \hat v^\beta} \hat c^\sigma_{\gamma\alpha} 
+ \frac{\p \hat E^\gamma}{\p \hat v^\alpha} \hat c^\sigma_{\gamma\beta} - \frac{\p \hat E^\sigma}{\p \hat v^\gamma} \hat c^\gamma_{\alpha\beta} \biggr)
= \hat c_{\alpha\beta}^\sigma.
\end{align}
These show that the Euler vector field~$E$ for~$M$ can also be served as the 
Euler vector field for $S_\kappa(M)$ to make it a 
Frobenius manifold of charge~$\hat D=-2\mu_\kappa$. 

Hence we arrive at the following theorem. 

\smallskip

\noindent {\bf Theorem A} (\cite{Du96})
{\it The quintuple  
$(B, \cdot,\, e,\, \langle\,,\,\rangle^{\wedge},\, E)$
is a Frobenius manifold of charge 
$-2\mu_\kappa$.}

\smallskip

As in~\cite{Du96} we call  
$(B,\, \cdot,\, e,\, \langle\,,\,\rangle^{\wedge},\, E)$ 
the {\it Legendre-type transformation of~$M$ given by $\frac{\p}{\p v^\kappa}$}, and 
 denote this Frobenius manifold for short as $S_\kappa(M)$. 
 (Again the Legendre-type transformation of $S_\kappa(M)$ given by $\frac{\p}{\p \hat v^\iota}$ is the 
 Frobenius manifold~$M$.)
 Denote by~$\I\subset \{1,\dots,n\}$ 
 the index set such that for each $\kappa\in \I$, $\frac{\p}{\p v^\kappa}\cdot$ is invertible on some open set. We call the collection of 
 Frobenius manifolds $S_\kappa(M)$, $\kappa\in \I$, the {\it cluster of Frobenius 
 manifolds}.  

\section{Monodromy data for $S_\kappa(M)$}\label{sectionmono}
In this section, 
we study the relations among the monodromy data of the Frobenius manifold cluster  
$S_\kappa(M)$, $\kappa\in \I$.

\subsection{Monodromy data at $z=0$}
Recall that 
the Frobenius manifold $S_\kappa(M)$ shares with~$M$ the same unity vector field~$e$, the same Euler vector field~$E$ and the same 
multiplication ``$\,\cdot\,$". 
The charge $\hat D$ of the Frobenius manifold $S_\kappa(M)$ is given by $\hat D=-2\mu_\kappa$.
Recall also from the previous section that 
in the coordinate system $(\hat v^1,\dots,\hat v^n)$, we have $e=\frac{\p}{\p \hat v^\kappa}$.

\begin{theorem}\label{thm2519}
The monodromy data $\mu,R$ at $z=0$ for the Frobenius manifold~$M$ can be 
served as the monodromy data at $z=0$ for~$S_\kappa(M)$.
Moreover, the holomorphic functions $\hat \theta_{\alpha,m}(\hat {\bf v})$ defined by~\eqref{caliSkappa} 
give a calibration on~$S_\kappa(M)$ associated to $\mu,R$. 
\end{theorem}
\begin{proof}
We prove the first statement by constructing a fundamental matrix solution 
to \eqref{pdedub517}--\eqref{odedub517} for the Frobenius manifold~$S_\kappa(M)$. Indeed, define 
$\hat \Theta (\hat {\bf v};z)=(\hat \Theta^\alpha_\beta (\hat {\bf v};z))$ by 
\beq
\hat \Theta^\alpha_\beta (\hat {\bf v};z) 
:= \sum_{\gamma=1}^n \eta^{\alpha\gamma}\frac{\p \hat \theta_{\beta}(\hat {\bf v}; z)}{\p \hat v^\gamma}, \quad \alpha,\beta=1,\dots,n,
\eeq
where $\hat \theta_{\beta}(\hat {\bf v}; z)$ is defined in~\eqref{defhatthetaz}. 
It has the following properties:
\beq
\hat \Theta(\hat {\bf v};0) = I, \quad \eta^{-1} \hat \Theta(\hat {\bf v};-z)^T \eta \hat \Theta(\hat {\bf v};z)= I.
\eeq
From the identity~\eqref{hattt517} we see that 
\beq
\hat \Theta^\alpha_\beta (\hat {\bf v};z) = \Theta^\alpha_\beta ({\bf v}(\hat {\bf v});z), \quad \alpha,\beta=1,\dots,n.
\eeq
Put
\beq\label{YhatmuR519}
\hat{\mathcal{Y}}_0(\hat {\bf v};z) := \hat \Theta(\hat {\bf v};z) z^\mu z^R \qquad ({\rm clearly,}~\hat{\mathcal{Y}}_0(\hat {\bf v};z) = \mathcal{Y}_0({\bf v}(\hat {\bf v});z)).
\eeq
Using~\eqref{hatcc519}, noticing that 
 $\mathcal{Y}_0 = \Theta({\bf v};z) z^\mu z^R$
is a fundamental matrix solution to the isomonodromic system \eqref{pdedub517}--\eqref{odedub517} for the Frobenius manifold~$M$, and 
noticing    
\beq
\hat{\mathcal{U}}^\alpha_\beta(\hat {\bf v}) 
= \sum_{\rho=1}^n \hat E^\rho(\hat {\bf v}) \hat c^\alpha_{\rho\beta}(\hat {\bf v}) = 
\sum_{\rho,\sigma=1}^n E^\sigma ({\bf v}(\hat {\bf v})) \frac{\p \hat v^\rho}{\p v^\sigma} \hat c^\alpha_{\rho\beta}(\hat {\bf v}) 
=\sum_{\sigma=1}^n E^\sigma ({\bf v}(\hat {\bf v})) c^\alpha_{\sigma\beta}({\bf v}(\hat {\bf v}))  = 
\mathcal{U}^\alpha_\beta({\bf v}(\hat {\bf v})),
\eeq
we find that $\hat{\mathcal{Y}}_0$ satisfies 
\begin{align}
&\frac{\p \hat{\mathcal{Y}}_0}{\p \hat v^\alpha} = z \hat c_\alpha(\hat {\bf v}) \hat{\mathcal{Y}}_0, \quad \alpha=1,\dots,n,\label{hatydv519}\\
& \frac{d \hat{\mathcal{Y}}_0}{dz} = \Bigl(\hat{\mathcal{U}}({\bf v})+\frac{\mu}{z}\Bigr) \hat{\mathcal{Y}}_0. \label{hatydz519}
\end{align}
This is the isomonodromic system for the Frobenius manifold~$S_\kappa(M)$. 
Obviously, $\hat{\mathcal{Y}}_0$ is non-degenerate. 
We conclude that $\hat{\mathcal{Y}}_0$ a fundamental matrix solution to \eqref{hatydv519}--\eqref{hatydz519} and that 
$(\mu, R)$ can be chosen as 
the monodromy data at $z=0$ for the Frobenius manifold~$S_\kappa(M)$.

To show the second statement, we note that from Proposition~\ref{propprefm510-2} we already know that 
the holomorphic functions $\hat \theta_{\alpha,m}(\hat {\bf v})$ form a calibration for 
$S_\kappa(M)$ as a pre-Frobenius manifold. 
Now using~\eqref{hatydz519} and~\eqref{YhatmuR519} we know that 
$\hat \theta_{\alpha,m}(\hat {\bf v})$ also satisfy 
\beq
\hat E \biggl(\frac{\p \hat \theta_{\alpha,m}(\hat {\bf v})} {\p \hat v^\beta} \biggr) = 
(m+\mu_\alpha+\mu_\beta) \frac{\p \hat \theta_{\alpha,m}(\hat {\bf v})} {\p \hat v^\beta} + 
\sum_{\gamma=1}^n \sum_{k=1}^m (R_k)^\gamma_\alpha \frac{\p \hat \theta_{\gamma,m-k}(\hat {\bf v})} {\p \hat v^\beta},
\eeq
where $\alpha=1,\dots,n$, $m\ge0$.
So $\hat \theta_{\alpha,m}(\hat {\bf v})$ form a calibration for~$S_\kappa(M)$ as a Frobenius manifold. 
The theorem is proved.
\end{proof}

We note that 
the statement that $M$ and $S_\kappa(M)$ share the same spectrum $\mu$ was given in~\cite{Du96}.

\subsection{Semisimplicity and monodromy data at $z=\infty$}
In this subsection, we assume that the Frobenius manifold $(M,\, \cdot,\, e,\, \langle \,,\, \rangle,\, E)$ is semisimple, 
and study the monodromy data of $S_\kappa(M)$ at $z=\infty$.

Let $(B, \phi)$ be a flat coordinate chart with $\phi(B)$ being a sufficiently small ball and with the flat coordinates taken as before, such 
that $B\subset M^0$ (so we can also take canonical coordinates $u_1,\dots,u_n$) 
and that $\frac{\p}{\p v^\kappa}\cdot$ is invertible. 
Since $M$ and its Legendre-type transformation $S_\kappa(M)$ have the same multiplication, 
the canonical coordinates $(u_1,\dots,u_n)$ also form the canonical coordinates 
for $S_\kappa(M)$. 
Using 
\beq
e= \frac{\p}{\p v^\iota} = \sum_{i=1}^n \psi_{i\iota}({\bf u}) \frac1{\sqrt{\eta_{ii}({\bf u})}} \frac{\p}{\p u_i} 
\eeq
and using~\eqref{ecanonical227}, we have
\beq
\psi_{i\iota}({\bf u}) = \sqrt{\eta_{ii}({\bf u})}, \quad i=1,\dots,n.
\eeq

Denote 
\beq
\hat \eta_{ij}({\bf u}) = \Bigl\langle \frac{\p}{\p u_i}, \frac{\p}{\p u_j} \Bigr\rangle^\wedge, \quad i,j=1,\dots,n.
\eeq
By definition we know that  
\beq
\hat \eta_{ii}({\bf u}) = \Bigl\langle \frac{\p}{\p v^\kappa} \cdot \frac{\p}{\p u_i}, \frac{\p}{\p v^\kappa} \cdot \frac{\p}{\p u_i} \Bigr\rangle 
= \psi_{i\kappa}({\bf u})^2, \quad i=1,\dots, n. 
\eeq
Let $\hat f_1,\dots,\hat f_n$ be the orthonormal frame with respect to $\langle \,, \rangle^\wedge$, given by 
\beq\label{hatfinormalize} 
\hat f_i = \frac1{\psi_{i\kappa}({\bf u})} \frac{\p}{\p u_i}, \quad i=1,\dots,n.
\eeq
Here we note that like in~\eqref{finormalize} one can again arbitrarily choose the signs for the square roots of $\hat \eta_{ii}({\bf u})$, 
but in~\eqref{hatfinormalize} we actually make the particular choice. This means that, once a choice of 
 signs for $\sqrt{\eta_{ii}({\bf u})}$ on~$M$ is made as we have done (cf.~\eqref{finormalize}), there is an according choice 
 of signs for $\sqrt{\hat \eta_{ii}({\bf u})}$ on~$S_{\kappa}(M)$.
 
Let 
$\hat \Psi({\bf u})=(\hat \psi_{i\alpha}({\bf u}))$ be the 
transition matrix from the orthonormal frame $(\hat f_1,\dots, \hat f_n)$ to the frame $(\frac{\p}{\p \hat v^1},\dots,\frac{\p}{\p \hat v^n})$, i.e.,
\beq\label{phatvalpha}
\frac{\p}{\p \hat v^\alpha} = \sum_{i=1}^n \hat f_i  \hat \psi_{i\alpha}({\bf u}) = 
\sum_{i=1}^n  \frac{\hat \psi_{i\alpha}({\bf u})}{\psi_{i\kappa}({\bf u})} \frac{\p}{\p u_i}, \quad \alpha=1,\dots,n.
\eeq
Using \eqref{ee513}, \eqref{ecanonical227} and~\eqref{phatvalpha}, we have  
\beq
\hat \psi_{i\kappa}({\bf u}) = \psi_{i\kappa}({\bf u}), \quad i=1,\dots,n.
\eeq

The following lemma is important. 
\begin{lemma}\label{psipsihatlemma} We have
\beq
\Psi({\bf u})  = \hat \Psi({\bf u}) .
\eeq
\end{lemma}
\begin{proof}
We have 
\begin{align}
& \hat \psi_{i\alpha}({\bf u}) = \Bigl \langle \frac{\p}{\p \hat v^\alpha} , \hat f_i \Bigr\rangle^\wedge 
= \Bigl\langle \frac{\p}{\p v^\kappa} \cdot \frac{\p}{\p \hat v^\alpha} , \frac{\p}{\p v^\kappa} \cdot \hat f_i \Bigr\rangle 
 = \Bigl\langle \frac{\p }{\p v^\alpha} , \frac{1}{\psi_{i\iota}({\bf u})}\frac{\p}{\p u_i} \Bigr\rangle = \psi_{i\alpha}({\bf u}),
\end{align}
where for the third equality we have used Lemma~\ref{anotherusefullemma}. 
\end{proof}

It follows from \eqref{defUV0103}, Theorem~\ref{thm2519} and Lemma~\ref{psipsihatlemma} that 
\beq\label{VVequal314}
\hat V({\bf u}) = V({\bf u}). 
\eeq

\begin{theorem}\label{maintheorem}
We have the following identity:
\beq\label{YMSkMequal}
Y_{\rm formal} ({\bf u};z) = \hat Y_{\rm formal} ({\bf u};z). 
\eeq
Moreover, choosing the admissible line (cf.~\eqref{admissible32}) for $S_{\kappa}(M)$ 
the same as that for~$M$, the Stokes matrix and the central connection matrix of 
$S_{\kappa}(M)$  coincide with those of~$M$.
\end{theorem}
\begin{proof}
By~\eqref{VVequal314}, we know that both $Y_{\rm formal} ({\bf u};z)$ and $\hat Y_{\rm formal} ({\bf u};z)$ 
satisfy the same equations~\eqref{Yeq10102}--\eqref{Yeq20102}. By the uniqueness of the formal solution to 
\eqref{Yeq10102}--\eqref{Yeq20102} we find~\eqref{YMSkMequal} holds. The second statement is then proved 
by observing that $\hat{\mathcal{Y}}(\hat {\bf v};z)= \mathcal{Y}({\bf v}(\hat {\bf v});z)$. 
\end{proof}

\begin{remark}
The validity of Theorem~\ref{thm2519} does not require the semisimplicity.
Theorem~\ref{maintheorem} and Theorem~\ref{thm2519} imply that, for a semisimple Frobenius manifold~$M$, 
under the assumptions that $\frac{\p}{\p v^\kappa}\cdot$ is invertible, the semisimple Frobenius manifold 
$S_\kappa(M)$ has the same monodromy data $\mu, R, S, C$ as~$M$. Although not completely obvious, this result 
can also be deduced from \cite[Theorems 4.4--4.6]{Du99} (cf.~\cite{Du96}) obtained by Dubrovin. Indeed, Dubrovin gave a reconstruction 
of the semisimple Frobenius manifold from the monodromy data $\mu, R, S, C$ by solving a Riemann--Hilbert boundary value problem, 
where a particular column vector is used~\cite{Du99} in the reconstruction of the Frobenius structure  (locally) of~$M$ 
which corresponds to a specific eigenvector of the matrix~$V$ subjected to the eigenvalue~$\mu_\iota$ ($\mu_1$ in Dubrovin's notation). Recall from 
the end of Section~\ref{section2} that the monodromy data of~$M$ includes a selection of an eigenvector of~$V$. We use~$e$ to record 
such a selection because its components in the orthonormal frame $f_1,\dots,f_n$ equal to those of the eigenvector, although a marking  
is already sufficient for recording such information. So, roughly speaking, 
at least in the situation when the eigenvalues $\mu_1$, \dots, $\mu_n$ are distinct, 
the cluster of Frobenius manifolds $S_\kappa(M)$ simply correspond to 
the selections of eigenvectors (eigenvalues) of~$V$. They share the same monodromy data $\mu, R, S, C$ in suitable chambers.
\end{remark}

An identity equivalent to~\eqref{YMSkMequal} is given by 
\beq
\Phi ({\bf u};z) = \hat \Phi({\bf u};z). 
\eeq

\section{Application I. Computing monodromy data using Theorems \ref{thm2519}, \ref{maintheorem}} \label{sectionapp1}
In the previous section, we have shown in Theorem~\ref{maintheorem} that, 
for a semisimple Frobenius manifold~$M$,
the cluster of semisimple Frobenius manifolds  
$S_\kappa(M)$, $\kappa\in I$, share the same monodromy data $\mu, R, S, C$. 
So, if we know the monodromy data 
of~$M$ we immediately get those for~$S_\kappa(M)$. 
In this section, we will 
consider several classes of examples 
of Frobenius manifolds whose monodromy data are known or conjecturally known.

\subsection{Quantum cohomology}
Let $X$ be a $D$-dimensional smooth projective variety with vanishing odd cohomology, and 
 $X_{g,k,\beta}$ the moduli space of 
stable maps of degree $\beta\in H_2(X,\mathbb{Z})/{\rm torsion}$ 
with target~$X$ from curves of genus~$g$ with~$k$ distinct marked points. 
Here, $g,k\ge0$. We denote the Poincar\'e pairing on 
$H^{*}(X;\mathbb{C})$ by $\langle\,,\,\rangle$.
Choose a homogeneous basis $\phi_1 = 1,\phi_2,\dots,\phi_n$ of 
the cohomology ring $H^{*}(X;\mathbb{C})$ 
such that $\phi_\alpha \in H^{2q_\alpha}(X;\CC)$, $\alpha=1,\dots,n$, 
 $0=q_1<q_2 \leq \dots \leq q_{n-1} < q_n = d$, and $\langle \phi_\alpha,\phi_\beta\rangle = \delta_{\alpha+\beta,n+1}$.
The integrals
\beq
\int_{[X_{g,k,\beta}]^{\rm virt}}  {\rm ev}_1^*(\phi_{\alpha_1}) \cdots {\rm ev}_k^*(\phi_{\alpha_k}), \quad 
\alpha_1,\dots,\alpha_k=1,\dots,n,
\eeq 
are called {\it primary Gromov--Witten (GW) invariants of~$X$ of genus~$g$ and degree~$\beta$ \cite{BF, Du96, KM94, LT, Manin, MS, RT}}. 
Here, ${\rm ev}_a$, $a=1,\dots,k$, are the evaluation maps 
and $[X_{g,n,\beta}]^{\rm virt}$ is the {\it virtual fundamental class}, 
which is an element in the Chow ring having the complex dimension 
\beq\label{vdim}
(1-g)(D-3) + k + \langle \beta, c_1(X)\rangle. 
\eeq 
The {\it genus~$g$ primary free energy $\F_g({\bf v};Q)$} for the GW invariants of~$X$ is defined by 
\beq\label{defF}
F_g({\bf v};Q) = \sum_{k\ge0}\sum_{\alpha_1,\dots,\alpha_k \ge0} \sum_{\beta}\frac{Q^\beta v^{\alpha_1} \cdots v^{\alpha_k} }{k!} 
\int_{[X_{g,k,\beta}]^{\rm virt}} {\rm ev}_1^*(\phi_{\alpha_1}) \cdots {\rm ev}_k^*(\phi_{\alpha_k}),
\eeq
where ${\bf v}=(v^1,\dots,v^n)$ is a vector of indeterminates, and 
\beq
Q^\beta=Q_1^{m_1} \cdots Q_r^{m_r} ~ ({\rm for}~\beta=m_1\beta_1+\dots+m_r\beta_r)
\eeq
is an element of the Novikov ring. Here, $(\beta_1,\dots,\beta_r)$ is a basis of $H_2(X;\mathbb{Z})/{\rm torsion}$.

We will restrict our consideration to the case that $X$ is Fano. Then for any fixed $k\ge0$ and $\alpha_1,\dots,\alpha_k \in \{1,\dots,n\}$, 
the sum $\sum_\beta$ in~\eqref{defF} is a finite sum. Moreover, by the divisor axiom we know that 
$F_0({\bf v};Q)$ is a power series of $v^1, Q_1 e^{v^2}, \dots, Q_r e^{v^{r+1}}, v^{r+2}, \dots, v^n$ (see e.g.~\cite{Du96, Manin}). 
So, without loss of information, we can 
take $Q_1=\cdots=Q_r=1$ and denote $F_0({\bf v}; Q)|_{Q_1=\cdots=Q_r=1}=:F_0({\bf v})$. 
The power series $F_0({\bf v}; Q)$ or simply $F_0({\bf v})$ leads to 
an $n$-dimensional {\it formal} Frobenius manifold  
of charge~$D$, often denoted $M=QH(X)$, 
with the invariant flat metric $\langle \,,\, \rangle$ given by $\langle \frac{\p }{\p v^\alpha}, \frac{\p}{\p v^\beta}\rangle = \delta_{\alpha+\beta,n+1}$, 
the unity vector field $e=\frac{\p}{\p v^1}$ and the Euler vector field~$E$ given by 
\beq
E= \sum_{\alpha=1}^n \bigl( (1-q_\alpha) v^\alpha + \langle c_1(X), \phi^\alpha \rangle \bigr) \frac{\p}{\p v^\alpha}.
\eeq
In particular, $F_0({\bf v})$ satisfies the WDVV equation~\eqref{wdvv64}. 

Let us assume that the power series $F_0({\bf v})$ has a convergence domain, so that $M$ can be viewed as a Frobenius manifold 
(namely, being {\it analytic}). 
Then from Theorem~\ref{maintheorem} we know that the monodromy data of the Legendre-type transformations 
$S_\kappa(M)$, $\kappa=1,\dots,n$, 
(whenever $\frac{\p}{\p v^\kappa}\cdot$ is invertible) is equal to that of~$M$. Particularly, since 
the constant matrix $R$ can be chosen to correspond to the cup product by $c_1(X)$, we immediately get $\hat R$ for $S_\kappa(M)$.

\begin{example}[$\mathbb{P}^1$ Frobenius manifold] \label{exampleA11} $X=\mathbb{P}^1$ and $M=QH(\mathbb{P}^1)$. This Frobenius manifold is of charge $D=1$.
The potential of~$M$, i.e., 
the genus zero primary free energy of~$X$, has the explicit expression
\beq\label{p1potentialf}
F({\bf v})=\frac12 (v^1)^2 v^2 + e^{v^2}.
\eeq
The Euler vector field on~$M$ is 
\beq
E=v^1\frac{\p}{\p v^1} + 2 \frac{\p}{\p v^2}.
\eeq
Under the calibration given in~\cite{Du96, Du98, Du99, DZ-norm}, we have 
\beq\label{monop164}
\eta=\begin{pmatrix} 0 & 1 \\ 1 & 0\end{pmatrix}, \quad \mu=\begin{pmatrix} -\frac12 & 0 \\ 0 & \frac12\end{pmatrix}, 
\quad R=\begin{pmatrix} 0 & 0 \\ 2 & 0\end{pmatrix}, \quad 
S = \begin{pmatrix} 1 & 2 \\ 0 & 1\end{pmatrix}.
\eeq

The Legendre-type transformation of this Frobenius manifold was given in details in~\cite{Du96}. 

The Frobenius manifold $S_1(M)$ is $M$ itself. 

Let $\kappa=2$. 
The Frobenius manifold $S_2(M)$ is of charge $\hat D=-1$ and has the potential
\beq\label{p1fs2}
F^{S_2(M)} = \frac12 (\hat v^2)^2 \hat v^1 + \frac12 (\hat v^1)^2 \log \hat v^1 - \frac34 (\hat v^1)^2.
\eeq
This Frobenius manifold is known as the NLS Frobenius manifold as mentioned in the Introduction. It has the Euler vector field
\beq
E= 2 \hat v^1\frac{\p}{\p \hat v^1} + \hat v^2 \frac{\p}{\p \hat v^2}.
\eeq
Using Theorem~\ref{maintheorem}, we immediately obtain the following theorem which was proved in~\cite{CvdLPS} in a slightly different form.
\begin{theorem} 
\label{thmCvdLPS}
The monodromy data $\mu, R, S, C$ for the Frobenius manifold $S_2(M)$ with the potential~\eqref{p1fs2} 
can have the same expression as those given in~\eqref{monop164}.
\end{theorem}
\begin{remark}
At a first look, one might find differences between the result of~\cite{CvdLPS} and Theorem~\ref{thmCvdLPS}, but 
actually the two agree when the indices for the flat coordinates of~$S_2(M)$ in~\cite{CvdLPS} are relabelled like what we do. 
In~\cite{CvdLPS} a direct computation is given; it is clear that 
our method is simpler.
\end{remark}
\end{example}


We have mentioned that the constant matrix $R$ can be chosen to correspond to the cup product by $c_1(X)$.
Actually, not only the matrices $\eta$, $\mu$, $R$ are related to geometry, but also the Stokes matrix~$S$ and the central connection matrix~$C$ 
are (conjecturally) related to geometry. The latter is the content of Dubrovin's Conjecture~\cite{Du98} (cf.~\cite{Zaslow} on a related conjecture), refined by 
 the Gamma Conjecture II~\cite{CDG18, GGI}. In his ICM talk, Dubrovin proposed the following conjecture.

\smallskip

\noindent {\bf Dubrovin's Conjecture}~(\cite{Du98}).
{\it  Assume that $X$ is Fano and that the power series $F_0({\bf v})$
has a convergence domain~$D$. Then 
\begin{itemize}
\item[1.] The quantum cohomology $QH(X)$ is semisimple if and only if 
the bounded derived category $D^b(X)$ admits a full exceptional collection $E_1,\dots,E_n$.
\item[2.] When $QH(X)$ is semisimple, the Stokes matrix $S$ of~$QH(X)$ is equal to the inverse of the Euler matrix of 
some full exceptional collection of~$D^b(X)$, i.e.,
$$
(S^{-1})_{ij} = \chi(E_i,E_j), \quad i,j=1,\dots,n.
$$
\item[3.] When $QH(X)$ is semisimple, the central connection matrix $C$ of~$QH(X)$ has a geometrically meaningful 
decomposition:
$$
C = C' C'',
$$
where $C'$, $C''$ are two matrices satisfying that the columns of $C''$ are components of the Chern character ${\rm ch}(E_j)\in H^*(X)$, 
and $C': H^*(X) \to H^*(X)$ is some operator satisfying $C'(c_1(X)a) = c_1(X) C'(a)$ for any $a \in H^*(X)$.
\end{itemize}}

Note that in the original statements of Dubrovin's Conjecture~\cite{Du98}, the Stokes matrix~$S$ is conjectured to be equal to the 
Euler matrix, however, in our notation the Stokes matrix~$S$ given in~\eqref{defiStokes} is equal to the inverse of the one in~\cite{Du98}. So, 
the above statements are the same as those in~\cite{Du98}. 
Parts 1,2 of Dubrovin's Conjecture were verified for a few cases, including, e.g., 
the Grassmannians~\cite{CDG18, Du98, Gu99, UEDA}.
Part~3 of Dubrovin's conjecture was refined by Cotti--Dubrovin--Guzzetti~\cite{CDG18}, namely, $C'$ corresponds to 
$
\frac{i^{\bar{D}}}{(2\pi)^{D/2}} \widehat \Gamma_X^- \cup e^{-2\pi i c_1(X)} \cup $,
where $\bar{D}$ is the residue of $D$ modulo~2, and $\Gamma_X^-:= \prod_{j=1}^D \Gamma(1-\delta_j)$ is the Gamma 
class 
with $\delta_{j}$ being Chern roots of the tangent bundle $TX$. 
This refinement is equivalent~\cite{CDG18} 
to the Gamma Conjecture~II proposed by Galkin--Golyshev--Iritani \cite{GGI}. 

\begin{remark}
In several occasions, the assumption that $F_0({\bf v})$ has a domain of convergence is indeed true. 
For example, it was suggested by Dubrovin~\cite{Du98} and 
proved by Cotti~\cite{Cotti} that, when the small quantum cohomology of~$M$ is 
semisimple, the power series $F_0({\bf v})$ has a non-empty domain of convergence. See~\cite{CI, Iritani} for 
some other occasions. 
\end{remark}

\begin{example}\label{exampleA21} The Chen--Ruan cohomology of the $\mathbb{P}^1$-orbifold $\mathbb{P}^1_{1,2}$~\cite{CR02, CR04, MT, Rossi}.
In this case, we have a three-dimensional Frobenius manifold $M$ of charge $D=1$ 
with potential:
\beq
F= \frac12 (v^1)^2 v^3 + \frac12 v^1 (v^2)^2 - \frac1{24} (v^2)^4 + v^2 e^{v^3}.
\eeq
This Frobenius manifold was found earlier in~\cite{Du96, DZ98}.
The Euler vector field is given by
\beq
E=v^1\frac{\p}{\p v^1} + \frac12 v^2 \frac{\p}{\p v^2} + \frac32 \frac{\p}{\p v^3}.
\eeq
For the knowledge about the Stokes matrix of this Frobenius manifold $M$ see e.g.~\cite{Gu01, IT}.

The Legendre-type transformations for this Frobenius manifold were explicitly known~\cite{Du96, LZZ, SS}.
Let $\kappa=2$. The Frobenius manifold $S_2(M)$ is of charge $\hat D=0$ and has the potential:
\beq\label{A2fs2}
F^{S_2(M)} = \frac{(\hat v^2)^3}{6}+\hat v^1 \hat v^2 \hat v^3 
+ \frac{1}{6} \hat v^1 (\hat v^3)^3+\frac{1}{2} (\hat v^1)^2 \log \hat v^1-\frac{3}{4}(\hat v^1)^2.
\eeq
The Euler vector field reads
\beq
E= \frac32 \hat v^1 \frac{\p}{\p \hat v^1} + \hat v^2 \frac{\p}{\p \hat v^2} + \frac12 \hat v^3 \frac{\p}{\p \hat v^3} .
\eeq
This Frobenius manifold corresponds to the the dispersionless limit 
of the (extended) constrained KP hierarchy~\cite{AK, CvdLPS2, Cheng, FYZ, LZZ}.

Let $\kappa=3$. The Frobenius manifold $S_3(M)$ is of charge $\hat D=-1$ and has the potential:
\beq\label{A2fs3}
F^{S_3(M)} = \frac{1}{2} (\hat v^3)^2 (\hat v^1)+\frac{1}{2} (\hat v^2)^2 (\hat v^3)+\frac{1}{2} (\hat v^1)^2 \log \hat v^2.
\eeq
The Euler vector field reads
\beq
E= 2 \hat v^1 \frac{\p}{\p \hat v^1} + \frac32 \hat v^2 \frac{\p}{\p \hat v^2} + \hat v^3 \frac{\p}{\p \hat v^3} .
\eeq
\end{example}

The cases of non-Fano smooth varieties and of equivariant GW theory shall be addressed elsewhere. 

\subsection{Legendre-type transformations of a tensor product}\label{sectionapp2}

The tensor product of quantum cohomologies was introduced by Kontsevich--Manin--Kaufmann~\cite{KMK} towards  
calculating the quantum cohomology of the direct product of two varieties. It was generalized to 
Frobenius manifolds in~\cite{Du99}.  In this section we study the Legendre-type transformations 
of the tensor product of two semisimple Frobenius manifolds.

Let $(M', \cdot', e', \langle,\rangle', E')$, $(M'', \cdot'', e'', \langle,\rangle'', E'')$ be two semisimple Frobenius manifolds of 
respective dimensions $n'$ and $n''$
and respective charges $D'$ and $D''$. Assume that we can take their 
flat coordinates $v^{\alpha'}$, $v^{\alpha''}$, $\alpha'=1',\dots,n'$, $\alpha''=1'',\dots,n''$, such that 
$e'=\frac{\p}{\p v^{\iota'}}$, $e''=\frac{\p}{\p v^{\iota''}}$ and 
\begin{align}
& E'=\sum_{\beta'=1}^{n'} E^{\beta'}({\bf v'}) \frac{\p}{\p v^{\beta'}}, \quad E^{\beta'}({\bf v'}) 
= \Bigl(1-\frac{D'}2-\mu_{\beta'}\Bigr) v^{\beta'} + r^{\beta'}, \label{E'528} \\
& E''=\sum_{\beta=1}^n E^\beta({\bf v''}) \frac{\p}{\p v^{\beta''}}, \quad E^{\beta''}({\bf v''}) 
= \Bigl(1-\frac{D''}2-\mu_{\beta''}\Bigr) v^{\beta''} + r^{\beta''}.\label{E''528} 
\end{align}
Recall from~\cite{Du99} that the tensor product $M=M'\otimes M''$ is a semisimple Frobenius 
manifold of dimension $n'n''$ with the Frobenius structure given by 
\begin{itemize}
\item[1.] The invariant flat metric $\langle\,,\rangle:=\langle\,,\rangle'\otimes \langle\,,\rangle''$, 
and the unity vector filed $e := e' \otimes e''$. The flat coordinates on~$M$ have double labels: $v^{(\alpha',\alpha'')}$,  $\alpha'=1,\dots,n'$, $\alpha''=1,\dots,n''$, and $e=\frac{\p}{\p v^{(\iota',\iota'')}}$. The gram matrix for $\langle\,,\rangle$ reads
\beq
\eta^{(\alpha',\alpha'') (\beta',\beta'')} = \eta^{\alpha'\beta'} \eta^{\alpha''\beta''},\quad \alpha',\beta'=1',\dots,n',~\alpha'',\beta''=1'',\dots,n''.
\eeq
\item[2.] At the points $p$ with 
\beq
v^{(\alpha',\alpha'')} = 0, \quad {\rm for}~\alpha'\neq \iota' \; {\rm and} \;  \alpha''\neq \iota'',
\eeq
the algebra $T_p M$ is given by the tensor product 
\beq
T_p M = T_{p'} M' \otimes T_{p''} M'',
\eeq
where
\beq
p'=(v^{(1',1'')}, \cdots, v^{(n',1'')}),  \quad 
p''=(v^{(1',1'')}, \cdots, v^{(1',n'')}).
\eeq
Namely,
\beq
c_{(\alpha',\alpha'')(\beta',\beta'')}^{(\gamma',\gamma'')}\big|_p = c_{\alpha'\beta'}^{\gamma'}\big|_{p'} c_{\alpha''\beta''}^{\gamma''}\big|_{p''},\quad 
\alpha',\beta',\gamma'=1',\dots,n',~\alpha'',\beta'',\gamma''=1'',\dots,n''.
\eeq
\item[3.] The charge $D=D'+D''$ and the Euler vector field is given by 
\beq
E= \sum_{1\le\alpha'\le n'\atop 1\le\alpha''\le n''} \Bigl(1-\frac{D}2-\mu_{\alpha'}-\mu_{\alpha''}\Bigr) v^{(\alpha',\alpha'')}\frac{\p }{\p v^{(\alpha',\alpha'')}}
 + \sum_{\alpha'=1}^{n'} r^{\alpha'} \frac{\p}{\p v^{(\alpha',\iota'')}} + \sum_{\alpha''=1}^{n''} r^{\alpha''} \frac{\p}{\p v^{(\iota',\alpha'')}}.
\eeq
\end{itemize}

We have the following theorem. 
\begin{theorem}
Let $\kappa'\in\{1,\dots,n'\}$ and $\kappa''\in\{1,\dots,n''\}$. Assume that 
$\frac{\p}{\p v^{\kappa'}}\cdot$ and $\frac{\p}{\p v^{\kappa''}}\cdot$ are invertible in $B'$ and $B''$, respectively, then we have 
the Frobenius manifold isomorphism 
\beq\label{Stensorformula}
S_{(\kappa',\kappa'')}(M' \otimes M'') = S_{\kappa'}(M') \otimes S_{\kappa''}(M'') .
\eeq
\end{theorem}
\begin{proof}
Denote by $\mu', R', S', C', \iota'$ and $\mu'', R'', S'', C'', \iota''$ the monodromy data of $M'$ and $M''$, respectively. 
Using Theorems \ref{thm2519}, \ref{maintheorem} we know that 
$\mu', R', S', C', \kappa'$ and $\mu'', R'', S'', C'', \kappa''$ are the monodromy data for $S_{\kappa'}(M')$ and respectively $S_{\kappa''}(M'')$. 
Using \cite[Theorem~4.10]{Du99}, we know that the monodromy data of $M' \otimes M''$ is 
given by 
\beq\label{monotensor}
\mu' \otimes I + I \otimes \mu'', \quad R' \otimes I + I \otimes R'', \quad 
S'\otimes S'', \quad C'\otimes C'', \quad (\iota', \iota'').
\eeq
Similarly, the monodromy data of $S_{\kappa'}(M') \otimes S_{\kappa''}(M'')$ is 
given by 
\beq\label{monotensorkaka}
\mu' \otimes I + I \otimes \mu'', \quad R' \otimes I + I \otimes R'', \quad 
S'\otimes S'', \quad C'\otimes C'', \quad (\kappa', \kappa'').
\eeq
Using Theorems \ref{thm2519}--\ref{maintheorem}, we know that \eqref{monotensorkaka} is also the monodromy data for 
$S_{(\kappa',\kappa'')}(M' \otimes M'')$.
The theorem is proved by using~\cite[Theorems~4.6, 4.10]{Du99}.
\end{proof}
\begin{remark}
In the above theorem, 
the isomorphism~\eqref{Stensorformula} 
is understood locally (i.e., subjected to suitable chosen admissible lines); we address the global meaning for~\eqref{Stensorformula} elsewhere.
\end{remark}

\subsection{Other examples}
In this subsection, we look at a few more examples. 

\begin{example} \label{examplea2}
Simple singularity of type $A_2$. In this case, we have a two-dimensional Frobenius manifold $M$ with potential 
\beq
F({\bf v}) = \frac12(v^1)^2 v^2 + \frac1{72} (v^2)^4.
\eeq
The matrix $\eta$ is given by $\eta_{\alpha\beta}=\delta_{\alpha+\beta,3}$, and the matrix 
$\mu={\rm diag} (-1/6, 1/6)$.
The differential system~\eqref{odedub517} at the special semisimple point $(v^1=0, v^2=3)$ read 
\begin{equation*}
\left\{
\begin{array}{c}
\frac{dy^1}{dz} = 2 y^2 - \frac16 \, \frac{y^1}{z}, \\
\\
\frac{dy^2}{dz} = 2 y^1 + \frac16 \, \frac{y^2}{z}. \\
\end{array}
\right.
\end{equation*}
The canonical coordinates at $(v^1=0, v^2=3)$ are given by $(u_1=-2,u_2=2)$. 
Choosing the admissible line $\ell$ such that $\ell_+=\{z \,|\, {\rm arg} \, z=3\pi/4\}$,  one finds by a direct calculation that 
\beq
R=0, \quad 
S = \begin{pmatrix} 1 & 0 \\  -1 & 1 \end{pmatrix}, \quad C=\frac{-i}{\sqrt{2\pi}} \begin{pmatrix} \Gamma(\frac23) & \Gamma(\frac23)\theta^{5/2} \\ 
& \\  \Gamma(\frac13) \theta^{3/2} & \Gamma(\frac13)\theta^2 \end{pmatrix},
\eeq
where $\theta=\exp(2\pi i/3)$ is the third root of unity. More details will be given in a forthcoming joint work with Zhengfei Huang. 
The Frobenius manifold $S_2(M)$ has the potential 
\beq
F^{S_2(M)} ({\bf \hat v}) = \frac{ (\hat v^2)^2 \hat v^1}{2}+\frac{4}{5} \sqrt{\frac{2}{3}} \, (\hat v^1)^{\frac52} .
\eeq
\end{example}

The reason, that the above Example~\ref{examplea2}  
 is referred to as the simple singularity of type~$A_2$, is because the Frobenius manifold structure in 
 this example can be defined on the miniversal deformation of the holomorphic function 
 $f(x)=x^3$~\cite{Du96, Hertling, Saito81, Saito83, W93}. Note 
 that $0$ is an isolated critical point of non-Morse type of~$f(x)$. This critical point is of  
 $A_2$-type \cite{Arnold}, as the Milnor lattice is isomorphic to the root lattice of $A_2$-type. 
More generally, for any simple singularity of $X_n$-type, $X=A$, $D$, or $E$, there is a Frobenius  
structure on its miniversal deformation~\cite{Du96, Du98-2, Saito81, Saito83, Sai89}. 
The Milnor lattice is isomorphic to the root lattice of $X_n$-type. 
Then, from  
the period mapping~\cite{AGV, Loo, Saito83} and the Gauss--Manin connection  
of the Frobenius manifold~\cite{Du99}, we know that 
the matrices $\mu$, $R$, $S$ can be given by  
\beq
\mu={\rm diag}(m_\alpha/h-1/2)_{\alpha=1,\dots,n}, \quad R=0, \quad S+S^T={\rm Cartan},
\eeq
where $m_\alpha$, $\alpha=1,\dots,n$, are the exponents of the $X_n$-lattice and $h$ is the Coxeter number. 

Alternatively, the Frobenius structure in Example~\ref{examplea2} can be given on the 
orbit space of the Coxeter group of $A_2$-type by an explicit construction of the flat pencil \cite{Du96} (see also~\cite{Saito93}).
In general, the flat pencil construction is given for any irreducible finite Coxeter group of $X_n$-type, $X=A,B,C,D,E,F$, or $G$, 
leading to an $n$-dimensional Frobenius manifold with a polynomial potential~\cite{Du96, Zuber} (see also~\cite{Saito93}). 
When $X=A, D$, or $E$, this Frobenius manifold is isomorphic to the one given by the miniversal deformation of 
a simple singularity of $X_n$-type. Such an equivalence can be explained for example by the Brieskorn correspondence
on the relationship between the discriminant of the miniversal deformation 
and that of the orbit space. Frobenius structures are also associated to the extended affine Weyl groups~\cite{Du96, DSZZ, DZ98} 
with the choice of vertices, 
to the CCC (Complex Crystallographic Coxeter) groups~\cite{Be1, Be2, Du96} and to the extended Jacobi groups~\cite{Almeida21, Almeida22, Du09}. 
For all these Frobenius manifolds, one should in principle be 
able to read the monodromy data from the algebraic structure 
(details of this general consideration will be discussed elsewhere). 
We remark that the Frobenius structure in Example~\ref{exampleA11} 
coincide with the one constructed~\cite{Du96, DZ98} on the orbit space of the extended affine Weyl group of $A_1^{(1)}$-type, 
and the Frobenius structure in 
Example~\ref{exampleA21} coincide with the one  
constructed~\cite{Du96, DZ98} on the orbit space of the extended affine Weyl group of $A_2^{(1)}$-type.

\section{Review on genus zero and higher genus theories}\label{sectionappa}
In this section, we will review the genus zero and higher genus theories associated to a Frobenius manifold.

\subsection{Tau-functions for the principal hierarchy}\label{setupgenuszero}
Let~$M$ be a calibrated $n$-dimensional pre-Frobenius manifold with 
flat coordiantes $v^1,\dots,v^n$ being taken as in previous sections. 

Since the flows in the principal hierarchy~\eqref{principalh520} pairwise commute~\cite{Du92, Du96},  
we can solve the 
PDEs in~\eqref{principalh520} together, yielding solutions of the form
${\bf v}={\bf v}({\bf t})$, where ${\bf t}=(t^{\beta,m})_{\beta=1,\dots,n,\, m\ge0}$. 
To be precise, let ${\bf v}({\bf t})$ be a solution to the principal hierarchy~\eqref{principalh520} in
\beq\label{solutionv1520}
\bigl(\CC[[t^{1,0}-t^1_o,\dots,t^{n,0}-t^n_o]][[{\bf t}_{>0}]]\bigr)^n, 
\eeq
where $t^\beta_o$, $\beta =1,\dots,n$, are constants which we call the {\it primary initials}, 
and ${\bf t}_{>0}$ denotes the infinite vector $(t^{\beta,m})_{\beta=1,\dots,n,\, m\ge1}$. We call $t^{1,0},\dots,t^{n,0}$ the {\it primary times}, 
and call $t^{\beta,m}$, $\beta=1,\dots,n$, $m\ge1$, the {\it descendent times}.
We also require that 
\beq\label{solutionv2520}
v^\alpha(t^1_o,\dots,t^n_o,0,0,\dots,0) = v^\alpha_o. 
\eeq
Solutions in~\eqref{solutionv1520} satisfying~\eqref{solutionv2520} can be uniquely 
characterized by their initial data  
\beq
f^\alpha(x)=v^\alpha({\bf t})|_{t^{\beta,0}=\delta^{\beta}_{\iota} (x-t^\iota_o) + t^\beta_o, \, {\bf t}_{>0}={\bf 0}, \, \beta=1,\dots,n}, \quad \alpha=1,\dots,n.
\eeq
Clearly, $f^\alpha(x)$ satisfies that $(f^1(t^\iota_o), \dots, f^n(t^\iota_o))\in \phi(B)$. 
It should be noted that for the flat chart $(B, \phi)$, the primary initials $t^\beta_o$, $\beta =1,\dots,n$,
are not necessarily to be fixed constants. 

\smallskip

\noindent {\bf Definition} (\cite{DZ-norm}).
The solution ${\bf v}({\bf t})$ is called {\it monotone} if the vector 
$$\sum_{\alpha=1}^n \frac{\p v^\alpha({\bf t})}{\p x}\bigg|_{{\bf t}=(t^1_o,\dots,t^n_o,0,0,\dots,0)}\frac{\p}{\p v^\alpha}$$ is an invertible element 
of the Frobenius algebra $T_pM$.

\smallskip

Depending on the classes of functions the monotonicity would have different definitions; the above one is suitable for power-series solutions that 
are under consideration.

For a solution ${\bf v}={\bf v}({\bf t})=(v^1({\bf t}),\dots, v^n({\bf t}))$ in~\eqref{solutionv1520}
to the principal hierarchy 
satisfying~\eqref{solutionv2520}, 
there exists a function $\tau^{M, [0]}({\bf t})\in\CC[[t^{1,0}-t^1_o,\dots,t^{n,0}-t^n_o]][[{\bf t}_{>0}]]$, such that 
\beq\label{definingzerofree423}
\frac{\p^2 \log \tau^{M, [0]}({\bf t})}{\p t^{\alpha,m_1} \p t^{\beta,m_2}} = \Omega_{\alpha,m_1; \beta,m_2}^{M, [0]}({\bf v}({\bf t})), \quad 
\alpha,\beta=1,\dots,n,\, m_1,m_2\ge0.
\eeq
We call $\tau^{M, [0]}({\bf t})$ the {\it tau-function of the solution~${\bf v}({\bf t})$ to the principal hierarchy~\eqref{principalh520}}, 
and call its logarithm, denoted $\F^{M}_0({\bf t})$, the {\it genus zero free energy of the solution~${\bf v}({\bf t})$ 
to the principal hierarchy~\eqref{principalh520}}, i.e., 
$\log \tau^{M, [0]}({\bf t})=\F^{M}_0({\bf t})$. Note that for a fixed solution~${\bf v}({\bf t})$, its genus zero free energy is 
uniquely determined up to the addition of an affine-linear function of~${\bt}$.

We call the map 
\beq
(t^{1,0},\dots,t^{n,0}) \mapsto \bigl(V^1,\dots,V^n\bigr) 
\eeq
defined by putting ${\bf t}_{>0}={\bf 0}$ in~${\bf v}({\bf t})$, i.e., 
\beq
V^\alpha(t^{1,0},\dots,t^{n,0})= v^\alpha({\bf t})|_{{\bf t}_{>0}={\bf 0}}
\eeq
the {\it Riemann map associated to the solution~${\bf v}({\bf t})$}.

The {\it topological solution} ${\bf v}^M_{\rm top}(\bf t)$ \cite{Du96, DZ-norm} to the principal hierarchy is the unique 
power-series solution in $(\CC[[t^{1,0}-t^1_o,\dots,t^{n,0}-t^n_o]][[{\bf t}_{>0}]])^n$ to~\eqref{principalh520} specified by the initial data
\beq\label{initopprin521}
f_{\rm top}^{\alpha,M}(x):= v_{\rm top}^{\alpha,M}({\bf t})|_{t^{\beta,0}=\delta^{\beta}_{\iota} (x-t^\iota_o) + t^\beta_o, \, {\bf t}_{>0}={\bf 0}, \, \beta=1,\dots,n} = \delta^\alpha_\iota (x-t^\iota_o) + t^\alpha_o, \quad \alpha=1,\dots,n. 
\eeq
\begin{lemma}\label{toporiemann521}
The {\it topological solution} ${\bf v}^M_{\rm top}(\bf t)$ has the property:
\beq\label{vtopini0521}
{\bf v}_{\rm top}^{\alpha, M}({\bf t})|_{{\bf t}_{>0}={\bf 0}} \equiv t^{\alpha,0}, \quad \alpha=1,\dots,n.
\eeq
\end{lemma}
\begin{proof}
By verifying that $v_{\rm top}^\alpha({\bf t})|_{{\bf t}_{>0}={\bf 0}}$ given by the right-hand side of~\eqref{vtopini0521}
does satisfy the initial condition~\eqref{initopprin521} and the $m=0$ equations of the principal hierarchy~\eqref{principalh520}.
\end{proof}

Lemma~\ref{toporiemann521} tells that the Riemann map associated to the topological solution is nothing but the identity map.  From~\eqref{solutionv2520} and formula~\eqref{vtopini0521}, we know in particular that
\beq\label{votopttop521}
v^\alpha_{o}=t^\alpha_{o,{\rm top}}, \quad \alpha=1,\dots,n,
\eeq
where $t^\alpha_{o,{\rm top}}$ denote the primary initials associated to the topological solution.
Namely, the primary initials for the topological solution equal the coordinates 
of the center of~$\phi(B)$. Another easy fact to observe is that the topological solution is always monotone.

There are a large class of solutions to the principal hierarchy determined 
by the following genus zero Euler--Lagrange equation \cite{Du96, DZ-norm}:
\beq\label{EL424}
\sum_{\beta=1}^n \sum_{m\ge0} \bigl(t^{\beta,m}-c^{\beta,m}\bigr) 
\frac{\p \theta_{\beta,m}}{\p v^\alpha}({\bf v}({\bf t})) = 0, \quad \alpha=1,\dots,n,
\eeq
where $c^{\beta,m}$ are certain constants. (A subclass of these solutions can be identified with the 
power series 
obtained via shifting the independent  
variables by constants directly from the power series ${\bf v}_{\rm top}(\bt)$.) 
Following~\cite{Du96}, we define
\beq\label{fm01225-general}
\F^{M}_{0}({\bf t}) = \frac12 \sum_{\alpha,\beta=1}^n \sum_{m_1,m_2\ge0} 
\bigl(t^{\alpha,m_1}-c^{\alpha,m_1}\bigr)
\bigl(t^{\beta,m_2}-c^{\beta,m_2}\bigr)
\Omega^{[0]}_{\alpha,m_1;\beta,m_2}({\bf v}({\bf t})),
\eeq
which can be verified to be the 
genus zero free energy of the solution~${\bf v}({\bf t})$. 
We also recall here that 
the power series $\F^{M}_{0}({\bf t})$ defined by~\eqref{fm01225-general} 
 satisfies the following equations:
\begin{align}
&\sum_{\alpha=1}^n \sum_{m\ge1} \bigl(t^{\alpha,m}-c^{\alpha,m}\bigr) \frac{\p \F^{M}_{0}({\bf t})}{\p t^{\alpha,m-1}} 
+ \frac12 \sum_{\alpha,\beta=1}^n \eta_{\alpha\beta}(t^{\alpha,0}-c^{\alpha,0})(t^{\beta,0}-c^{\beta,0}) = 0, \label{string0423}\\
& \sum_{\alpha=1}^n \sum_{m\ge0} \bigl(t^{\alpha,m}-c^{\alpha,m}\bigr) \frac{\p \F^{M}_{0}({\bf t})}{\p t^{\alpha,m}} 
= 2 \F^{M}_{0}({\bf t}),\label{dilaton0423}
\end{align}
which are called the genus zero 
{\it string equation} and the genus zero {\it dilaton equation}, respectively. 
For the case when $c^{\alpha,m}=\delta^{\alpha,\iota}\delta^{m,1}$, 
we have ${\bf v}(\bt)= {\bf v}_{\rm top}(\bt)$,  
and the power series $\F^{M}_{0}({\bf t})$ defined by~\eqref{fm01225-general} is called the {\it topological genus zero free energy of~$M$}, denoted by 
$\F^{M}_{{\rm top}, 0}({\bf t})$. 

\subsection{Partition functions of a semisimple Frobenius manifold}\label{appa86}
In this subsection, we consider the case that $M$ is a calibrated semisimple Frobenius manifold with the flat coordinates $v^1,\dots, v^n$ 
being taken as in Section~\ref{section2}. 

\subsubsection{The Dubrovin--Zhang approach}\label{DZapproachA}
Following~\cite{DZ-norm}, define the linear operators $a_{\alpha,p}$ by
\beq\label{creation_annihilation}
a_{\alpha,p} = \left\{\begin{array} {cl}
\epsilon \frac{\p}{\p
t^{\alpha,p-1/2}}, & p>0, \\
\epsilon^{-1} (-1)^{p+1/2} \eta_{\alpha\beta} \,
(t^{\beta,-p-1/2}-c^{\beta,-p-1/2}), &  p<0, \end{array}\right.
\eeq
where $\alpha=1,\dots,n$ and $p\in \mathbb{Z}+1/2$.
Denote
\begin{align}
f_\alpha(\lambda;\nu) &= \int_0^\infty \frac{dz}{z^{1-\nu}} e^{-\lambda z}
\sum_{p\in\mathbb{Z}+\frac12} \sum_{\beta=1}^n a_{\beta,p} \left(z^{p+\mu} z^R\right)^\beta_\alpha, \label{natural_with_nu}\\
G^{\alpha\beta}(\nu) &= -\frac{1}{2\pi} \Bigl[\bigl(e^{\pi i R}e^{\pi i (\mu+\nu)}+e^{-\pi i R}e^{-\pi i (\mu+\nu)}\bigr)\eta^{-1}\Bigr]^{\alpha\beta}.\label{Galphabetanu0411}
\end{align}
Here, the right-hand side of~\eqref{natural_with_nu} is understood as a term-by-term integral.
Following~\cite{DZ-norm}, define the regularized stress tensor $T(\lambda;\nu)$ as follows:
\begin{equation}
T(\lambda;\nu) = -\frac12 : \p_\lambda(f_\alpha(\lambda;\nu)) G^{\alpha\beta}(\nu) \p_\lambda(f_\beta(\lambda;-\nu)):
+ \frac{1}{4\lambda^2} {\rm tr} \Bigl(\frac{1}{4}-\mu^2\Bigr). \label{Virasoro_field_with_nu}
\end{equation}
Here, ``$:~:$" denotes the {\it normal
ordering} (putting the annihilation operators always on the right of the creation operators).
The regularized stress tensor has the form~\cite{DZ-norm}:
\begin{align}\label{TnuLmnu427}
& T(\lambda;\nu) = \sum_{i\in \mathbb{Z}} \frac{L_i(\nu)}{\lambda^{i+2}}.
\end{align}
Then, as it is shown in~\cite{DZ-norm} (see also~\cite{DZ04}), the operators $L_i$, $i\ge-1$, defined by 
\beq\label{viralm0922}
L_i:=\lim_{\nu\to 0} L_i(\nu), \quad i\geq-1,
\eeq 
yield the {\it Virasoro operators} \cite{DZ99, DZ-norm, EHX, LiuT} for the Frobenius manifold~$M$.
 The operators $L_i$ satisfy the following Virasoro commutation relations:
\beq\label{Lmn0407}
[L_i,L_j] = (i-j) L_{i+j}, \quad \forall\,i,j\geq -1.
\eeq
In~\cite{DZ99} Dubrovin--Zhang obtain the following genus zero Virasoro constraints for 
the power series $\F^{M}_{0}({\bf t})$ defined by~\eqref{fm01225-general}
\beq\label{virasorogenus00524}
e^{-\e^{-2} \F^M_0(\bt)} L_i \bigl(e^{\e^{-2} \F^M_0(\bt)}\bigr) = O(1) \quad (\e\to0), 
\quad 
i\geq-1.
\eeq
We note that the $i=-1$ equation of~\eqref{virasorogenus00524} coincides with~\eqref{string0423}.

Under the semisimplicity assumption, Dubrovin--Zhang~\cite{DZ-norm} give a way of reconstructing the 
higher genus theory via solving  
 Virasoro constraints together with the dilaton equation:
\begin{align}
& L_i Z_{\rm top}^M(\bt;\e) = 0, \quad \forall\, i\geq-1, \label{virazM1213}\\
& \sum_{\alpha=1}^n \sum_{m\ge0} \bigl(t^{\alpha,m}-c^{\alpha,m}\bigr) 
\frac{\p Z_{\rm top}^M(\bt;\e)}{\p t^{\alpha,m}} + \e \frac{\p Z_{\rm top}^M(\bt;\e)}{\p \e}+ \frac{n}{24} Z_{\rm top}^M(\bt;\e) = 0, \label{dilatonzM1213}
\end{align}
where $c^{\alpha,m}=\delta^{\alpha,\iota}\delta^{m,1}$.
For more details about Virasoro constraints see \cite{DVV91-2, DZ99, DZ-norm, EHX, Givental, LiuT}. 
We note that, in the approach of Dubrovin--Zhang, in order to get uniqueness it is also required that 
$\log Z_{\rm top}^M(\bt;\e)$
 has the genus expansion (see~\cite{Du14, DZ-norm} and the references therein) of the following form:
\begin{align}
& \log Z_{\rm top}^M(\bt;\e) = \sum_{g\ge0} \e^{2g-2} \F_{{\rm top},g}^{M}(\bt),
\end{align}
where  
$\F_{{\rm top},g}^M(\bt)$, $g\geq1$, are power series to be determined and are 
assumed to satisfy the following ansatz:
\begin{align}
&  \F_{{\rm top},g}^M(\bt) 
= F^M_g\biggl({\bf v}_{\rm top}(\bt),\frac{\p {\bf v}_{\rm top}(\bt)}{\p x}, \dots, \frac{\p^{3g-2} {\bf v}_{\rm top}(\bt)}{\p x^{3g-2}}\biggr). \label{fmg8}
\end{align}

Dubrovin and Zhang~\cite{DZ-norm} recast the above equations into the {\it loop equation} 
\begin{align}
& \sum_{\rho=1}^n \sum_{r\ge0} \frac{\p \Delta F^M}{\p v^{\rho}_r} \p^r\biggl(\biggl(\frac1{E-\lambda}\biggr)^\rho\biggr) 
+ \sum_{\rho,\alpha,\beta=1}^n 
\sum_{r\ge1} \frac{\p \Delta F^M}{\p v^{\rho}_r} \sum_{k=1}^r 
\binom{r}{k} \p^{k-1}\biggl(\frac{\p p_\alpha}{\p v^1}\biggr) G^{\alpha\beta} \p^{r-k+1} \biggl(\frac{\p p_\beta}{\p v_\rho}\biggr) \nn \\
& = - \frac1{16} {\rm tr} \Bigl((\mathcal{U}-\lambda)^{-2}\Bigr) + \frac14 {\rm tr} \Bigl(\bigl((\mathcal{U}-\lambda)^{-1} \mathcal{V}\bigr)^2\Bigr) \nn\\
& \quad + \frac{\e^2}2 \sum_{\alpha,\beta,\rho,\sigma=1}^n\sum_{k,\ell\ge0}  \biggl(\frac{\p^2 \Delta F^M}{\p v^{\rho}_k \p v^{\sigma}_\ell} + 
\frac{\p \Delta F^M}{\p v^{\rho}_k }\frac{\p \Delta F^M}{\p v^{\sigma}_\ell}\biggr) 
\p^{k+1}\biggl(\frac{\p p_\alpha}{\p v_\rho}\biggr) 
G^{\alpha\beta} \p^{\ell+1} \biggl(\frac{\p p_\beta}{\p v_\sigma}\biggr) \nn\\
& \quad + \frac{\e^2}2 \sum_{\alpha,\beta,\rho=1}^n \sum_{k\ge0} \frac{\p \Delta F^M}{\p v^{\rho}_k} \p^{k+1} \biggl(\biggl(\nabla \biggl(\frac{\p p_\alpha}{\p \lambda}\biggr) \cdot 
\nabla \biggl(\frac{\p p_\beta}{\p \lambda}\biggr) \cdot \p({\bf v})\biggr)^\rho \, \biggr) G^{\alpha\beta} \label{loopM1208}
\end{align}
and
\begin{align}
& \sum_{\alpha=1}^n \sum_{k\ge1} k v^{\alpha}_k \frac{\p F^M_g}{\p v^{\alpha}_k} = (2g-2) F^M_g + \delta_{g,1} \frac{n}{24}, \quad g\geq1.
\label{dilatonjet1208}
\end{align}
Here $\p=\sum_{\alpha=1}^n\sum_{k\ge0} v^{\alpha}_{k+1} \frac{\p}{\p v^\alpha_k}$, 
$F_g^M=F_g^M({\bf v}, {\bf v}_1, \dots, {\bf v}_{3g-2})$, $g\geq1$, and $\Delta F^M = \sum_{g\geq1} \e^{2g-2} F^M_g$. 

\smallskip

\noindent {\bf Theorem B} (Dubrovin--Zhang~\cite{DZ98oneloop, DZ99, DZ-norm}).
{\it Let $M$ be a semisimple calibrated Frobenius manifold. The loop equations~\eqref{loopM1208}--\eqref{dilatonjet1208} 
have a unique solution up to an additive pure constant. Moreover, 
for $g=1$, $F^M_1({\bf v}, {\bf v}_1)$ can have the explicit expression
\beq
F^M_1({\bf v}, {\bf v}_1) = G({\bf v}) + \frac1{24} 
\log \det \biggl( \,\sum_{\gamma=1}^n c^{\alpha}_{\beta\gamma}({\bf v}) v^\gamma_1\biggr) , \quad G({\bf v}) := \frac{\tau_I({\bf u}({\bf v}))}{(\det(\frac{\p v^\alpha}{\p u_i}))^{\frac{1}{24}}} ,
\eeq
where $\tau_I({\bf u})$ is the isomonodromic tau-function defined via~\eqref{isotaueq}, and 
for $g\ge2$, the unique solution $F^M_g$ depends  
polynomially on $v^{\alpha}_2$, \dots, $v^{\alpha}_{3g-2}$ and rationally on $v^{\alpha}_1$, $\alpha=1,\dots,n$.}

\smallskip

\noindent We call $F^{M}_g(\bt)$ the {\it genus $g$ free energy of~$M$}.
We note that the function $G({\bf v})$ in the above Theorem~B is called the {\it $G$-function}, 
specified~\cite{DW90, DZ98oneloop, Getzler} by Getzler's equation~\cite{Getzler} and 
a certain quasi-homogeneity. 
We also note that the genus one free energy $F^M_1({\bf v}, {\bf v}_1)$ is defined up to the addition of an arbitrary pure constant, 
so the identification of the genus one free energy allows the freedom of adding a pure constant. 

We call $\F^{M}_{{\rm top},g}(\bt)$ the {\it topological genus $g$ free energy of~$M$}, 
call $\F^{M}_{{\rm top}}(\bt;\e)$ the {\it topological free energy of~$M$}, and 
 call $Z^{M}_{{\rm top}}(\bt;\e)$ the {\it topological partition function of~$M$}. 
 For general constants $c^{\beta,m}$ such that ${\bf v}(\bt)$ is a power series solution to~\eqref{EL424}, 
 define $\F^{M}_g(\bt)$ as the right-hand side of~\eqref{fmg8} but with ${\bf v}_{\rm top}(\bt)$ replaced by ${\bf v}(\bt)$.
We call $\F^{M}_g(\bt)$ the {\it genus $g$ free energy of~$M$} associated to~${\bf v}(\bt)$, call $\F^M(\bt;\e):=\sum_{g\ge0} \e^{2g-2}\F_g^M(\bt;\e)$  
the {\it free energy of~$M$} associated to~${\bf v}(\bt)$, and call $Z^M(\bt;\e):=\exp \F^M(\bt;\e)$ the {\it partition function of~$M$} associated to~${\bf v}(\bt)$ 
which satisfies~\eqref{virazM1213}, \eqref{dilatonzM1213} with $Z^M_{\rm top}$ replaced by $Z^M$.

In the approach of Dubrovin and Zhang, a hierarchy of evolutionary PDEs is also constructed, which is shown to be a 
tau-symmetric integrable hierarchy of hamiltonian evolutionary 
PDEs~\cite{BPS12-1, BPS12-2, DZ-norm}, called the {\it integrable hierarchy of topological type associated to~$M$}, 
 also known as the {\it Dubrovin--Zhang hierarchy of~$M$}.
This hierarchy is defined as the substitution of the following quasi-Miura transformation 
\beq\label{qm1225}
u^{\rm norm}_\alpha = v_\alpha + \e^2 \p_x \p_{t^{\alpha,0}} \bigl(F^M_g\bigr), \quad \alpha=1,\dots,n,
\eeq
in the principal hierarchy of~$M$. 
Clearly, 
\beq
u^{\rm norm}_{\alpha, {\rm top}}(\bt;\e) := \e^2 \frac{\p^2 \F^{M}_{{\rm top}}(\bt;\e)}{\p t^{\alpha,0} \p t^{\iota,0}}
\eeq
is a particular solution to the Dubrovin--Zhang hierarchy, which is called the {\it topological solution to the Dubrovin--Zhang hierarchy}.
Moreover, define $\Omega^M_{\alpha,m_1; \beta, m_2}$ as the substitution of the inverse of the quasi-Miura transformation~\eqref{qm1225} in 
\beq\label{defiOmega}
\Omega_{\alpha,m_1; \beta, m_2}^{M, [0]}({\bf v}) + \sum_{g\geq1} \e^{2g}  \p_{t^{\alpha,m_1}} \p_{t^{\beta,m_2}} \bigl(F^M_g\bigr).
\eeq
In particular, $\Omega_{\alpha,0; \iota,0}^M=u^{\rm norm}_\alpha$, $\alpha=1,\dots,n$.
It was proved in~\cite{BPS12-1, BPS12-2} that for all $\alpha,\beta=1,\dots,n$ and $m_1,m_2\ge0$, $\Omega^M_{\alpha,m_1; \beta, m_2}$ 
are differential polynomials of~${\bf u}^{\rm norm}=(u^{\rm norm}_1,\dots,u^{\rm norm}_n)$.
We call $\Omega^M_{\alpha,m_1; \beta, m_2}$ the {\it two-point correlation functions}, {\it aka} the {\it tau-structure}, for the 
Dubrovin--Zhang hierarchy, and call ${\bf u}^{\rm norm}$ the {\it normal coordinates}. 
For an arbitrary solution~${\bf u}^{\rm norm}({\bf t};\e)$ to the Dubrovin--Zhang hierarchy, there exists a series $\tau({\bf t};\e)$ such that 
\beq
\Omega^M_{\alpha,m_1; \beta, m_2}\Big|_{u^{\rm norm}_{\rho, \ell} \, \mapsto \, \frac{\p^\ell u^{\rm norm}_{\rho}(\bt;\e)}{\p x^\ell}, \, \ell\ge0, \, \rho=1,\dots,n} = \e^2 \frac{\p^2 \log \tau(\bt;\e)}{\p t^{\alpha,m_1} \p t^{\beta,m_2}}.
\eeq
We call $\tau({\bf t};\e)$ the {\it tau-function of the solution~${\bf u}^{\rm norm}({\bf t};\e)$} to the Dubrovin--Zhang hierarchy of~$M$.
It is clear from the definition that partition functions of~$M$ are tau-functions for the Dubrovin--Zhang hierarchy of~$M$.
We often refer to the topological partition function $Z^{M}_{{\rm top}}(\bt;\e)$ as 
the topological tau-function for the Dubrovin--Zhang hierarchy of~$M$.

\subsubsection{The Givental approach} 
Givental in~\cite{Givental} (cf.~\cite{DZ-norm, Givental01-1})
introduced a quantization approach to the construction of the partition function.

Let $V$ be a vector space. On~$V$ assign a non-degenerate symmetric bilinear form
$\langle\,,\rangle_V$. 
Denote by $\V$ the space of $V$-valued functions defined on the unit circle $S^1$ which can be
extended to an analytic function in a small annulus. On $\V$ there is a natural polarization $\V=\V_+\oplus\V_-$, where functions in $\V_+$
can be analytically continued inside of $S^1$, while functions in $\V_-$ can be analytically continued outside of $S^1$ and vanish at $z=\infty$.
Define a symplectic structure $\omega$ on $\V$ defined by
\begin{equation*}
\omega(f(z),g(z))=\frac{1}{2\pi i}\oint_{S^1} \langle f(-z),g(z)\rangle_V dz, \quad \forall\ f(z),g(z)\in\V.
\end{equation*}
The pair $(\V,\omega)$ is called \emph{Givental's symplectic space} associated to $(V, \langle\,,\rangle_V)$.

Take $w_1,\dots,w_n$ a basis of~$V$. Let $w^1,\dots,w^n$ be the dual basis with respect to $\langle\,,\rangle_V$. Any element
$f(z)\in\V$ can be written as
\[f(z)=\sum_{k\ge0}\bigl((-1)^{k+1} p_k z^{k}+q^k z^{-k-1}\bigr),\]
where $q^k=\sum_{\alpha=1}^n q^{\alpha,k} w_\alpha$ and $p_k=\sum_{\alpha=1}^n p_{\alpha,k} w^\alpha$. This gives the
Darboux coordinates
\[\{q^{\alpha,k}, p_{\alpha,k} \mid \alpha=1,\dots,n,\, k\ge 0\}\]
of the symplectic structure $\omega$.

Let $p^*\in M^0$ and take 
${\bf u}=(u_1,\dots,u_n)$ a system of canonical coordinates around~$p$ as before.
Let $V=T_{p^*}M$ and let $\langle\,,\rangle_V$ be the invariant flat metric of~$M$ being restricted at~$p^*$. 
Take the basis $w_\alpha$ to be the coordinate vector fields $\frac{\p}{\p v^\alpha}$ being restricted at~$p^*$. 
Define the canonical quantization of the corresponding Darboux coordinates $q^{\alpha,k}, p_{\alpha,k}$ as follows:
\beq
\Q(p_{\alpha,k}) = \e \frac{\p}{\p t^{\alpha,k}},\quad \Q(q^{\alpha,k})=\e^{-1} \tilde t^{\alpha,k}.
\eeq
Similarly, for $V={\rm Span}_\CC \{f_i|_p \mid i=1,\dots,n\}$ with $f_i$ being the orthonormal frame as in the previous sections, 
define the canonical quantization of the corresponding Darboux coordinates $q^{i,k}, p_{i,k}$ as follows:
\beq
\Q(p_{i,k}) = \e \frac{\p}{\p t^{(i),k}},\quad \Q(q^{i,k})=\e^{-1} \tilde t^{(i),k}.
\eeq

Let $A(z)$ be an $\mathrm{End}(V)$-valued function satisfying $A^*(-z)+A(z)=0$, where 
$*$ denotes the transposition with respect to $\langle \,, \rangle_V$. 
Such an $A(z)$ is an infinitesimal symplectic transformation
of $(\V,\omega)$, whose hamiltonian is given by
\beq
H_{A(z)}(f(z))=\frac12 \omega(f(z), A(z) f(z)).
\eeq 
This hamiltonian is a quadratic function on $\V$, and its quantization is defined by
\beq
\Q(p_Ip_J)=\e^2\frac{\p^2}{\p t^I\p t^J},\quad \Q(p_Iq^J)=\tilde t^J\frac{\p}{\p t^{I}}, \quad \Q(q^Iq^J)=\e^{-2} \tilde t^I \tilde t^J,
\eeq
where $I,J$ are $(\alpha,k)$, $(\beta,\ell)$ or $((i),k)$, $((j),\ell)$. 
Denote the quantization of $H_{A(z)}$ by $Q(H_{A(z)})$. These operators
satisfy the commutation relation
\beq
\bigl[\Q(H_{A(z)}), \Q(H_{B(z)})\bigr] = \Q(H_{[A(z), B(z)]}) +\mathcal{C}\bigl(H_{A(z)}, H_{B(z)}\bigr),
\eeq
where $\mathcal{C}$ is the $2$-cocycle satisfying 
\beq
\mathcal{C}(p_Ip_J, q^Kq^L)=-\mathcal{C}(q^Kq^L, p_Ip_J)=\delta_I^K\delta_J^L+\delta_I^L\delta_J^K, 
\eeq
and $\mathcal{C}=0$ for any other pairs of quadratic monomials. Here $I, J, K, L$ are indices of the form $(\alpha,p)$ or $((i),k)$.
Let $G(z)=e^{A(z)}$ be the symplectic transformation, then the quantization $\Q(G(z))$ of $G(z)$ is defined as $e^{\Q(H_{A(z)})}$.

Two important types of symplectic transformations are calculated in~\cite{Givental}:
\begin{itemize}
\item Type I \quad Let $G(z)$ be a symplectic transformation which is analytic inside of the unit circle. Then, for an arbitrary function
$I[\mathbf{q}(z)]$ defined on $\V_-$, we have
\beq
\bigl(\Q(G(z))^{-1} I\bigr)[\mathbf{q}(z)]=e^{\frac{1}{2\e^2}\langle\mathbf{q},\Omega \mathbf{q}\rangle_V}I\bigl[(G(z)\mathbf{q}(z))_{-}\bigr],
\eeq
where $\langle\mathbf{q},\Omega \mathbf{q}\rangle_V=\sum_{k,l\ge0}\langle q^k, \Omega_{kl}q^l\rangle_V$ is defined by
\[\sum_{k,l\geq 0} \Omega_{kl} w^{k}z^{l}=\frac{G^{*}(w) G(z) -{\rm id}}{w+z}.\]
\item Type II \quad Let $G(z)$ be a symplectic transformation which is analytic outside of the unit circle. Then for an arbitrary function
$I[\mathbf{q}(z)]$ defined on $\V_-$ we have
\beq
\bigl(\Q(G)(z) I\bigr)[\mathbf{q}(z)]=\bigl(\,e^{\frac{\e^2}{2}\langle\p_\mathbf{q}, W \p_\mathbf{q}\rangle_V}I \,\bigr)\bigl[G^{-1}(z)\mathbf{q}(z)\bigr],
\eeq
where $\langle\p_\mathbf{q}, W\p_\mathbf{q}\rangle_V=\sum_{k,l\geq 0} \langle p_k, W_{kl} p_l\rangle_V$ is defined by
$$\sum_{k,l\geq 0} (-1)^{k+l} W_{kl} w^{-k} z^{-l}=\frac{G^*(w)G(z)-{\rm id}}{w^{-1}+z^{-1}}.$$
\end{itemize}

\begin{theorem}[\cite{Givental} (cf.~\cite{DZ-norm})]
The topological partition function of~$M$ has the expression:
\[Z^{M, {\rm top}}(\bt;\e)=\tau_I({\bf u}) \Q(\Theta({\bf v}({\bf u});z))^{-1} \Q(\Psi_u)^{-1} \Q(\Phi({\bf u};z))e^{\Q(z U)}\biggl( \,\prod_{i=1}^n Z^{\rm WK}(\bt^{(i)};\e) \biggr). \]
\end{theorem}

\section{Application II. The $\kappa$th partition function}\label{sectionapp3}
Let~$M$ be a calibrated $n$-dimensional pre-Frobenius manifold. We will use the setup the same as in Section~\ref{sectionappa}.
Fix $\kappa\subset \{1,\dots,n\}$ and we will assume that $\frac{\p}{\p v^\kappa}\cdot$ is invertible on some open set.
In this section, we define the $\kappa$th solution to the principal hierarchy~\eqref{principalh520} 
of~$M$ from the viewpoint of Legendre-type transformations,  
and use it to define the $\kappa$th partition function of~$M$  
for the case when $M$ is a calibrated semisimple Frobenius manifold. 

 The following lemma is helpful. 

\begin{lemma}\label{propints521}
Assume that $\frac{\p}{\p v^\kappa}\cdot$ is invertible on some open set~$B$.
If ${\bf v}({\bf t})$ satisfies the principal hierarchy of~$M$, then 
$\hat {\bf v}({\bf t})(=\hat\phi\circ\phi^{-1}({\bf v}({\bf t})))$ satisfies the principal hierarchy of~$S_{\kappa}(M)$; 
conversely, if $\hat {\bf v}({\bf t})$ is a solution to the principal hierarchy of~$S_{\kappa}(M)$, then 
${\bf v}({\bf t})(=\phi \circ \hat\phi^{-1}(\hat {\bf v}({\bf t})))$ satisfies the principal hierarchy of~$M$.
\end{lemma}
\begin{proof}
We have
\begin{align}
\frac{\p \hat v^\alpha}{\p t^{\beta,m}} 
& = \sum_{\rho=1}^n \frac{\p \hat v^\alpha}{\p v^\rho} \sum_{\gamma,\varepsilon,\sigma=1}^n \eta^{\rho\gamma} 
c_{\gamma\sigma}^\varepsilon 
\frac{\p \theta_{\beta,m}}{\p v^\varepsilon} \frac{\p v^\sigma}{\p x}
=  \sum_{\rho,\varepsilon,\sigma=1}^n  c^\alpha_{\kappa\rho} c_{\sigma\varepsilon}^{\rho} 
\frac{\p \theta_{\beta,m}}{\p v_\varepsilon} \frac{\p v^\sigma}{\p x}
\nn \\
& = \sum_{\rho,\varepsilon,\sigma=1}^n  c^\alpha_{\varepsilon\rho} c_{\sigma\kappa}^{\rho} 
\frac{\p \hat \theta_{\beta,m}}{\p \hat v_\varepsilon} \frac{\p v^\sigma}{\p x}
= \sum_{\rho,\varepsilon=1}^n  c^\alpha_{\varepsilon\rho} 
\frac{\p \hat \theta_{\beta,m}}{\p \hat v_\varepsilon} \frac{\p v^\rho}{\p t^{\kappa,0}}
= \sum_{\rho,\gamma,\varepsilon=1}^n  c^\alpha_{\varepsilon\rho} 
\frac{\p \hat \theta_{\beta,m}}{\p \hat v_\varepsilon} \frac{\p v^\rho}{\p \hat v^\gamma} \frac{\p \hat v^\gamma}{\p t^{\kappa,0}} \nn\\
&= \sum_{\gamma,\varepsilon=1}^n  \hat c^\alpha_{\varepsilon\gamma} 
\frac{\p \hat \theta_{\beta,m}}{\p \hat v_\varepsilon} \frac{\p \hat v^\gamma}{\p t^{\kappa,0}}.
\end{align}
Similarly,
\begin{align}
\frac{\p v^\alpha}{\p t^{\beta,m}} 
& = \sum_{\gamma=1}^n \frac{\p v^\alpha}{\p \hat v_\gamma} \sum_{\varepsilon,\sigma} 
\hat c_{\gamma\sigma}^\varepsilon 
\frac{\p \hat \theta_{\beta,m}}{\p \hat v^\varepsilon} \frac{\p \hat v^\sigma}{\p t^{\kappa,0}} \nn\\
& = \sum_{\gamma=1}^n \frac{\p v^\alpha}{\p \hat v_\gamma} \sum_{\varepsilon,\sigma, \lambda=1}^n 
\hat c_{\gamma\sigma}^\varepsilon 
\frac{\p \theta_{\beta,m}}{\p v^\varepsilon} \frac{\p \hat v^\sigma}{\p v^\lambda} \frac{\p v^\lambda}{\p t^{\kappa,0}} \nn\\
& = \sum_{\gamma=1}^n \frac{\p v^\alpha}{\p \hat v_\gamma} \sum_{\varepsilon, \lambda, \rho =1}^n 
 c_{\gamma\lambda}^\varepsilon 
\frac{\p \theta_{\beta,m}}{\p v^\varepsilon} c^{\lambda}_{\kappa\rho} \frac{\p v^\rho}{\p x} 
= \sum_{\gamma,\varepsilon, \lambda, \rho=1}^n \frac{\p v^\alpha}{\p \hat v_\gamma}  
 c_{\rho\lambda}^\varepsilon 
\frac{\p \theta_{\beta,m}}{\p v^\varepsilon} c^{\lambda}_{\kappa\gamma} \frac{\p v^\rho}{\p x} \nn\\
& =  \sum_{\varepsilon, \lambda, \rho=1}^n 
 \eta^{\alpha\lambda} c_{\rho\lambda}^\varepsilon 
\frac{\p \theta_{\beta,m}}{\p v^\varepsilon} \frac{\p v^\rho}{\p x} .
\end{align}
The lemma is proved.
\end{proof}

Assume that $\frac{\p}{\p v^\kappa}\cdot$ is invertible on some open set. Denote  
 by $\hat {\bf v}_{\rm top}({\bf t})$ the topological solution to the principal hierarchy of~$S_{\kappa}(M)$. 
We know from Lemma~\ref{propints521} that 
\beq\label{defvv1013}
{}^\kappa {\bf v}^M({\bf t}):=\phi \circ \hat \phi^{-1}(\hat {\bf v}_{\rm top}({\bf t}))\in\bigl(\CC[[t^{1,0}-t^1_{o,\kappa},\dots,t^{n,0}-t^n_{o,\kappa}]][[{\bf t}_{>0}]]\bigr)^n
\eeq 
satisfies the principal hierarchy of~$M$. We 
 call ${}^\kappa v^M({\bf t})$ the {\it $\kappa$th solution} to the principal hierarchy of~$M$.
 Recall that $\hat {\bf v}_{\rm top}({\bf t})$
 satisfies the following genus zero Euler--Lagrange equation:
\beq\label{hatEL0520}
\sum_{\beta=1}^n\sum_{m\ge0} \bigl(t^{\beta,m}-\delta^{\beta,\kappa} \delta^{m,1}\bigr) 
\frac{\p \hat\theta_{\beta,m}}{\p \hat v^\alpha} \bigl(\hat {\bf v}_{\rm top}({\bf t})\bigr) = 0, \quad \alpha=1,\dots,n.
\eeq
By using~\eqref{hattt517} it follows that 
\beq\label{ELkappa0521}
\sum_{\beta=1}^n\sum_{m\ge0} \bigl(t^{\beta,m}-\delta^{\beta,\kappa} \delta^{m,1}\bigr) 
\frac{\p \theta_{\beta,m}}{\p v^\alpha} \bigl({}^\kappa {\bf v}^M({\bf t})\bigr) = 0, \quad \alpha=1,\dots,n.
\eeq
Clearly, equation~\eqref{ELkappa0521} corresponds to 
$c^{\beta,m} = \delta^{\beta,\kappa}\delta^{m,1}$ ($\beta=1,\dots,n$, $m\geq0$)  
in~\eqref{EL424}, and the $\kappa$th solution to the principal hierarchy of~$M$ could be equivalently defined by using~\eqref{ELkappa0521}. 

The initial data ${}^\kappa f^M(x)$ of the $\kappa$th solution ${}^{\kappa}v^M({\bf t})$ satisfies the equality
\beq\label{inifkappa520}
{}^\kappa f^M(x) = \phi \circ \hat \phi^{-1}\bigl(t^{1}_{o, \kappa},\dots, x= t^{\iota}_{o,\kappa}+(x- t^{\iota}_{o, \kappa}),\dots, t^{n}_{o, \kappa}\bigr),
\eeq
In particular, 
the primary initials for ${}^\kappa v^M({\bf t})$ satisfy 
$t^\alpha_{o,\kappa}=\hat{v}^\alpha_{o}$, $\alpha=1,\dots,n$, 
where $\hat v^\alpha_{o}$, $\alpha=1,\dots,n$, are the coordinates of $\hat \phi (p^*)$ with $\phi(p^*)$ being the center of~$\phi(B)$.
For the case when $\kappa=\iota$ (in this case $\p_\iota=e$ is always invertible), we know that $\phi=\hat \phi$,  so 
\beq
{}^\iota v^M({\bf t}) = v^M_{\rm top}({\bf t}).
\eeq

Since ${}^\kappa v^M({\bf t})$ is a solution to the principal hierarchy of~$M$, we 
define the {\it $\kappa$th genus zero free energy $\F^{M,\kappa}_{0}({\bf t})$ of~$M$} by 
\beq\label{f0kappa1225}
\F^{M,\kappa}_{0}({\bf t}) = \frac12 \sum_{\alpha,\beta=1}^n \sum_{m_1,m_2\ge0} 
\bigl(t^{\alpha,m_1}-\delta^{\alpha,\kappa} \delta^{m_1,1}\bigr)
\bigl(t^{\beta,m_2}-\delta^{\beta,\kappa} \delta^{m_2,1}\bigr)
\Omega^{M, [0]}_{\alpha,m_1;\beta,m_2} \bigl({}^{\kappa}v({\bf t})\bigr).
\eeq
It belongs to 
$\CC[[t^{1,0} - t^1_{o,\kappa},\dots, t^{n,0} - t^n_{o,\kappa}]][[{\bf t}_{>0}]]$,
and it is the genus zero free energy of the solution ${}^\kappa v^M({\bf t})$
(see~\eqref{fm01225-general}).

\begin{prop} \label{stduality00521}
Assume that $\frac{\p}{\p v^\kappa}\cdot$ is invertible on~$B$ for $\phi(B)$ being a sufficiently small ball.
Then the $\kappa$th genus zero free energy of~$M$ equals  
the topological genus zero free energy of~$S_{\kappa}(M)$.
\end{prop}
\begin{proof}
The proposition is proved by using \eqref{defvv1013}, \eqref{f0kappa1225}, \eqref{fm01225-general} and~\eqref{identitiestwopointOmega}.
\end{proof}

Proposition~\ref{stduality00521} reveals the space/time duality in the global form in genus zero. 

It follows either from Proposition~\ref{stduality00521} or directly from Corollary~\ref{simplecor514} the following
\begin{cor}\label{taugenuszerovv}
Let $M$ be a pre-Frobenius manifold. Assume that $\frac{\p}{\p v^\kappa}\cdot$ is invertible on~$B$ for $\phi(B)$ being a sufficiently small ball.
Let $S_\kappa(M)$ be the Legendre-type transformation given by $\frac{\p}{\p v^\kappa}$.
An arbitrary tau-function for the principal hierarchy of~$M$ is 
also a tau-function for the principal hierarchy of~$S_{\kappa}(M)$; vice versa. 
\end{cor}

We will proceed to the case when $M$ is a semisimple Frobenius manifold.
Before proceeding let us prove the following lemma.
\begin{lemma}\label{monotonicitylemma}
Assume that $\frac{\p}{\p v^\kappa}\cdot$ is invertible on the open set~$B$. Then 
the $\kappa$th solution ${}^{\kappa}v^M({\bf t})$ to the principal hierarchy~\eqref{principalh520} of~$M$ is monotone. 
\end{lemma}
\begin{proof}
From the definition of the principal hierarchy of~$M$ we know that 
\beq\label{phkappa0}
\frac{\p ({}^{\kappa}v^{\alpha,M}({\bf t}))}{\p t^{\kappa,0}} 
= \sum_{\beta=1}^n c^{\alpha}_{\kappa\beta}({}^{\kappa}{\bf v}^M({\bf t})) \frac{\p ({}^{\kappa}v^{\beta,M}({\bf t}))}{\p x} \,.
\eeq
Taking $t^{\alpha,0}=t^\alpha_{o,\kappa}$ and $t^{\alpha,k}=0$ ($k\ge1$), $\alpha=1,\dots,n$, on both sides of~\eqref{phkappa0}, 
we find 
\beq\label{phkappa0restrict}
\frac{\p v^{\alpha}}{\p \hat v^\rho}\bigg|_{\hat v=\hat v_o}  \delta^\rho_\kappa
= c^{\alpha}_{\kappa\beta}\bigl(\phi \circ \hat \phi^{-1}(\hat {\bf v}_o)\bigr) 
\frac{\p ({}^{\kappa}v^{\beta,M}({\bf t}))}{\p x}\bigg|_{t^{\alpha,0}=t^\alpha_{o,\kappa}, \, t^{\alpha,k}=0\, (k\ge1), \, \alpha=1,\dots,n} \,.
\eeq
So 
\begin{align}
\xi^\beta:=\frac{\p ({}^{\kappa}v^{\beta,M}({\bf t}))}{\p x}\bigg|_{t^{\alpha,0}=t^\alpha_{o,\kappa}, \, t^{\alpha,k}=0\, (k\ge1), \, \alpha=1,\dots,n} = 
\sum_{\alpha=1}^n \frac{\p v^{\alpha}}{\p \hat v^\kappa}\bigg|_{\hat {\bf v}=\hat {\bf v}_o} \frac{\p v^\beta}{\p \hat v^\alpha}\bigg|_{\hat {\bf v}=\hat {\bf v}_o}.
\end{align}
Let $\xi=\sum_{\beta=1}^n \xi^\beta \frac{\p}{\p v^\beta}$. Then 
\begin{align}
\xi \cdot \frac{\p}{\p v^\rho} 
& = \sum_{\alpha,\beta, \sigma=1}^n \frac{\p v^{\alpha}}{\p \hat v^\kappa}\bigg|_{\hat {\bf v}=\hat {\bf v}_o} 
\frac{\p v^\beta}{\p \hat v^\alpha}\bigg|_{\hat {\bf v}=\hat {\bf v}_o} 
c^\sigma_{\rho\beta}|_{{\bf v}={\bf v}_o} \frac{\p}{\p v^\sigma}
= \sum_{\alpha, \sigma=1}^n \frac{\p v^{\alpha}}{\p \hat v^\kappa}\bigg|_{\hat {\bf v}=\hat {\bf v}_o} 
\hat c^\sigma_{\rho\alpha}|_{\hat {\bf v}= \hat {\bf v}_o} \frac{\p}{\p v^\sigma} \nn\\
& = \sum_{\alpha, \beta, \sigma=1}^n \frac{\p v^{\alpha}}{\p \hat v^\kappa}\bigg|_{\hat {\bf v}=\hat {\bf v}_o} 
\frac{\p v^\beta}{\p \hat v^\rho}\bigg|_{\hat {\bf v}=\hat {\bf v}_o} 
c^\sigma_{\beta\alpha}|_{{\bf v}={\bf v}_o} \frac{\p}{\p v^\sigma}
= \sum_{\beta, \sigma=1}^n 
\frac{\p v^\beta}{\p \hat v^\rho}\bigg|_{\hat {\bf v}=\hat {\bf v}_o} 
\hat c^\sigma_{\beta\kappa}|_{\hat {\bf v}=\hat {\bf v}_o} \frac{\p}{\p v^\sigma}\nn\\
& = \sum_{\sigma=1}^n 
\frac{\p v^\sigma}{\p \hat v^\rho}\bigg|_{\hat {\bf v}=\hat {\bf v}_o}  \frac{\p}{\p v^\sigma}.
\end{align}
The monotonicity is then implied by Lemma~\ref{invertible}.
\end{proof}

We now consider the case when $M$ is a semisimple Frobenius manifold.
Using Lemma~\ref{monotonicitylemma}, we can 
define the $\kappa$th genus $g$ ($g\ge1$) free energy $\F^{M,\kappa}_{g}({\bf t})$ of~$M$ by  
\beq\label{jet1225}
\F^{M,\kappa}_{g}({\bf t}) = F^M_g|_{{\bf v}_m\mapsto \p_x^m({}^{\kappa}{\bf v}^M({\bf t})), \, m\ge0}.
\eeq
We call 
\beq
\F^{M,\kappa}({\bf t};\e) = \sum_{g\ge0} \e^{2g-2} \F^{M,\kappa}_{g}({\bf t})
\eeq
the {\it $\kappa$th free energy of~$M$}, and call
\beq\label{genuszmkappa1225}
Z^M_\kappa({\bf t};\e) = \exp \F^{M,\kappa}({\bf t};\e)
\eeq
the {\it $\kappa$th partition function of~$M$}. 
Obviously, the $\iota$th free energy $\F^{M,\iota}({\bf t})$ 
is nothing but 
the topological free energy of~$M$, and 
 the $\iota$th partition function of~$M$ is the 
topological partition function of~$M$.

In the following proposition we list more properties for the $\kappa$th partition function of~$M$.
\begin{prop}
The $\kappa$th partition function of~$M$ has the following properties:
\begin{itemize}
\item[i)] It satisfies the Virasoro constraints:
\beq\label{virakappa1225}
L_i \bigl(Z^M_\kappa(\bt;\e)\bigr) = 0, \quad \tilde t^{\alpha,m}=t^{\alpha,m}-\delta^{\alpha,\kappa}\delta^{m,1}, \, i\geq-1.
\eeq 
\item[ii)] It satisfies the dilaton equation:
\beq\label{dilatonkappa1225}
\sum_{\alpha=1}^n \sum_{m\ge0} \bigl(t^{\alpha,m}-\delta^{\alpha,\kappa}\delta^{m,1}\bigr)  \frac{\p Z_\kappa^M(\bt;\e)}{\p t^{\alpha,m}} 
+ \e \frac{\p Z_\kappa^M(\bt;\e)}{\p \e}+ \frac{n}{24} Z_\kappa^M(\bt;\e) = 0.
\eeq
\item[iii)] It is a tau-function for the Dubrovin--Zhang hierarchy associated to~$M$. 
\item[iv)] Its logarithm has the genus expansion (cf.~\eqref{genuszmkappa1225}).
\item[v)] The genus zero part has the expression~\eqref{f0kappa1225}.
\item[vi)] The genus~$g$ ($g\ge1$) part has the jet-variable representation (cf.~\eqref{jet1225}). 
\end{itemize}
\end{prop}
\begin{proof}
The properties iv), v), vi) are true by definition. Using the property~vi) 
we find that \eqref{virakappa1225} can be recast to the same loop equation as~\eqref{loopM1208}, which 
holds by the definition~\eqref{jet1225}. This proves~i). In a similar way we find that~ii) is true. To show~iii), first, 
we note that 
\beq
{}^{\kappa}{\bf u}^M({\bf t};\e):= {}^{\kappa} {\bf v}^M({\bf t}) + \sum_{g\geq1} \e^{2g} \frac{\p^2 \F^{M,\kappa}_{g}(\bt;\e)}{\p x^2}
\eeq
is a particular solution~\cite{Du14, DZ-norm} to the Dubrovin--Zhang hierarchy of~$M$. Secondly, 
the derivatives 
\begin{align}
\e^2 \frac{\p^2 \F^{M,\kappa}(\bt;\e)}{\p t^{\alpha,m_1} \p t^{\beta, m_2}} 
&= \Omega_{\alpha,m_1;\beta,m_2}^{M, [0]}\bigl({}^{\kappa}{\bf v}({\bf t})\bigr) + \sum_{g\ge1} \e^{2g} \frac{\p^2 \F^{M,\kappa}_{g}(\bt)}{\p t^{\alpha,m_1} \p t^{\beta, m_2}}\nn\\
&= \Omega_{\alpha,m_1;\beta,m_2}^{M, [0]}\bigl({}^{\kappa}{\bf v}({\bf t})\bigr) 
+ \sum_{g\ge1} \e^{2g} \frac{\p^2 (F^M_g|_{{\bf v}_m\mapsto \p_x^m({}^{\kappa}{\bf v}^M({\bf t})),m\ge0})}{\p t^{\alpha,m_1} \p t^{\beta, m_2}} \nn\\
&= \Omega^M_{\alpha,m_1; \beta, m_2}\big|_{u_{\rho, \ell} \, \mapsto \, \frac{\p^\ell ({}^{\kappa} u^M_\rho({\bf t};\e))}{\p (t^{\iota,0})^\ell}, \, \ell\ge0, \, \rho=1,\dots,n}  . \nn
\end{align}
Here in the last equality we first substitute the inverse of the quasi-Miura map~\eqref{qm1225} then take $u_\rho$ to be ${}^{\kappa}u^M_\rho({\bf t};\e)$.
The proposition is proved.
\end{proof}

We establish in the following theorem the relationship between the $\kappa$th partition function of~$M$ 
and the topological partition function of~$S_\kappa(M)$. 

\begin{theorem}\label{mainthm1229}
The following identity holds true:
\beq
Z^M_{\kappa}({\bf t}) = Z^{S_\kappa(M)}_{\rm top}(\bt).
\eeq
\end{theorem} 
\begin{proof}
By the construction of Dubrovin--Zhang (see~\cite{DZ-norm} or e.g. Section~\ref{DZapproachA}) and Theorem~\ref{thm2519}, we find that 
the partition function $Z^{S_\kappa(M)}_{\rm top}(\bt)$ satisfies the same Virasoro constraints~\eqref{virakappa1225} 
as $Z^M_{\kappa}(\bt)$ does and 
the same dilaton equation~\eqref{dilatonkappa1225}. The logarithm of $Z^{S_\kappa(M)}_{\rm top}(\bt)$ also has the genus expansion, whose 
genus zero part, according to Proposition~\ref{stduality00521}, 
 coincides with that of $\log Z^M_{\kappa}(\bt)$. 
 Using~\eqref{defvv1013} and using~\eqref{phkappa0}, noticing that $c^{\alpha}_{\kappa\beta}({\bf v})$ is invertible within~$B$, we find that 
 the jet-variable representation~\eqref{jet1225} in~${\bf v}$ with $x$-derivatives of
  the genus $g$ ($g\geq1$) part of~$\log Z^M_{\kappa}(\bt)$ can be transformed to the jet-variable representation in~$\hat {\bf v}$ with 
  $\p_{t^{\kappa,0}}$-derivatives. 
 The theorem is proved by the uniqueness of the loop equation (see~\cite{DZ-norm} or Section~\ref{DZapproachA}).
\end{proof}

An essential point for the beginning of the above proof is that by Theorem~\ref{thm2519} the Virasoro operators,  
 whose definition (see~\cite{DZ-norm, EHX, Givental, LiuT} or see \eqref{creation_annihilation}--\eqref{viralm0922})
 uses only
$\eta, \mu, R$, are identical for $M$ and for $S_\kappa(M)$. This identification does not require semisimplicity.
We also note that in the above proof an alternative equation for~\eqref{phkappa0} to transform the jet-variable representation is the following:
 \beq\label{phkappa0iota}
 \frac{\p \hat v^\alpha}{\p x} = \sum_{\beta=1}^n \hat c^{\alpha}_{\iota\beta}(\hat {\bf v}) \frac{\p \hat v^\beta}{\p t^{\kappa,0}}, \quad \alpha=1,\dots,n.
 \eeq
 
We have the following simple corollary of Theorem~\ref{mainthm1229}.
\begin{cor}\label{thiscorollary}
For each $g\ge1$, we have
\beq\label{FMgFSMg}
F_g^M({\bf v}, {\bf v}_1, \dots, {\bf v}_{3g-2})=F^{S_\kappa(M)}_g(\hat {\bf v}, \hat {\bf v}_1, \dots, \hat {\bf v}_{3g-2}),
\eeq
where $({\bf v}_k)_{k\ge0}$ and $(\hat {\bf v}_k)_{k\ge0}$ are related iterately using~\eqref{defhatv59} 
and either~\eqref{phkappa0} or~\eqref{phkappa0iota}. Here we note that ${\bf v}_k=\p_{t^{\iota,0}}^k({\bf v})$, 
$\hat {\bf v}_k=\p_{t^{\kappa,0}}^k({\bf v})$, $k\ge0$.
\end{cor}

For the particular example when $M$ is the $\mathbb{P}^1$ Frobenius manifold (cf.~Example~\ref{exampleA11}),  
Corollary~\ref{thiscorollary} can also be deduced from~\cite{Du09, Y2} (cf.~\cite{CDZ, CvdLPS, CvdLPS2, FY, FYZ, Y1}). 

Let us illustrate the special $g=1$ case of Corollary~\ref{thiscorollary} by means of examples. 
In Example~\ref{exampleA11} of Section~\ref{sectionapp1}, the genus one free energy of $M$ has the expression
\beq\label{genus1Fcp1}
F^M_1(v,u,v_1,u_1) = \frac1{24} \log (v_1^2-e^u u_1^2) - \frac1{24} u,
\eeq
where $v=v^1, u=v^2$, 
and the genus one free energy of $S_2(M)$ has the expression
\beq\label{genus1FS2}
F^{S_2(M)}_1(\rho, \varphi, \rho_1, \varphi_1) = \frac1{24} \log \Bigl(\varphi_1^2-\frac1\rho \rho_1^2\Bigr) - \frac1{12} \log \rho.
\eeq
Here $\varphi=\hat v^2, \rho=\hat v^1$, and by~\eqref{vvhatcoord} we know that 
\beq\label{cp1legendretype}
\varphi=v, \quad \rho = e^u .
\eeq
Equations~\eqref{phkappa0} with~\eqref{cp1legendretype} give
\beq\label{evolutiontodadispersionless}
\frac{\p \varphi}{\p t^{2,0}} = \frac{\p \rho}{\p x}, \quad \frac{\p \rho}{\p t^{2,0}} = \rho \frac{\p \varphi}{\p x}.
\eeq
Substituting \eqref{evolutiontodadispersionless}, \eqref{cp1legendretype} into~\eqref{genus1FS2} and using~\eqref{genus1Fcp1}, we find that \eqref{FMgFSMg} with $g=1$ is indeed true up to an irrelevant constant that is allowed in genus one free energies.

In Example~\ref{examplea2}, the genus one free energy of $M$ has the expression
\beq\label{genus1Fa2}
F^M_1(v,u,v_1,u_1) = \frac1{24} \log \Bigl(v_1^2- \frac13 u u_1^2\Bigr),
\eeq
where $v=v^1, u=v^2$, and the genus one free energy of $S_2(M)$ has the expression
\beq\label{genus1FS2a2}
F^{S_2(M)}_1(\rho, \varphi, \rho_1, \varphi_1) =  \frac1{24} \log \Bigl(\varphi_1^2 - \frac{\sqrt{3}}{\sqrt{2\rho}} \rho_1^2\Bigr)  -  \frac1{48} \log \rho 
\eeq
Here $\varphi=\hat v^2, \rho=\hat v^1$, and by~\eqref{vvhatcoord} we know that  
\beq\label{a2legendretype}
\varphi=v, \quad \rho=\frac16 u^2.
\eeq
Equations~\eqref{phkappa0} with~\eqref{a2legendretype} give
\beq\label{evolutiona2dispersionless}
\frac{\p \varphi}{\p t^{2,0}} = \frac{\p \rho}{\p x}, \quad \frac{\p \rho}{\p t^{2,0}} = \frac{\sqrt{6\rho}}3 \frac{\p \varphi}{\p x}.
\eeq
Substituting \eqref{evolutiona2dispersionless}, \eqref{a2legendretype} into~\eqref{genus1FS2a2} and using~\eqref{genus1Fa2}, we find that \eqref{FMgFSMg} with $g=1$ is indeed true up to an irrelevant constant that is allowed in genus one free energies.

More generally, consider the two-dimensional Frobenius manifold with potential 
\beq\label{F2dgeneral}
F^M(v,u)= \frac12 v^2 u + f(u), \quad f(u) := c \, u^m, \quad D=\frac{m-3}{m-1},
\eeq
where $c$ is an arbitrarily given non-zero constant, $v=v^1, u=v^2$, and $m\neq 0$, $m\neq1$, $m\neq 2$, $m\neq \frac32$. 
The $\p/\p t^{2,0}$ flow in the principal hierarchy of~$M$ reads 
\beq\label{t20flowf}
 \frac{\p u}{\p t} = \frac{\p v}{\p x}, \quad \frac{\p v}{\p t}= f'''(u) \frac{\p u}{\p x}, \quad t=t^{2,0}.
\eeq
We have
\beq\label{F12d}
F_1^M(v,u,v_1,u_1) = \frac1{24}\log\bigl(v_1^2-c\,m(m-1)(m-2)u^{m-3}u_1^2\bigr) -\frac{(m-3)(m-4)}{24 \, (m-1)} \log u.
\eeq
Consider the Legendre-type transformation $S_2$. 
Let $\varphi=\hat v^2, \rho=\hat v^1$, i.e., by~\eqref{vvhatcoord} 
\beq\label{map2d2}
\varphi=v, \quad \rho=f''(u).
\eeq 
The potential of the $S_2(M)$ has the expression
\beq\label{F1S22d}
F^{S_2(M)}(\rho, \varphi) = \frac12 \varphi^2 \rho + \frac{(cm)^{-\frac{1}{m-2}}(m-1)^{-\frac{m-1}{m-2}}(m-2)^2}{2m-3} \rho^{\frac{2m-3}{m-2}}, 
\quad \widehat{D}=-D,
\eeq
and the genus one free energy of $S_2(M)$ has the expression
\beq
F^{S_2(M)}_1(\rho, \varphi, \rho_1, \varphi_1) = \frac1{24}\log\Bigl(\varphi_1^2- \frac{(cm(m-1))^{-\frac{1}{m-2}}}{m-2}\rho^{-\frac{m-3}{m-2}} \rho_1^2\Bigr) -\frac{(m-3)(2m-5)}{24 \, (m-1)(m-2)} \log \rho.
\eeq
Using \eqref{t20flowf}, \eqref{F1S22d}, \eqref{F12d}, \eqref{map2d2}, 
we find that~\eqref{FMgFSMg} with $g=1$ is indeed true up to an irrelevant constant that is allowed in genus one free energies, where in this straightforward verification of the $G$-function parts we used the elementary equality
$$
\frac{m-3}{24}-\frac{(m-3)(2m-5)}{24 \, (m-1)(m-2)} (m-2) = -\frac{(m-3)(m-4)}{24 \, (m-1)}.
$$

Finally, consider the case 
$$
F^M(v,u)= \frac12 v^2 u + c \, u^{\frac32}, \quad D=-3,
$$
where $c\neq0$.
The verification procedure for~\eqref{FMgFSMg} with $g=1$ is similar to the above, and we will only list some useful explicit expressions. 
Equations~\eqref{F2dgeneral}--\eqref{F12d} still hold; in particular, 
\beq\label{F12d32}
F_1^M(v,u,v_1,u_1) = \frac1{24}\log\Bigl(v_1^2+\frac{3c}{8} u^{-\frac32}u_1^2\Bigr) - \frac{5}{16} \log u,
\eeq
and $u_t=v_x$, $v_t=-\frac38 c \, u^{-\frac32} u_x$, $t=t^{2,0}$.
We have
\beq\label{logfm}
F^{S_2(M)}(\rho, \varphi) = \frac12 \varphi^2 \rho -\frac{9}{16} c^2  \log \rho, \quad \widehat{D}=3,
\eeq
as well as
$$
F^{S_2(M)}_1(\rho, \varphi, \rho_1, \varphi_1) = \frac1{24}\log\Bigl(\varphi_1^2 + \frac{9}8 c^2 \rho^{-3} \rho_1^2\Bigr) + \frac12 \log \rho.
$$
Here,  
$$
\varphi=v, \quad \rho=\frac{3c}{4} u^{-\frac12}.
$$
For the analytic aspect of the Frobenius manifold $S_2(M)$ with potential~\eqref{logfm} we refer to~\cite{Gu01}.

Another corollary of Theorem~\ref{mainthm1229} is from the perspective of integrable systems. 
\begin{cor}\label{cortau}
Let $M$ be a semisimple calibrated Frobenius manifold, and $S_\kappa(M)$ its Legendre-type transformation.
An arbitrary tau-function for the Dubrovin--Zhang hierarchy of~$M$ is 
also a tau-function for the Dubrovin--Zhang hierarchy of~$S_{\kappa}(M)$; vice versa.  
\end{cor}
\begin{proof} 
Using \eqref{identitiestwopointOmega}, \eqref{FMgFSMg}, and~\eqref{defiOmega}, we find the identities 
\beq
\Omega_{\alpha,m_1;\beta,m_2}^{S_\kappa(M)}
= \Omega_{\alpha,m_1;\beta,m_2}^{M}, \quad \forall\,\alpha,\beta=1,\dots,n,\, \forall\,m_1,m_2\geq0, \label{OmOmidallgenus}
\eeq
whose validity can be understood as functions of ${\bf v}, {\bf v}_1, {\bf v}_2$, \dots (or alternatively 
as functions of $\hat {\bf v}, \hat {\bf v}_1, \hat {\bf v}_2$, \dots). Thus they are also valid as differential
polynomials of~${\bf u}^{\rm norm}$ (or alternatively as differential polynomials of $\hat {\bf u}^{\rm norm}$). The corollary is proved.
\end{proof}

For the special example when $M$ is the 
Frobenius manifold corresponding to the extended bigraded Toda hierarchy of $(m,1)$-type~\cite{Carlet}
and for a special~$\kappa$, 
Theorem~\ref{mainthm1229} and Corollary~\ref{cortau}
can be deduced from~\cite{CDZ, CvdLPS, CvdLPS2, Du09, FY, FYZ, Y1, Y2} (cf.~Examples \ref{exampleA11}, \ref{exampleA21}). 

By taking $\beta=\kappa, m_1=m_2=0$ in~\eqref{OmOmidallgenus} we get an expression of $\hat{u}^{\rm norm}_\alpha$ as 
a differential polynomial of~$\bf u^{\rm norm}$ (the differentiation is with respect to~$x=t^{\iota,0}$); similarly, 
by taking $\beta=\iota, m_1=m_2=0$ in~\eqref{OmOmidallgenus} we get an expression of $u^{\rm norm}_\alpha$ as 
a differential polynomial of~$\hat{\bf u}^{\rm norm}$ (the differentiation is with respect to~$t^{\kappa,0}$). 
For special examples these expressions can be more explicitly calculated out; see~e.g. \cite{CDZ, FY, FYZ}.

Generalizations of Theorem~\ref{mainthm1229} and Corollary~\ref{cortau} to 
other situations (cf.~e.g.~\cite{DS, DZ-norm, HM, Manin, SS, Vekslerchik, YZ, YZhang}) will be 
given in a subsequent publication.

We end the main body of the paper with 
a second proof of Theorem~\ref{mainthm1229} by using the Givental quantization. 
\begin{proof} [A second proof of Theorem~\ref{mainthm1229}.] We have
\begin{align}
Z^{S_\kappa(M), {\rm top}}(\bt;\e)=& \tau_I({\bf u}) \Q(\hat \Theta(\hat {\bf v} ({\bf u});z))^{-1} \Q(\hat \Psi_u)^{-1} \Q( \hat \Phi({\bf u};z))e^{\Q(z U)}\Biggl(\prod_{i=1}^n Z^{\rm WK}(\bt^{(i)};\e)\Biggr) \nn\\
= & \tau_I({\bf u}) \Q(\Theta({\bf v}({\bf u});z))^{-1} \Q(\Psi_u)^{-1} \Q(\Phi({\bf u};z))e^{\Q(z U)}\Biggl(\prod_{i=1}^n Z^{\rm WK}(\bt^{(i)};\e)\Biggr), \nn
\end{align}
where we note that 
\beq
 \Q(q^{\alpha,k})=\e^{-1} (t^{\alpha,k}-\delta^{\alpha,\kappa}\delta^{k,1}).
\eeq
The theorem is proved.
\end{proof}

\begin{appendix}

\section{Further examples}\label{appexamples}
In this appendix, we will provide three more examples of Legendre-type transformations of a Frobenius manifold.
\begin{example}\label{examplep2quantumcoho}
$X=\mathbb{P}^2$ and $M=QH(\mathbb{P}^2)$. This example is famous according to Kontsevich's counting 
of degree~$d$ rational nodal curves on $\mathbb{P}^2$ (\cite{DiFI, Du96, KM94, RT}). The potential, i.e., 
the genus zero primary free energy, now has the expression
\beq
F({\bf v}) = \frac{(v^1)^2 v^3 + v^1 (v^2)^2}{2} + \sum_{d\ge1} \frac{N_{d}}{(3d-1)!} (v^3)^{3d-1} e^{dv^2},
\eeq
where $N_{d}$ denotes the number of nodal curves of genus zero with degree $d$ passing 
through $(3d-1)$ general points on~$\mathbb{P}^2$. The first few $N_d$'s are $N_1=1$, $N_2=1$, $N_3=12$, $N_4=620$. Following~\cite{Du96}, if we write 
\beq
F({\bf v}) = \frac{(v^1)^2 v^3 + v^1 (v^2)^2}{2} + \frac1{v^3} \, f\bigl(v^2+3\log v^3\bigr),
\eeq
then the WDVV equations written in terms of the function $f=f(s)$ become
\beq
27 f'''-3 f''f''' +2 f'f''' -f''^2-54 f''+33 f'-6 f =0, \quad '=\frac{d}{ds}.
\eeq
This equation implies the following 
recursive relation for $N_{d}$: 
\beq
N_{d} = \sum_{d_1,d_2>0\atop d_1+d_2=d} N_{d_1} N_{d_2} \Biggl(d_1^2 d_2^2 \binom{3d-4}{3d_1-2} - d_1^3 d_2 \binom{3d-4}{3d_1-1}\Biggr)\,.
\eeq
The charge~$D$ of~$M$ is~2, and 
the Euler vector field reads
\beq
E = v^1 \frac{\p}{\p v^1} + 3 \frac{\p}{\p v^2} - v^3 \frac{\p}{\p v^3}.
\eeq
The analyticity of~$F({\bf v})$ was known~\cite{DiFI, Du96, Du99}, and the monodromy data $\eta, \mu, R$, $S$ are given by~\cite{Du96, Du99}
\beq\label{muR1012}
\eta=\begin{pmatrix} 0 & 0 & 1 \\ 0 & 1 & 0\\ 1 & 0 & 0\end{pmatrix}, \quad \mu=\begin{pmatrix} -1 & 0 & 0\\ 0 & 0 & 0\\ 0 & 0 & 1\end{pmatrix}, 
\quad R=\begin{pmatrix} 0 & 0  & 0 \\ 3 & 0 & 0\\ 0 & 3 & 0\end{pmatrix}, \quad 
S=\begin{pmatrix} 1 & 3 & 3 \\ 0 & 1 & 3\\ 0 & 0 & 1\end{pmatrix}.
\eeq

Let $\kappa=2$. The Frobenius manifold $S_2(M)$ is of charge~0 and has the potential 
\begin{align}
F^{S_2(M)} ({\bf \hat v}) =\frac{(\hat v^2)^3}{6}+ \hat v^1 \hat v^2 \hat v^3 + e^{\hat v^3} + \frac{(\hat v^1)^3}6 e^{-\hat v^3} 
- \frac{(\hat v^1)^6}{360}e^{-3 \hat v^3}
+ \cdots. 
\end{align}
Here we note that the analyticity of $F^{S_2(M)} ({\bf \hat v})$ can be deduced from that of~$M$ and the inverse function theorem. 
The Euler vector field for $S_2(M)$ reads
\beq
E= 2\hat v^1 \frac{\p}{\p \hat v^1} + \hat v^2 \frac{\p}{\p \hat v^2} + 3 \frac{\p}{\p \hat v^3}.
\eeq
Write 
\begin{align}
F^{S_2(M)} ({\bf v}) = \frac{(\hat v^2)^3}{6}+ \hat v^1 \hat v^2 \hat v^3 + e^{\hat v^3} p\bigl((\hat v^1)^3 e^{-2\hat v^3}\bigr), \label{p2fs2}
\end{align}
where 
\beq
p(t) = \sum_{k\ge0} \frac{C_k}{(3k)!} t^k=1+O(t), \quad C_0=1,
\eeq
is a power series of one variable~$t$. Then we have
\begin{align}
&144 t^4 p''' p'' + 24 t^3 (p''' p' + 12 p''^2 ) + 27 t^2 ( p''' p + 3 p' p'') \nn\\
&\qquad + 6 t (9 p p'' +  p'^2)+6 p p' -1 =0, \qquad '=d/dt. \label{p2fs2eqG}
\end{align}
Equation~\eqref{p2fs2eqG} implies a recursion for the coefficients $C_k$ of $p(t)$. The first few values of $C_k$, $k\ge0$, are 
$$
1, ~1, ~ -2, ~104, ~ -24920, ~16361976, ~ -22819065536, ~\dots.
$$
We expect that $C_k$ are all integers. Recently, Zhengfei Huang has proved this integrality, which  
 follows from the recursion directly. 

Let $\kappa=3$. The Frobenius manifold $S_3(M)$ is of charge~$-2$ and has the potential 
\beq
F^{S_3(M)}({\bf \hat v}) =  \frac{(\hat v^3)^2 \hat v^1}2 + \frac{(\hat v^2)^2 \hat v^3}2 
+ (\log(\hat v^1)-1) \hat v^1 \hat v^2 + \frac1{24} \frac{(\hat v^2)^4}{\hat v^1} + \frac1{1260}\frac{(\hat v^2)^7}{(\hat v^1)^3} + \cdots.   
\eeq
The Euler vector field now reads
\beq
E= 3 \hat v^1 \frac{\p }{\p \hat v^1} + 2\hat v^2 \frac{\p}{\p \hat v^2} + \hat v^3 \frac{\p}{\p \hat v^3}.
\eeq
Write 
\beq
F^{S_3(M)}({\bf \hat v}) = \frac{(\hat v^3)^2 \hat v^1}2 + \frac{(\hat v^2)^2 \hat v^3}2 
+ \left(\log(\hat v^1)-1\right) \hat v^1 \hat v^2 + \hat v^1 \hat v^2 \, m \biggl(\frac{(\hat v^2)^3}{(\hat v^1)^2}\biggr), \label{p2fs3}
\eeq
where 
\beq\label{expinitialS3p2}
m(r)= \sum_{k\ge1} \frac{4^{k-1} M_k}{(3k+1)!} r^k=\frac{r}{24}+O(r^2), \quad M_1=1,
\eeq
is a power series of one variable~$r$.
Then we have
\begin{align}
&9 ( 3 + 2 \dot{m} + 6 \ddot{m} ) \dddot{m} - 3 (4 \dot{m} + 3 \ddot{m} ) \ddot{m} - 3 (1+\dot{m}) \dot{m}
= r (8 \dddot{m} -2 \dot{m} +1), \quad \dot{}=r \frac{d}{dr}. \label{ODES3p2}
\end{align}
The coefficients $M_k$ $(k\ge1)$ are determined by~\eqref{ODES3p2} and the initial value given in~\eqref{expinitialS3p2}, 
which begin with
$$1, ~1, ~8, ~177, ~6234, ~-67965, ~\dots.$$ 
We expect that $M_k$ are all integers. This integrality away from~2 has been proved by Zhengfei Huang.
Again this part is obtained directly from the recursion.

Theorem~\ref{maintheorem} says that 
the monodromy data $\mu, R, S$ of the Frobenius manifolds $S_2(M)$ and $S_3(M)$   
could have the same expression as those in \eqref{muR1012}. 
\end{example}

The next example is about the Legendre-type transformations of a tensor product.

\begin{example}
Consider the $\mathbb{P}^1$ Frobenius manifold~$M_{\mathbb{P}^1}=QH(\mathbb{P}^1)$, which has the potential~\eqref{p1potentialf}.
The tensor product 
\beq
M := M_{\mathbb{P}^1} \otimes M_{\mathbb{P}^1}
\eeq
gives the Frobenius 
manifold of charge $D=2$ for the quantum cohomology of $\mathbb{P}^1\times \mathbb{P}^1$~\cite{DiFI, KMK}, which has the potential
\beq
F= \frac12 (v^{(1,1)})^2 v^{(2,2)} + v^{(1,1)} v^{(1,2)} v^{(2,1)} 
+ \sum_{k,l\ge0 \atop k+l\ge 1} \frac{N_{k,l}}{(2k+2l-1)!} e^{kv^{(1,2)}+lv^{(2,1)}} (v^{(2,2)})^{2k+2l-1}
\eeq
and the Euler vector field
\beq
E= v^{(1,1)} \frac{\p }{\p v^{(1,1)}} - v^{(2,2)} \frac{\p }{\p v^{(2,2)}} + 2  \frac{\p }{\p v^{(1,2)}} + 2 \frac{\p }{\p v^{(2,1)}} .
\eeq
Here $N_{k,l}$ counts the number of rational curves on the quadric, with bi-degree $(k,l)$, passing through 
$2(k + l)-1$ points. 
If we write 
\beq
F=\frac12 (v^{(1,1)})^2 v^{(2,2)} + v^{(1,1)} v^{(1,2)} v^{(2,1)} 
+  \frac1{v^{(2,2)}} \, f\bigl(v^{(1,2)}+2\log v^{(2,2)}, v^{(2,1)}+2\log v^{(2,2)} \bigr),
\eeq
where 
\beq\label{ansatzfp1p1}
f(x,y)= \sum_{k,l\ge0 \atop k+l\ge 1} \frac{N_{k,l}}{(2k+2l-1)!} e^{kx+ly},
\eeq
then the WDVV equations become 6 non-trivial and nonlinear PDEs, including for example 
\begin{align}
& 0= f_{xxx}f_{yyy} -f_{xyy} f_{xxy} + 2 f_{xy}-4 f_{xyy}-4 f_{xxy}. 
\end{align}
The 6 PDEs form an overdetermined system, which uniquely determine all the numbers $N_{k,l}$ in~\eqref{ansatzfp1p1},
 with the initial data $N_{0,1}=N_{1,0}=1$ only. 
The matrices $\eta$, $\mu$, $R$, $S$ for~$M$ are given by 
\beq
\eta=\begin{pmatrix} 0 & 0 & 0 & 1 \\ 0 & 0 & 1 & 0\\  0& 1 & 0 & 0 \\ 1 & 0 & 0 &0 \end{pmatrix}, 
\;
\mu = \begin{pmatrix} -1 & 0 & 0 & 0 \\
0 & 0 & 0 & 0\\
0 & 0 & 0 & 0\\
0 & 0 & 0 & 1 
\end{pmatrix}, \;
R=\begin{pmatrix} 0 & 0 & 0 & 0 \\
2 & 0 & 0 & 0\\
2 & 0 & 0 & 0\\
0 & 2 & 2 & 0 
\end{pmatrix}, \;
S = 
\begin{pmatrix} 
1 & 2 & 2 & 4 \\
0 & 1 & 0 & 2\\
0 & 0 & 1 & 2\\
0 & 0 & 0 & 1 
\end{pmatrix}.
\eeq

Let $\kappa=(2,2)$. The Frobenius manifold $S_{(2,2)}(M)$ is of charge~$-2$ and has the potential 
\begin{align}
&F^{S_{(2,2)}(M)} ({\bf \hat v}) =  \frac{(\hat v^4)^2}{2} \hat v^1 + \hat v^2 \hat v^3 \hat v^4 
+ \hat v^1 \left(-\hat v^2- \hat v^3+ \hat v^2 \log (\hat v^2)+ \hat v^3 \log (\hat v^3)\right)\nn\\ 
&\qquad  +\frac{1}{\hat v^2 \hat v^3} \frac{(\hat v^1)^3}{3!}
+ \frac{2\hat v^2+2\hat v^3}{ (\hat v^2)^3 (\hat v^3)^3}\frac{(\hat v^1)^5}{5!} 
+ \frac{24 (\hat v^2)^2 + 38 (\hat v^2) (\hat v^3) + 24 (\hat v^3)^2}{(\hat v^2)^5 (\hat v^3)^5}\frac{(\hat v^1)^7}{7!}+ \cdots.\label{p1p1f22}
\end{align}
The Euler vector field for $S_{(2,2)}(M)$ is given by
\beq
E= 3 \hat v^1 \frac{\p }{\p \hat v^1} + 2 \hat v^2 \frac{\p }{\p \hat v^2} +2 \hat v^3 \frac{\p }{\p \hat v^3} + \hat v^4 \frac{\p }{\p \hat v^4}.
\eeq
Write 
\beq
F^{S_{(2,2)}(M)} ({\bf \hat v}) =  \frac{(\hat v^4)^2}{2} \hat v^1 + \hat v^2 \hat v^3 \hat v^4 
+ \hat v^1 \hat v^2 \log \hat v^2 + \hat v^1\hat v^3 \log \hat v^3 + 
(\hat v^1)^{\frac{5}{3}} \, p\biggl(\frac{(\hat v^1)^{\frac23}}{\hat v^2}, \frac{(\hat v^1)^{\frac23}}{\hat v^3}\biggr),
\eeq
where 
\beq\label{ansatzfp1p1s22}
p=p(a,b)=-\frac1a-\frac1b+\sum_{k,l\ge0} \frac{C_{k,l}}{(\frac23 (k+l)+\frac53)!} a^k b^l, \quad C_{k,l}\in \CC.
\eeq
Then the WDVV equations in terms of the function~$p=p(x,y)$ become 6 non-trivial and nonlinear PDEs. For example, one of them is of the form
\beq
27 a b + 10 p + \dots+ 12 a^5 b^5 p_{aaa} p_{bbb} = 0.
\eeq
The 6 PDEs are overdetermined and they uniquely determine all the coefficients $C_{k,l}$, $k,l\ge0$, in~\eqref{ansatzfp1p1s22}.
We expect that $C_{k,l}$ vanishes if and only if 
$k+l \not\equiv 2 \, ({\rm mod} \, 3)$ or $2k<l+1$ or $2l<k+1$, and that if $C_{k,l}$ is not zero it is a positive integer.

Let $\kappa=(2,1)$. The Frobenius manifold $S_{(2,1)}(M)$ is of charge~$0$ and has the potential 
\begin{align}
&F^{S_{(2,1)}(M)} ({\bf \hat v}) = \frac{(\hat v^3)^2}{2} \hat v^2 + \hat v^1 \hat v^3 \hat v^4 
+ \hat v^1\hat v^2  \Bigl(\log (\hat v^1)-1\Bigr) + e^{\hat v^4}\hat v^2 +\frac{ e^{\hat v^4}}{ \hat v^1} \frac{(\hat v^2)^3}{3!}  \nn\\ 
&\quad + \biggl(\frac{e^{\hat v^4}}{ (\hat v^1)^2}+2\frac{e^{2 \hat v^4}}{ (\hat v^1)^3}\biggr) \frac{(\hat v^2)^5}{5!} 
+ \biggl(\frac{e^{\hat v^4}}{ (\hat v^1)^3} + 20 \frac{e^{2 \hat v^4}}{ (\hat v^1)^4}+ 24\frac{e^{3\hat v^4}}{ (\hat v^1)^5}\biggr) \frac{(\hat v^2)^7}{7!}
+ \cdots.\label{p1p1f21}
\end{align}
The Euler vector field for $S_{(2,1)}(M)$ is given by
\beq
E= 2 \hat v^1 \frac{\p }{\p \hat v^1} + \hat v^2 \frac{\p }{\p \hat v^2} + \hat v^3 \frac{\p }{\p \hat v^3} + 2 \frac{\p }{\p \hat v^4}.
\eeq
Write 
\beq
F^{S_{(2,1)}(M)} ({\bf \hat v}) = \frac{(\hat v^3)^2}{2} \hat v^2 + \hat v^1 \hat v^3 \hat v^4 
+ \hat v^1\hat v^2  \Bigl(\log (\hat v^1)-1\Bigr)  + 
(v^1)^{\frac{3}{2}} \, m\biggl(\frac{\hat v^2}{(\hat v^1)^{\frac12}}, \frac{e^{\hat v^4}}{\hat v^1}\biggr),
\eeq
where $m=m(r_1,r_2)$ is a power series in $r_1,r_2$ of the form
\beq\label{ansatzfp1p1s21}
m(r_1, r_2)= \sum_{m_1,m_2\ge1} a_{m_1,m_2} r_1^{m_1} r_2^{m_2}, \quad a_{1,1}=1. 
\eeq
Then the WDVV equations in terms of the function $m=m(r_1,r_2)$ again contain 6 nontrivial and nonlinear PDEs. 
This overdetermined system uniquely determine~$p$ with initial data $a_{1,1}=1$. 
We expect that $a_{m_1,m_2}$ vanishes if $m_1$ is even, 
as well as that for any $k\ge0$ the product $(2k+1)! a_{2k+1,m_2}$ is a positive integer if $1\leq m_2\leq k$ and vanishes if $m_2>k$. 
\end{example}

The next example is on the Legendre-type transformation of the Frobenius manifold corresponding to the CCC group of $A_1^{(1,1)}$-type.
\begin{example} \label{examplea111legendre}
Consider the three-dimensional Frobenius manifold with potential given by
\beq
F = \frac12 (v^1)^2 v^3 + \frac12 v^1 (v^2)^2 - \frac{(v^2)^4}{16} \gamma(v^3),
\eeq
where $\gamma$ is the following analytic function:
\beq
\gamma(v^3)=\frac{1}{6} \biggl(1-24 \sum_{m\ge1} \sigma(m) e^{mv^3}\biggr).
\eeq
Here, $\sigma(m)$ denotes the sum of all the divisors of~$m$. 
This Frobenius manifold was introduced by Dubrovin~\cite{Du96}. It is of charge $D=1$ and has the Euler vector field
\beq
E=v^1 \frac{\p }{\p v^1} + \frac12 v^2 \frac{\p }{\p v^2}.
\eeq
The WDVV equation written in terms of~$\gamma$ is the {\it Chazy equation}~\cite{Du96}:
\beq
\gamma''' = 6 \gamma \gamma'' -9 \gamma'^2.
\eeq
The matrices $\eta$, $\mu$ and $R$ for this example are given by 
\beq
\eta=\begin{pmatrix} 0 & 0 & 1 \\ 0 & 1 & 0\\ 1 & 0 & 0\end{pmatrix}, \quad \mu = \begin{pmatrix} -\frac12 & 0 & 0 \\
0 & 0 & 0\\
0 & 0 & \frac12 \end{pmatrix}, \quad R=0.
\eeq

Let $\kappa=2$. The Frobenius manifold $S_2(M)$ is of charge $\hat D=0$ and has the potential 
\begin{align}
&F^{S_2(M)} ({\bf \hat v}) = \frac{(\hat v^2)^3}{6} + \hat v^1 \hat v^2 \hat v^3 + \frac{1}{24} \hat v^1 (\hat v^3)^3 
+ \frac{1}{4} (\hat v^1)^2 \biggl(2 \log \biggl(\frac{\hat v^1}{(\hat v^3)^3}\biggr)-3\biggr) \nn\\ 
&\qquad -\frac{(\hat v^1)^3}{(\hat v^3)^3} + \frac72 \frac{(\hat v^1)^4}{(\hat v^3)^6} - 23\frac{(\hat v^1)^5}{(\hat v^3)^9} + 201\frac{(\hat v^1)^6}{(\hat v^3)^{12}} 
- \frac{10368}5 \frac{ (\hat v^1)^7}{ (\hat v^3)^{15}} +23871 \frac{ (\hat v^1)^8}{(\hat v^3)^{18}} + \cdots.\label{a111fs2}
\end{align}
The Euler vector field now reads
\beq
E= \frac32 \hat v^1 \frac{\p }{\p \hat v^1} +  \hat v^2 \frac{\p }{\p \hat v^2} + \frac12 \hat v^3 \frac{\p }{\p \hat v^3}.
\eeq
Write 
\beq
F^{S_2(M)} ({\bf \hat v}) = \frac{(\hat v^2)^3}{6}+\hat v^1 \hat v^2 \hat v^3 + 
(\hat v^1)^2 \, p \biggl(\frac{\hat v^1}{(\hat v^3)^3}\biggr).
\eeq
Then the WDVV equation in terms of the function~$p=p(t)$ reads
\begin{align}
12 \, t \, ( \dot{p} + 3 \ddot{p}) \bigl(2 \dot{p} + 3 \ddot{p} + \dddot{p}\bigr) = 1, \qquad \dot{}=t \frac{d}{dt}. 
\label{S2ccca1}
\end{align}
We thank Don Zagier~\cite{Zagier} for the simplification of this equation.
Write 
\beq
p(t)=\frac1{24 \, t} +  \frac{1}{2} \log t- \frac34  + \sum_{k\ge1} Q_k t^k, \quad Q_1=-1. 
\eeq
Then~\eqref{S2ccca1} gives a recursion for the coefficients $Q_k$. 
We expect that $kQ_k$ are all integers, with first few being
$-1, 7, -69, 804$. This integrality is not easy to obtain directly from the recursion, 
but is now partially proved by Don Zagier, who has found out a beautiful change of variable deduced from the definition of the 
Legendre-type transformation. 

Let $\kappa=3$. The Frobenius manifold $S_3(M)$ is of charge $\hat D=-1$ and has the potential 
\begin{align}
F^{S_3(M)}({\bf \hat v}) = & \frac12 (\hat v^2)^2 \hat v^3 + \frac12 \hat v^1 (\hat v^3)^2 
+ \frac{1}{4} (\hat v^1)^2 \biggl(2 \log \biggl(\frac{(\hat v^2)^4}{(\hat v^1)^3}\biggr)+9 -12 \log 2\biggr)   \nn\\
& + \frac3{32} \frac{(\hat v^2)^4}{\hat v^1} + \frac3{4096} \frac{ (\hat v^2)^8}{ (\hat v^1)^4}+\frac1{65536} \frac{(\hat v^2)^{12}}{(\hat v^1)^7} 
- \frac{9}{16777216} \frac{ (\hat v^2)^{16}}{(\hat v^1)^{10}} + \cdots.
\end{align}
Here the coordinates ${\bf \hat v}$ and ${\bf v}$ are related by 
\begin{align}
& \hat v^3 = \hat v_1 = \frac{\p^2 F}{\p v^1 \p v^3} = v^1, \nn\\
& \hat v^2 = \hat v_2 = \frac{\p^2 F}{\p v^2 \p v^3} = - \frac{(v^2)^3}{4} \gamma'(v^3), \nn\\
& \hat v^1 = \hat v_3 = \frac{\p^2 F}{\p v^3 \p v^3} = - \frac{(v^2)^4}{16} \gamma''(v^3). \nn
\end{align}
The Euler vector field now reads
\beq
E= 2 \hat v^1 \frac{\p }{\p \hat v^1} +  \frac32 \hat v^2 \frac{\p }{\p \hat v^2} + \hat v^3 \frac{\p }{\p \hat v^3}.
\eeq
Write 
\begin{align}
F^{S_3(M)}({\bf \hat v}) = & \frac12 (\hat v^2)^2 \hat v^3 + \frac12 \hat v^1 (\hat v^3)^2 
+ (\hat v^1)^2 \, m\biggl(\frac{(\hat v^2)^4}{(\hat v^1)^3}\biggr),\label{a111fs3}
\end{align}
where $m(r)=\frac12 \log r + \frac3{32} r + O(r^2)$ (we omit a quadratic term in $F^{S_3(M)}({\bf \hat v})$). Then 
\begin{align}\label{simplifieda111fs3}
32 (\theta(m)) \bigl(\theta^2(1+\theta)(m)\bigr) - 16 (\theta^2(m)) \bigl(\theta^2(13-12 \theta) (m)\bigr)= r \, 3 \theta (3\theta-1)(3\theta-2) (m) ,
\end{align}
where $\theta:=r\frac{d}{dr}$.
Again this gives a recursion for the coefficients $W_k$ in 
$$m(r)=\frac12 \log r + \sum_{k\ge1} W_k r^k$$ with $W_1=3/32$, and 
we expect that $2^{6k} k W_k/6$ are integers.
\end{example}

It is pointed out by Don Zagier \cite{Zagier} that the ODE~\eqref{ODES3p2} of Example~\ref{examplep2quantumcoho} and the ODEs \eqref{S2ccca1}, \eqref{simplifieda111fs3} of Example~\ref{examplea111legendre} all can be easily reduced to second-order ones. It will be interesting to study the relations between 
these ODEs and the Painlev\'e VI found in~\cite{DiFI, Du96} (cf.~\cite{Gu01}).

\end{appendix}

\medskip

\noindent {\bf Acknowledgements}.  The author wishes to thank Boris Dubrovin for his guidance. He also thanks 
the anonymous referees for constructive comments that help to improve a lot the presentation of the paper. 
The author is also grateful to Don Zagier and Youjin Zhang for their advice.
The work was supported by NSFC No.~12371254, NSFC No.~12061131014 and CAS No. YSBR-032.

\smallskip

\noindent {\bf Data availability}. 
Data sharing is not applicable to this article as no datasets were generated or analyzed during the current study.


\begin{thebibliography}{999}
\bibitem{Almeida21}
Almeida, G. F.: 
Differential geometry of orbit space of extended affine Jacobi group $A_1$. SIGMA~{\bf 17} (2021), Paper No.~022.

\bibitem{Almeida22}
Almeida, G. F.: 
The differential geometry of the orbit space of extended affine Jacobi group  $A_n$. 
J. Geom. Phys.~{\bf 171} (2022), Paper No. 104409.

\bibitem{AK}
Aoyama, S., Kodama, Y.:
Topological Landau-Ginzburg theory with a rational potential and the dispersionless KP hierarchy.
Comm. Math. Phys.~{\bf 182} (1996), 185--219.

\bibitem{Arnold}
Arnold, V. I.: 
Critical points of smooth functions, and their normal forms, Uspekhi Mat. Nauk~{\bf 30} (1975), 3--65 (in Russian); 
English translation, Russian Math. Surveys~{\bf 30} (1975), 1--75.

\bibitem{AGV}
Arnold, V. I., Gusein-Zade, S. M., and Varchenko, A. N.: 
{\it Singularities of differentiable maps. Vol. II. Monodromy and asymptotics of integrals}. 
Monographs in Mathematics, vol.~{\bf 83}. Birkh\"auser Boston, Inc., Boston, MA, 1988.

\bibitem{BJL0}
Balser, W., Jurkat, W. B., Lutz, D. A.: Birkhoff invariants and Stokes' multipliers for meromorphic linear differential equations.
J. Math. Anal. Appl.~{\bf 71} (1979), 48--94.

\bibitem{BJL}
Balser, W., Jurkat, W. B., Lutz, D. A.: On the reduction of connection problems for differential equations with an irregular singular point to ones with only regular singularities. I. SIAM J. Math. Anal.~{\bf 12} (1981), 691--721.

\bibitem{BF}
Behrend, K., Fantechi, B.: The intrinsic normal cone. Invent. Math.~{\bf 128} (1997), 45--88.

\bibitem{Be1}
Bertola, M.: 
Frobenius manifold structure on orbit space of Jacobi groups: Parts I. Differ. Geom. Appl.~{\bf 13} (2000), 19--41.

\bibitem{Be2}
Bertola, M.: 
Frobenius manifold structure on orbit space of Jacobi groups: Parts II. Differ. Geom. Appl.~{\bf 13} (2000), 213--233.

\bibitem{BIZ}
Bessis, D., Itzykson, C., Zuber, J.-B.: Quantum field theory techniques in graphical enumeration. Adv. Appl. Math.~{\bf 1} (1980), 109--157.

\bibitem{BPS12-1}
Buryak, A., Posthuma, H., Shadrin, S.: A polynomial bracket for the Dubrovin-Zhang hierarchies. 
J. Differential Geom.~{\bf 92} (2012), 153--185.

\bibitem{BPS12-2}
Buryak, A., Posthuma, H., Shadrin, S.: On deformations of quasi-Miura transformations and the Dubrovin-Zhang bracket. 
J. Geom. Phys.~{\bf 62} (2012), 1639--1651.

\bibitem{Carlet}
Carlet, G.: The extended bigraded Toda hierarchy. Journal of Physics A: Mathematical and General~{\bf 39} (2006), 9411--9435.

\bibitem{CDZ}
Carlet, G., Dubrovin, B., Zhang, Y.: The extended Toda hierarchy. Mosc. Math. J.~{\bf 4} (2004), 313--332.

\bibitem{CvdLPS}
Carlet, G., van de Leur, J., Posthuma, H., Shadrin, S.: Higher genera Catalan numbers and Hirota equations for extended nonlinear Schr\"odinger hierarchy. Lett. Math. Phys.~{\bf 111} (2021), Paper No.~63.

\bibitem{CvdLPS2}
Carlet, G., van de Leur, J., Posthuma, H., Shadrin, S.: 
Enumeration of hypermaps and Hirota equations for extended rationally constrained KP. Commun. Number Theory Phys.~{\bf 17} (2023), 643--708.

\bibitem{CR02}
Chen, W., Ruan, Y.: Orbifold Gromov--Witten theory. Orbifolds in mathematics and physics. Contemp. Math.~{\bf 310} (2002), 25--85.

\bibitem{CR04}
Chen, W., Ruan, Y.: A new cohomology theory of orbifold. Comm. Math. Phys.~{\bf 248} (2004), 1--31.

\bibitem{Cheng}
Cheng, Y.: 
Constraints of the Kadomtsev--Petviashvili hierarchy. Journal of Mathematical Physics~{\bf 33} (1992), 3774--3782.

\bibitem{CI}
Coates, T., Iritani, H.: 
On the convergence of Gromov-Witten potentials and Givental's formula. Michigan Math. J.~{\bf 64} (2015), 587--631.

\bibitem{Cotti}
Cotti, G.: Degenerate Riemann-Hilbert-Birkhoff problems, semisimplicity, and convergence of WDVV-potentials. 
Lett. Math. Phys.~{\bf 111} (2021), Paper No.~99.

\bibitem{CDG18}
Cotti, G., Dubrovin, B., Guzzetti, D.: 
Helix structures in quantum cohomology of Fano varieties. arXiv:1811.09235.

\bibitem{CDG19}
Cotti, G., Dubrovin, B., Guzzetti, D.: Isomonodromy deformations at an irregular singularity with coalescing eigenvalues.
Duke Math. J.~{\bf 168} (2019), 967--1108.

\bibitem{CDG20} 
Cotti, G., Dubrovin, B., Guzzetti, D.: 
Local moduli of semisimple Frobenius coalescent structures. 
Symmetry Integrability Geom. Methods Appl.~{\bf 16} (2020), Paper No.~040.

\bibitem{DS}
David, L, Strachan, Ian A.~B.: Symmetries of $F$-manifolds with eventual identities and special families of connections. 
Ann. Sc. Norm. Super. Pisa Cl. Sci. (5)~{\bf 13} (2014), 641--674. 

\bibitem{DiFI}
Di Francesco, P., Itzykson, C.: 
Quantum intersection rings. In: Dijkgraaf, R., Faber, C. and van de Geer, G. (eds) ``The Moduli Space of Curves",  pp.~81--148, Birkh\"auser, Basel, 
1995.

\bibitem{DVV91-1}
Dijkgraaf, R., Verlinde, H., Verlinde, E.: 
Notes on topological string theory and 2D quantum gravity. 
In: ``String theory and quantum gravity" (Trieste, 1990), pp.~91--156, World Sci. Publ., River Edge, NJ, 1991.

\bibitem{DVV91-2}
Dijkgraaf, R., Verlinde, H., Verlinde, E.: Topological strings in $d<1$. Nuclear Phys. B~{\bf 352} (1991), 59--86.

\bibitem{DW90}
Dijkgraaf, R., Witten, E.: Mean field theory, topological field theory, and multi-matrix models. Nucl. Phys. B~{\bf 342} (1990), 486--522.

\bibitem{Du92}
Dubrovin, B.: Integrable systems in topological field theory. Nuclear Phys. B~{\bf 379} (1992), 627--689.

\bibitem{Du96}
Dubrovin, B.: Geometry of 2D topological field theories. 
In Francaviglia, M., Greco S. (eds) ``Integrable Systems and Quantum Groups" (Montecatini Terme, Italy 1993), pp.~120--348. Springer Lecture Notes in Math.~{\bf 1620}, 1996. 

\bibitem{Du98}
Dubrovin, B.: Geometry and analytic theory of Frobenius manifolds. 
Proceedings of the International Congress of Mathematicians, Vol.~{\bf II} (Berlin, 1998). Doc. Math. 1998, Extra Vol.~{\bf II}, 315--326.

\bibitem{Du98-2}
Dubrovin, B.: Flat pencils of metrics and Frobenius manifolds. In 
``Integrable systems and algebraic geometry" (Kobe/Kyoto, 1997), pp.~47--72, 
World Sci. Publ., River Edge, NJ, 1998.

\bibitem{Du99} 
Dubrovin, B.: Painlev\'e transcendents in two-dimensional topological field theory. 
In ``The Painlev\'e Property: 100 years later", pp.~287--412. CRM Ser. Math. Phys., Springer, New York, 1999.

\bibitem{Du09}
Dubrovin, B.: Hamiltonian perturbations of hyperbolic PDEs: from classification results to the
properties of solutions. In: Sidoravi\v{c}ius, V. (ed.) ``New Trends in Mathematical Physics", pp. 231--276. Springer, Dordrecht, 2009.

\bibitem{Du14}
Dubrovin, B.: Gromov-Witten invariants and integrable hierarchies of topological type. In: ``Topology, Geometry, Integrable Systems, and Mathematical Physics", Amer. Math. Soc. Transl. Ser. (2), vol.~{\bf 234}, AMS, Providence, R.I., pp. 141--171, 2014.

\bibitem{DLYZ}
Dubrovin, B., Liu, S.-Q., Yang, D., Zhang, Y.: 
Hodge integrals and tau-symmetric integrable hierarchies of Hamiltonian evolutionary PDEs. 
Adv. Math.~{\bf 293} (2016), 382--435.

\bibitem{DLZ06}
Dubrovin, B., Liu, S.-Q., Zhang, Y.: 
On Hamiltonian perturbations of hyperbolic systems of conservation laws. I. Quasi-triviality of bi-Hamiltonian perturbations. 
Comm. Pure Appl. Math.~{\bf 59} (2006), 559--615.

\bibitem{DLZ}
Dubrovin, B., Liu, S.-Q., Zhang, Y.: Bihamiltonian cohomologies and integrable hierarchies II: The tau structures.
Comm. Math. Phys.~{\bf 361} (2018), 467--524.

\bibitem{DN83}
Dubrovin, B. A., Novikov, S. P.: Hamiltonian formalism of one-dimensional systems of the 
hydrodynamic type and the Bogolyubov-Whitham averaging method. Dokl. Akad. Nauk SSSR~{\bf 270} (1983), 781--785. 

\bibitem{DSZZ}
Dubrovin, B., Strachan, I. A. B., Zhang, Y., Zuo, D.:
Extended affine Weyl groups of BCD-type: their Frobenius manifolds and Landau-Ginzburg superpotentials.
Adv. Math.~{\bf 351} (2019), 897--946.

\bibitem{DuY17}
Dubrovin, B., Yang, D.: Generating series for GUE correlators. Lett. Math. Phys.~{\bf 107} (2017), 
1971--2012.

\bibitem{DZ98}
Dubrovin, B., Zhang, Y.: Extended affine Weyl groups and Frobenius manifolds.
Compositio Math.~{\bf 111} (1998), 167--219.

\bibitem{DZ98oneloop}
Dubrovin, B., Zhang, Y.: 
Bihamiltonian hierarchies in 2D topological field theory at one-loop approximation. Comm. Math. Phys.~{\bf 198} (1998), 311--361.

\bibitem{DZ99}
Dubrovin, B., Zhang, Y.: Frobenius manifolds and Virasoro constraints. 
Selecta Math. (N.S.)~{\bf 5} (1999), 423--466.

\bibitem{DZ-norm}
Dubrovin, B., Zhang, Y.: 
Normal forms of hierarchies of integrable PDEs, Frobenius manifolds and Gromov-Witten invariants, 
a new 2005 version of arXiv:math/0108160v1, 295~pp.

\bibitem{DZ04}
Dubrovin, B., Zhang, Y.: Virasoro symmetries of the extended Toda hierarchy.
Comm. Math. Phys.~{\bf 250} (2004), 161--193.

\bibitem{EHX}
Eguchi, T., Hori, K., Xiong, C.-S.: Quantum Cohomology and Virasoro Algebra. Phys. Lett. B~{\bf 402} (1997), 71--80. 

\bibitem{FJR13}
Fan, H., Jarvis, T., Ruan, Y.: 
The Witten equation, mirror symmetry, and quantum singularity theory. Ann. of Math. (2)~{\bf 178} (2013), 1--106.

\bibitem{FFJMR}
Fan, H., Francis, A., Jarvis, T., Merrell, E., Ruan, Y.:
Witten's $D_4$ integrable hierarchies conjecture. 
Chinese Ann. Math. Ser. B~{\bf 37} (2016), 175--192. 

\bibitem{FY}
Fu, A., Yang, D.: The matrix-resolvent method to tau-functions for the nonlinear Schr\"odinger hierarchy. 
J. Geom. Phys.~{\bf 179} (2022), Paper No.~104592, 17 pp. 

\bibitem{FYZ}
Fu, A., Yang, D., Zuo, D.: The constrained KP hierarchy and the bigraded Toda hierarchy of $(M,1)$-type. 
 Letters in Mathematical Physics~{\bf 113} (2023), Paper No.~124.

\bibitem{GGI}
Galkin, S., Golyshev, V., Iritani, H.: 
Gamma classes and quantum cohomology of Fano manifolds: Gamma conjectures. Duke Math. J.~{\bf 165} (2016), 2005--2077.

\bibitem{GMS}
Galkin, S., Mellit, A., Smirnov, M.: Dubrovin's conjecture for $IG(2,6)$. IMRN~{\bf 2015} (2015), 8847--8859.

\bibitem{Getzler}
Getzler, E.: 
Intersection theory on  $\overline{\mathcal{M}}_{1,4}$ and elliptic Gromov-Witten invariants.
J. Amer. Math. Soc.~{\bf 10} (1997), 973--998.

\bibitem{Givental01-1}
Givental, A.~B.: Semisimple Frobenius structures at higher genus. IMRN~{\bf 2001} (2001), 1265--1286.

\bibitem{Givental}
Givental, A.~B.: Gromov-Witten invariants and quantization of quadratic Hamiltonians. 
Mosc. Math. J.~{\bf 1} (2001), 551--568.

\bibitem{Gu99}
Guzzetti, D: 
Stokes matrices and monodromy for the quantum cohomology of projective spaces. Comm. Math. Phys.~{\bf 207} (1999), 341--383. 

\bibitem{Gu01}
Guzzetti, D: 
Inverse problem and monodromy data for three-dimensional Frobenius manifolds.
Math. Phys. Anal. Geom.~{\bf 4} (2001), 245--291.

\bibitem{Gu16}
Guzzetti, D.: 
On stokes matrices in terms of connection coefficients.
Funkcial. Ekvac.~{\bf 59} (2016), 383--433.

\bibitem{Gu21}
Guzzetti, D.: 
Isomonodromic Laplace transform with coalescing eigenvalues and confluence of Fuchsian singularities.
Lett. Math. Phys.~{\bf 111} (2021), Paper No. 80, 70 pp.

\bibitem{HZ}
Harer, J., Zagier, D.: The Euler characteristic of the moduli space of curves. Invent. Math.~{\bf 85} (1986), 457--485.

\bibitem{Hertling}
Hertling, C.: {\it Frobenius Manifolds and Moduli Spaces for Singularities}. 
Cambridge Tracts in Mathematics, vol.~{\bf 151}. Cambridge University Press, Cambridge, 2002.

\bibitem{HM}
Hertling, C., Manin, Yu.:
Weak Frobenius manifolds. IMRN~{\bf 1999} (1999), 277--286. 

\bibitem{Iritani}
Iritani, H.: 
Convergence of quantum cohomology by quantum Lefschetz. J. Reine Angew. Math.~{\bf 610} (2007), 29--69.

\bibitem{IT}
Iwaki, K., Takahashi, A.: Stokes matrices for the quantum cohomologies of a class of orbifold projective lines.
J. Math. Phys.~{\bf 54} (2013), paper No.~101701.

\bibitem{JMU}
Jimbo, M., Miwa, T., Ueno, K.: 
Monodromy preserving deformations of linear ordinary differential equations with rational coefficients I. Physica D~{\bf 2} (1981), 306--352.

\bibitem{KM94}
Kontsevich, M., Manin, Yu.: 
Gromov-Witten classes, quantum cohomology, and enumerative geometry. Comm. Math. Phys.~{\bf 164} (1994), 525--562.

\bibitem{KMK}
Kontsevich, M., Manin, Yu.: 
Quantum cohomology of a product. With an appendix by R. Kaufmann. Invent. Math.~{\bf 124} (1996), 313--339.

\bibitem{LT}
Li, J., Tian, G.: Virtual moduli cycles and Gromov-Witten invariants of algebraic varieties. J. Amer. Math. Soc.~{\bf 11} (1998), 119--174.

\bibitem{LWZ21}
Liu, S.-Q., Wang, Z., Zhang, Y.: Linearization of Virasoro symmetries associated with semisimple Frobenius manifolds. arXiv:2109.01846.

\bibitem{LWZ}
Liu, S.-Q., Wang, Z., Zhang, Y.: Variational Bihamiltonian Cohomologies and Integrable Hierarchies III: Linear Reciprocal Transformations.
Comm. Math. Phys.~{\bf 403} (2023), 1109--1152.

\bibitem{LZ05}
Liu, S.-Q., Zhang, Y.: Deformations of semisimple bihamiltonian structures of hydrodynamic type. J. Geom. Phys.~{\bf 54} (2005), 427--453. 

\bibitem{LZZ} 
Liu, S.-Q., Zhang, Y., Zhou, X.: Central invariants of the constrained KP hierarchies.
J. Geom. Phys.~{\bf 97} (2015), 177--189.

\bibitem{LiuT}
Liu, X., Tian, G.:
Virasoro constraints for quantum cohomology. 
J. Differential Geom.~{\bf 50} (1998), 537--590.

\bibitem{Loo}
Looijenga, E.:
A period mapping for certain semi-universal deformations.
Compositio Math.~{\bf 30} (1975), 299--316.

\bibitem{Lorenzoni}
Lorenzoni, P.: Deformations of bihamiltonian structures of hydrodynamic type. J. Geom. Phys.~{\bf 44} (2002), 331--375.

\bibitem{Manin}
Manin, Yu. I.: {\it Frobenius manifolds, Quantum Cohomology, and Moduli Spaces}. Amer. Math. Soc,
Providence, RI, 1999.

\bibitem{MS}
McDuff, D., Salamon, D.: {\it $J$-holomorphic curves and quantum cohomology}. 
University Lecture Series~{\bf 6}, American Mathematical Society, Providence, RI, 1994.

\bibitem{Mehta}
Mehta, M.L.: {\it Random Matrices}. Second edition. Academic Press, Cambridge, 1991. 

\bibitem{MT}
Milanov, T. E., Tseng, H.-H.:
The spaces of Laurent polynomials, Gromov-Witten theory of $\mathbb{P}^1$-orbifolds, and integrable hierarchies.
J. Reine Angew. Math.~{\bf 622} (2008), 189--235.

\bibitem{Pavlov}
Pavlov, M.: Conservation of the ``forms" of the Hamiltonian structures upon linear substitution for
independent variables. Math. Notes~{\bf 57} (1995), 489--495.

\bibitem{Rossi}
Rossi, P.: Gromov--Witten theory of orbicurves, 
the space of tri-polynomials and symplectic field theory of Seifert fibrations. Math. Ann.~{\bf 348} (2010), 265--287.

\bibitem{RT}
Ruan, Y., Tian, G.: A mathematical theory of quantum cohomology. J. Differential Geom.~{\bf 42} (1995), 259--367.

\bibitem{Sabbah}
Sabbah, C.: {\it D\'eformations isomonodromiques et vari\'et\'es de Frobenius}. EDP Sciences, LesUlis, Paris, 2002.

\bibitem{Saito81}
Saito, K.: Primitive forms for a universal unfolding of a function with an isolated critical point. 
J. Fac. Sci. Univ. Tokyo Sect. IA Math.~{\bf 28} (1981), 775--792.

\bibitem{Saito83}
Saito, K.: Period mapping associated to a primitive form. Publ. Res. Inst. Math. Sci.~{\bf 19} (1983), 1231--1264.

\bibitem{Saito93}
Saito, K.: On a linear structure of the quotient variety by a finite reflexion group. Publ. Res. Inst. Math. Sci.~{\bf 29} (1993), 535--579.

\bibitem{Sai89}
Saito, M.: On the structure of Brieskorn lattice. Ann. Inst. Fourier~{\bf 39} (1989), 27--72.

\bibitem{SS}
Strachan, Ian A.~B., Stedman, R.: 
Generalized Legendre transformations and symmetries of the WDVV equations. 
Journal of Physics A-Mathematical and Theoretical~{\bf 50} (2017), paper No. 095202.

\bibitem{Teleman}
Teleman, C.: The structure of 2D semi-simple field theories. Invent. Math.~{\bf 188} (2012), 525--588.

\bibitem{Tsarev}
Tsarev, S.: The geometry of Hamiltonian systems of hydrodynamic type. 
The generalized hodograph method, Math. USSR Izv.~{\bf 37} (1991), 397--419.

\bibitem{UEDA}
Ueda, K.: Stokes matrices for the quantum cohomologies of Grassmannians. IMRN~{\bf 2005} (2005), 2075--2086.

\bibitem{Vekslerchik}
Vekslerchik, V. E.: Universality of the Ablowitz--Ladik hierarchy. arXiv:9807005.

\bibitem{W90}
Witten, E.: On the structure of the topological phase of two-dimensional gravity. Nucl. Phys. B~{\bf 340} (1990), 281--332.

\bibitem{W93}
Witten, E.: Algebraic geometry associated with matrix models of two-dimensional gravity. 
In ``Topological methods in modern mathematics" (Stony Brook, NY, 1991), 235--269.
Publish or Perish, Inc., Houston, TX, 1993.

\bibitem{XZ}
Xue, T., Zhang, Y.: Bihamiltonian systems of hydrodynamic type and reciprocal transformations. Lett.
Math. Phys.~{\bf 75} (2006), 79--92.

\bibitem{Y1}
Yang, D.: GUE via Frobenius manifolds. I. From matrix gravity to topological gravity and back. 
Acta Mathematical Sinica (English Series)~{\bf 40} (2024), 383--405.

\bibitem{Y2}
Yang, D.: GUE via Frobenius manifolds. II. Loop equations. preprint.

\bibitem{YZ}
Yang, D., Zagier, D.: Mapping partition functions. arXiv:2308.03568.

\bibitem{YZhang}
Yang, D.,  Zhang, Q.: On the Hodge-BGW correspondence. arXiv:2112.12736.

\bibitem{YZhou}
Yang, D., Zhou, J.: Grothendieck's dessins d'enfants in a web of dualities. III. Journal of Physics A: 
Mathematical and Theoretical~{\bf 56} (2023), paper No. 055201.

\bibitem{Zagier}
Zagier, private communication.

\bibitem{Zaslow}
Zaslow, E.: Solitons and helices: The search for a math-physics bridge. Comm. Math. Phys.~{\bf 175} (1996), 337--375.

\bibitem{Zhou}
Zhou, J.: Hermitian one-matrix model and KP hierarchy. arXiv:1809.07951.

\bibitem{Zuber}
Zuber, J.-B.: On Dubrovin topological field theories. 
Modern Phys. Lett. A~{\bf 9} (1994), 749--760.

\end{thebibliography}
\end{document}